\documentclass[reqno,10pt,centertags]{amsart}
\usepackage{amsmath,amsthm,amscd,amssymb,latexsym,upref,mathrsfs}
\usepackage{curves}
\usepackage{hyperref}

\newcommand{\bbN}{{\mathbb{N}}}
\newcommand{\bbR}{{\mathbb{R}}}

\newcommand{\bbC}{{\mathbb{C}}}

\newcommand{\cA}{{\mathcal A}}
\newcommand{\cB}{{\mathcal B}}

\newcommand{\cD}{{\mathcal D}}
\newcommand{\cE}{{\mathcal E}}

\newcommand{\cG}{{\mathcal G}}
\newcommand{\cH}{{\mathcal H}}

\newcommand{\cK}{{\mathcal K}}

\newcommand{\cN}{{\mathcal N}}
\newcommand{\cO}{{\mathcal O}}

\newcommand{\cS}{{\mathcal S}}

\newcommand{\cV}{{\mathcal V}}
\newcommand{\cW}{{\mathcal W}}
\newcommand{\cX}{{\mathcal X}}

\newcommand{\no}{\notag}
\newcommand{\lb}{\label}

\newcommand{\ol}{\overline}

\newcommand{\wti}{\widetilde}

\newcommand{\Oh}{O}
\newcommand{\oh}{o}
\newcommand{\f}{\frac}
\newcommand{\loc}{\text{\rm{loc}}}
\newcommand{\bi}{\bibitem}
\newcommand{\hatt}{\widehat}

\renewcommand{\Re}{\mathop\mathrm{Re}}
\renewcommand{\Im}{\mathop\mathrm{Im}}
\renewcommand{\ge}{\geqslant}
\renewcommand{\le}{\leqslant}

\DeclareMathOperator{\dom}{dom}
\DeclareMathOperator{\ran}{ran}

\DeclareMathOperator*{\slim}{s-lim}

\newcommand{\dott}{\,\cdot\,}
\newcommand{\abs}[1]{\lvert#1\rvert}
\newcommand{\norm}[1]{\left\Vert#1\right\Vert}

\newcommand{\Om}{\Omega}
\newcommand{\dOm}{{\partial\Omega}}
\newcommand{\si}{\sigma}

\newcommand{\ga}{\gamma}

\newcommand{\LdOm}{L^2(\dOm;d^{n-1} \omega)}

\allowdisplaybreaks \numberwithin{equation}{section}
\newtheorem{theorem}{Theorem}[section]
\newtheorem{proposition}[theorem]{Proposition}
\newtheorem{lemma}[theorem]{Lemma}
\newtheorem{corollary}[theorem]{Corollary}
\newtheorem{definition}[theorem]{Definition}
\newtheorem{hypothesis}[theorem]{Hypothesis}

\theoremstyle{remark}
\newtheorem{remark}[theorem]{Remark}

\begin{document}

\title[Spectral Theory for Perturbed Krein Laplacians]{Spectral
Theory for Perturbed Krein Laplacians \\ in Nonsmooth Domains}

\author[M.\ S.\ Ashbaugh]{Mark S.\ Ashbaugh}
\address{Department of Mathematics,
University of Missouri, Columbia, MO 65211, USA}
\email{ashbaughm@missouri.edu}
%\urladdr{\href{http://www.math.missouri.edu/personnel/faculty/ashbaughm.html}
%{http://www.math.missouri.edu/personnel/faculty/ashbaughm.html}}
\urladdr{http://www.math.missouri.edu/personnel/faculty/ashbaughm.html}

\author[F.\ Gesztesy]{Fritz Gesztesy}
\address{Department of Mathematics,
University of Missouri, Columbia, MO 65211, USA}
\email{gesztesyf@missouri.edu}
%\urladdr{\href{http://www.math.missouri.edu/personnel/faculty/gesztesyf.html}
%{http://www.math.missouri.edu/personnel/faculty/gesztesyf.html}}
\urladdr{http://www.math.missouri.edu/personnel/faculty/gesztesyf.html}

\author[M.\ Mitrea]{Marius Mitrea}
\address{Department of Mathematics, University of
Missouri, Columbia, MO 65211, USA}
\email{mitream@missouri.edu}
%\urladdr{\href{http://www.math.missouri.edu/personnel/faculty/mitream.html}
%{http://www.math.missouri.edu/personnel/faculty/mitream.html}}
\urladdr{http://www.math.missouri.edu/personnel/faculty/mitream.html}

\author[G.\ Teschl]{Gerald Teschl}
\address{Faculty of Mathematics\\ University of Vienna\\
Nordbergstrasse 15\\ 1090 Wien\\ Austria\\ and International Erwin Schr\"odinger
Institute for Mathematical Physics\\ Boltzmanngasse 9\\ 1090 Wien\\ Austria} 
%\email{\href{mailto:Gerald.Teschl@univie.ac.at}{Gerald.Teschl@univie.ac.at}}
\email{Gerald.Teschl@univie.ac.at}
%\urladdr{\href{http://www.mat.univie.ac.at/~gerald/}
%{http://www.mat.univie.ac.at/\string~gerald/}}
\urladdr{http://www.mat.univie.ac.at/\string~gerald/}

\thanks{Based upon work partially supported by the US National Science
Foundation under Grant Nos.\ DMS-0400639 and
FRG-0456306 and the Austrian Science Fund (FWF) under Grant No.\ Y330.}
%\dedicatory{}
\thanks{Adv. Math. {\bf 223}, 1372--1467 (2010).}
\date{\today}
%\date{July 20, 2008}
\subjclass[2000]{Primary 35J25, 35J40, 35P15; Secondary 35P05, 46E35,
47A10, 47F05.}
\keywords{Lipschitz domains, Krein Laplacian, eigenvalues, spectral
analysis, Weyl asymptotics, buckling problem}

%%%%%%%%%%%%%%%%%%%%%%%%%%%%%%%%%%%%%%%%
%%%%%%%%%%%%%%%%%%%%%%%%%%%%%%%%%%%%%%%%
\begin{abstract}
We study spectral properties for $H_{K,\Omega}$, the Krein--von Neumann 
extension of the perturbed Laplacian $-\Delta+V$ defined on 
$C^\infty_0(\Omega)$, where $V$ is measurable, bounded and nonnegative, in 
a bounded open set $\Omega\subset\mathbb{R}^n$ belonging to a class of 
nonsmooth domains which contains all convex domains, along with all domains 
of class $C^{1,r}$, $r>1/2$. In particular, in the aforementioned context we 
establish the Weyl asymptotic formula
\[
\#\{j\in\mathbb{N}\,|\,\lambda_{K,\Omega,j}\leq\lambda\}
= (2\pi)^{-n} v_n |\Omega|\,\lambda^{n/2}+O\big(\lambda^{(n-(1/2))/2}\big)
\, \mbox{ as }\, \lambda\to\infty,
\]
where $v_n=\pi^{n/2}/ \Gamma((n/2)+1)$ denotes the volume of the unit ball
in $\mathbb{R}^n$, and $\lambda_{K,\Omega,j}$, $j\in\mathbb{N}$, are the
non-zero eigenvalues of $H_{K,\Omega}$, listed in increasing order
according to their multiplicities. We prove this formula by showing
that the perturbed Krein Laplacian (i.e., the Krein--von Neumann extension of
$-\Delta+V$ defined on $C^\infty_0(\Omega)$) is spectrally equivalent to the 
buckling of a clamped plate problem, and using an abstract result of Kozlov
from the mid 1980's. Our work builds on that of Grubb in the early 1980's,
who has considered similar issues for elliptic operators in smooth domains,
and shows that the question posed by Alonso and Simon in 1980
pertaining to the validity of the above Weyl asymptotic formula
continues to have an affirmative answer in this nonsmooth setting. 

We also study certain exterior-type domains $\Omega = \mathbb{R}^n\backslash K$, 
$n\geq 3$, with $K\subset \mathbb{R}^n$ compact and vanishing Bessel capacity  
$B_{2,2} (K) = 0$, to prove equality of Friedrichs and Krein Laplacians in 
$L^2(\Omega; d^n x)$, that is, $-\Delta|_{C_0^\infty(\Omega)}$ has a unique 
nonnegative self-adjoint extension in $L^2(\Omega; d^n x)$. 
\end{abstract}
%%%%%%%%%%%%%%%%%%%%%%%%%%%%%%%%%%%%%%%%
%%%%%%%%%%%%%%%%%%%%%%%%%%%%%%%%%%%%%%%%

\maketitle

{\scriptsize \tableofcontents}

%%%%%%%%%%%%%%%%%%%%%%%%%%%%%%%%%%%%%%%%
%%%%%%%%%%%%%%%%%%%%%%%%%%%%%%%%%%%%%%%%
\section{Introduction}
\label{s1}
%%%%%%%%%%%%%%%%%%%%%%%%%%%%%%%%%%%%%%%%
%%%%%%%%%%%%%%%%%%%%%%%%%%%%%%%%%%%%%%%%

Let $-\Delta_{D,\Om}$ be the Dirichlet Laplacian associated with an open
set $\Omega\subset\bbR^n$, and denote by $N_{D,\Om}(\lambda)$
the corresponding spectral distribution function (i.e., the number of
eigenvalues of $-\Delta_{D,\Om}$ not exceeding $\lambda$).
The study of the asymptotic behavior of $N_{D,\Om}(\lambda)$ as $\lambda\to\infty$
has been initiated by Weyl in 1911--1913 (cf.\ \cite{We12a}, \cite{We12}, and the 
references in \cite{We50}), in response to a question
posed in 1908 by the physicist Lorentz, pertaining to the equipartition
of energy in statistical mechanics. When $n=2$ and $\Omega$ is a bounded
domain with a piecewise smooth boundary, Weyl has shown that
\begin{equation} \label{WA-1}
N_{D,\Om}(\lambda)=\frac{{\rm area}\,(\Om)}{4\pi}\lambda+o(\lambda)
\, \mbox{ as }\, \lambda\to\infty,
\end{equation}
along with the three-dimensional analogue of \eqref{WA-1}. In particular,
this allowed him to complete a partial proof of Rayleigh, going back to 1903.
This ground-breaking work has stimulated a great deal of activity in the
intervening years, in which a large number of authors have provided sharper
estimates for the remainder, and considered more general elliptic operators
equipped with a variety of boundary conditions. For a general elliptic
differential operator $\cA$ of order $2m$ ($m\in\bbN$), with smooth coefficients,
acting on a smooth subdomain $\Om$ of an $n$-dimensional smooth manifold,
spectral asymptotics of the form
\begin{equation} \label{WA-2}
N_{D,\Om}(\cA;\lambda)=(2\pi)^{-n}\biggl(\int_{\Omega}dx\int_{a^0(x,\xi)<1}d\xi\biggr)
\lambda^{n/(2m)}+O\big(\lambda^{(n-1)/(2m)}\big)
\, \mbox{ as }\, \lambda\to\infty,
\end{equation}
where $a^0(x,\xi)$ denotes the principal symbol of $\cA$, have then been
subsequently established in increasing generality
(a nice exposition can be found in \cite{Ag97}).
At the same time, it has been realized that, as the smoothness of the
domain $\Om$ and the coefficients of $\cA$ deteriorate, the degree of detail
with which the remainder can be described decreases accordingly.
Indeed, the smoothness of the boundary of the underlying domain
$\Omega$ affects both the nature of the remainder in \eqref{WA-2}, as well
as the types of differential operators and boundary conditions for which
such an asymptotic formula holds. Understanding this correlation then
became a central theme of research. For example, in the case of the Laplacian
in an arbitrary bounded, open subset $\Om$ of $\bbR^n$,
Birman and Solomyak have shown in \cite{BS70} (see also \cite{BS71}, \cite{BS72}, 
\cite{BS73}, \cite{BS79}) that
the following Weyl asymptotic formula holds
\begin{equation} \label{Wey-1}
N_{D,\Om}(\lambda)=(2\pi)^{-n}v_n|\Omega|\,\lambda^{n/2}+o\big(\lambda^{n/2}\big)
\, \mbox{ as }\, \lambda\to\infty,
\end{equation} 
where $v_n$ denotes the volume of the unit ball in $\bbR^n$,
and $|\Omega|$ stands for the $n$-dimensional Euclidean volume of $\Omega$.
On the other hand, it is known that \eqref{Wey-1} may fail for
the Neumann Laplacian $-\Delta_{N,\Om}$. Furthermore, if $\alpha\in(0,1)$ then
Netrusov and Safarov have proved that
\begin{equation} \label{Wey-2}
\Omega\in{\rm Lip}_{\alpha}\,\text{ implies }\,
N_{D,\Om}(\lambda)=(2\pi)^{-n}v_n|\Omega|\,\lambda^{n/2} 
+ O\big(\lambda^{(n-\alpha)/2}\big) \, \mbox{ as }\, \lambda\to\infty,
\end{equation} 
where ${\rm Lip}_{\alpha}$ is the class of bounded domains whose
boundaries can be locally described by means of graphs of functions
satisfying a H\"older condition of order $\alpha$; this result is sharp.
See \cite{NS05} where this intriguing result (along with others, similar
in spirit) has been obtained. Surprising connections between Weyl's asymptotic
formula and geometric measure theory have been explored in \cite{Cae95},
\cite{HL97}, \cite{LF06} for fractal domains. Collectively, this body of work shows
that the nature of the Weyl asymptotic formula is intimately related not
only to the geometrical properties of the domain (as well as the type of
boundary conditions), but also to the smoothness properties of its boundary (the monograph by Safarov and Vassiliev \cite{SV97} contains a wealth of information on this circle of ideas).

These considerations are by no means limited to the Laplacian; see
\cite{Cae98} for the case of the Stokes operator, and \cite{BF07}, \cite{BS87} 
for the case the Maxwell system in nonsmooth domains.
However, even in the case of the Laplace operator, besides $-\Delta_{D,\Om}$
and $-\Delta_{N,\Om}$ there is a multitude of other concrete extensions of
the Laplacian $-\Delta$ on $C^\infty_0(\Om)$ as a nonnegative,
self-adjoint operator in $L^2(\Om;d^nx)$.
The smallest (in the operator theoretic order sense) such realization has been
introduced, in an abstract setting, by M.\ Krein \cite{Kr47}. Later it was realized that in the case 
where the symmetric operator, whose self-adjoint extensions are sought, has a strictly positive 
lower bound, Krein's construction coincides with one that von Neumann had discussed in 
his seminal paper \cite{Ne29}  in 1929.   

For the purpose of this introduction we now 
briefly recall the construction of the Krein--von Neumann extension
of appropriate $L^2(\Om; d^n x)$-realizations of the  differential operator $\cA$ 
of order $2m$, $m\in\bbN$, 
\begin{align}
& \cA = \sum_{0 \leq |\alpha| \leq 2m} a_{\alpha}(\cdot) D^{\alpha},   \lb{Wey-3} \\
& D^{\alpha} = (-i \partial/\partial x_1)^{\alpha_1} \cdots  
(-i\partial/\partial x_n)^{\alpha_n}, 
\quad \alpha =(\alpha_1,\dots,\alpha_n) \in \bbN_0^n,    \lb{Wey-3A} \\
& a_{\alpha} (\cdot) \in C^\infty(\ol \Om),  \quad 
C^\infty(\ol \Om) = \bigcap_{k\in\bbN_0} C^k(\ol \Om),   \lb{Weyl-3B}
\end{align}
where
$\Omega\subset\bbR^n$ is a bounded $C^\infty$ domain. Introducing the particular 
$L^2(\Om; d^n x)$-realization $A_{c,\Om}$ of $\cA$ defined by 
\begin{equation}
A_{c,\Om} u = \cA u, \quad u \in \dom (A_{c,\Om}):=C^\infty_0(\Om),  \lb{Wey-3a}
\end{equation}
we assume the coefficients $a_\alpha$ in $\cA$ are chosen such that 
$A_{c,\Om}$ is symmetric, 
\begin{equation} \label{Wey-4}
(u, A_{c,\Om} v)_{L^2(\Om;d^nx)}=(A_{c,\Om} u, v)_{L^2(\Om;d^nx)},
\quad u,v\in C^\infty_0(\Om),
\end{equation} 
has a (strictly) positive lower bound, that is, there exists $\kappa_0>0$ such that
\begin{equation} \label{Wey-4bis}
(u, A_{c,\Om} u)_{L^2(\Om;d^nx)}\geq\kappa_0\,\|u\|^2_{L^2(\Om;d^nx)},
\quad u\in C^\infty_0(\Om),
\end{equation} 
and is strongly elliptic, that is, there exists $\kappa_1>0$ such that 
\begin{equation} \label{Wey-5}
a^0(x,\xi):= \Re\bigg(\sum_{|\alpha|=2m}
a_{\alpha}(x) \xi^{\alpha}\bigg) \geq \kappa_1\,|\xi|^{2m},
\quad x\in\ol{\Om}, \; \xi \in\bbR^n.
\end{equation} 
Next, let $A_{min,\Om}$ and $A_{max,\Om}$ be the $L^2(\Om;d^nx)$-realizations
of $\cA$ with domains (cf.\ \cite{Ag97}, \cite{Gr09})
\begin{align}\label{Wey-6} 
\dom (A_{min,\Om})&:=H^{2m}_0(\Om),    \\
\dom (A_{max,\Om})&:=\big\{u\in L^2(\Omega;d^nx)\,\big|\, \cA u\in L^2(\Omega;d^nx)\big\}. 
\end{align}
Throughout this manuscript, 
$H^s(\Om)$ denotes the $L^2$-based Sobolev space of order $s\in\bbR$ in $\Om$,
and $H_0^{s}(\Omega)$ is the subspace of $H^{s}(\bbR^n)$ consisting
of distributions supported in $\ol{\Om}$ (for $s>\frac{1}{2}$,
$\big(s-\frac{1}{2}\big)\notin\bbN$, the space $H_0^{s}(\Omega)$ can be alternatively
described as the closure of $C^\infty_0(\Om)$ in $H^s(\Om)$).
Given that the domain $\Om$ is smooth, elliptic regularity implies
\begin{equation} \label{Kre-DefY}
(A_{min,\Om})^*=A_{max,\Om}\, \mbox{ and }\, \ol{A_{c,\Om}}=A_{min,\Om}.
\end{equation} 
Functional analytic considerations (cf.\ the discussion in Section \ref{s2})
dictate that the Krein--von Neumann (sometimes also called the  ``soft'') extension 
$A_{K,\Om}$ of $A_{c,\Om}$ on
$C^\infty_0(\Om)$ is the $L^2(\Om;d^nx)$-realization of $A_{c,\Om}$ with domain 
(cf.\ \eqref{SK} derived abstractly by Krein)
\begin{equation} \label{Kre-DefX}
\dom(A_{K,\Om})=\dom\big(\ol{A_{c,\Om}}\big)\,
\dot{+}\ker\big((A_{c,\Om})^*\big).
\end{equation} 
Above and elsewhere, $X\dot{+}Y$ denotes the direct sum of two
subspaces, $X$ and $Y$, of a larger space $Z$, with the property that $X\cap Y=\{0\}$.
Thus, granted \eqref{Kre-DefY}, we have
\begin{align}\label{Gr-r2}
\begin{split}
\dom(A_{K,\Om}) &= \dom(A_{min,\Om})\,\dot{+}\ker(A_{max,\Om})   \\
&= H^{2m}_0(\Om)\,\dot{+}\,
\big\{u\in L^2(\Om;d^nx)\,\big|\,Au=0\mbox{ in }\Omega\big\}.
\end{split}
\end{align}
In summary, for domains with smooth boundaries, $A_{K,\Om}$ is the self-adjoint 
realization of $A_{c,\Om}$ with domain given by \eqref{Gr-r2}. 

Denote by $\gamma^{m}_D u:=
\bigl(\gamma_N^ju\bigr)_{0\leq j\leq m-1}$ the Dirichlet trace operator
of order $m\in\bbN$ (where $\nu$ denotes the outward unit normal
to $\Om$ and $\gamma_N u:=\partial_{\nu}u$ stands for the normal derivative,
or Neumann trace), and let $A_{D,\Om}$ be the Dirichlet (sometimes also called the ``hard'')
realization of $A_{c,\Om}$ in $L^2(\Omega;d^nx)$ with domain
\begin{equation} \label{Wey-8}
\dom(A_{D,\Om}):=\big\{u\in H^{2m}(\Omega)\,\big|\,\gamma^{m}_D u=0\big\}.
\end{equation} 
Then $A_{K,\Om}$, $A_{D,\Om}$ are ``extremal'' in the following sense:
Any nonnegative self-adjoint extension $\widetilde{A}$ in $L^2(\Om;d^nx)$ of
$A_{c,\Om}$ (cf.\ \eqref{Wey-3a}), necessarily satisfies 
\begin{equation} \label{Wey-10}
A_{K,\Om} \leq \widetilde{A} \leq A_{D,\Om}
\end{equation} 
in the sense of quadratic forms (cf.\ the discussion surrounding \eqref{AleqB}). 

Returning to the case where $A_{c,\Om}=-\Delta|_{C^\infty_0(\Om)}$, for a bounded
domain $\Om$ with a $C^\infty$-smooth boundary, $\partial\Om$, the corresponding Krein--von Neumann extension admits the following description 
\begin{align}
\begin{split}
& -\Delta_{K,\Om} u:= -\Delta u,   \\ 
& \; u\in \dom(-\Delta_{K,\Om}):=\{v\in\dom(-\Delta_{max,\Om})\,|\,
\gamma_N v +M_{D,N,\Om}(\gamma_D v)=0\},    \label{Wey-11}
\end{split}
\end{align}
where $M_{D,N,\Om}$ is (up to a minus sign) an energy-dependent Dirichlet-to-Neumann map, or Weyl--Titchmarsh operator for the Laplacian. Compared with 
\eqref{Gr-r2}, the description
\eqref{Wey-11} has the advantage of making explicit the boundary condition implicit
in the definition of membership to $\dom(-\Delta_{K,\Om})$.
Nonetheless, as opposed to the classical Dirichlet and Neumann boundary
condition, this turns out to be {\it nonlocal} in nature, as it involves
$M_{D,N,\Om}$ which, when $\Om$ is smooth, is a boundary
pseudodifferential operator of order $1$. Thus, informally speaking, 
\eqref{Wey-11} is the realization of  the Laplacian with the boundary condition
\begin{equation} \label{B.A-1}
\partial_\nu u=\partial_{\nu}H(u)\, \mbox{ on }\, \partial\Omega,
\end{equation} 
where, given a reasonable function $w$ in $\Om$, $H(w)$ is the harmonic
extension of the Dirichlet boundary trace $\gamma^0_D w$ to $\Omega$ 
(cf.\ \eqref{2.5}).

While at first sight the nonlocal boundary condition 
$\gamma_N v +M_{D,N,\Om}(\gamma_D v)=0$ in \eqref{Wey-11} for the Krein Laplacian 
$-\Delta_{K,\Om}$ may seem familiar from the abstract 
approach to self-adjoint extensions of semibounded symmetric operators within the theory 
of boundary value spaces, there are some crucial distinctions in the concrete case of 
Laplacians on (nonsmooth) domains which will be delineated at the end of Section \ref{s8}. 

For rough domains, matters are more delicate as the nature of the
boundary trace operators and the standard elliptic regularity theory
are both fundamentally affected. Following work in \cite{GM10}, here we
shall consider the class of {\it quasi-convex domains}. The latter is
the subclass of bounded, Lipschitz domains in $\bbR^n$ characterized by
the demand that
\begin{enumerate}
\item[$(i)$] there exists a sequence of relatively compact, $C^2$-subdomains 
exhausting the original domain, and whose second fundamental forms are bounded
from below in a uniform fashion (for a precise formulation see Definition \ref{Def-AC}),
\end{enumerate}
or
\begin{enumerate}
\item[$(ii)$] near every boundary point 
there exists a suitably small $\delta>0$, such that
the boundary is given by the graph of a function
$\varphi:\bbR^{n-1}\to\bbR$ (suitably rotated and translated) which
is Lipschitz and whose derivative satisfy the pointwise $H^{1/2}$-multiplier
condition
\begin{equation} \label{MaS-T4}
\sum_{k=1}^{n-1}\|f_k\,\partial_k\varphi_j\|_{H^{1/2}(\bbR^{n-1})}\leq\delta
\sum_{k=1}^{n-1}\|f_k\|_{H^{1/2}(\bbR^{n-1})},
\quad f_1,...f_{n-1}\in H^{1/2}(\bbR^{n-1}).
\end{equation} 
\end{enumerate}
See Hypothesis \ref{h.Conv} for a precise formulation.
In particular, \eqref{MaS-T4} is automatically satisfied when
$\omega(\nabla\varphi,t)$, the modulus of continuity of $\nabla\varphi$
at scale $t$, satisfies the square-Dini condition
(compare to \cite{MS85}, \cite{MS05}, where this type of domain was 
introduced and studied), 
\begin{equation} \label{MaS-T7}
\int_0^1\Bigl(\frac{\omega(\nabla\varphi;t)}{t^{1/2}}\Bigr)^2\,\frac{dt}{t}
<\infty.
\end{equation} 
In turn, \eqref{MaS-T7} is automatically satisfied if the Lipschitz
function $\varphi$ is of class $C^{1,r}$ for some $r>1/2$.
As a result, examples of quasi-convex domains include: 
\begin{enumerate}
\item[$(i)$] All bounded (geometrically) convex domains.  
\item[$(ii)$] All bounded Lipschitz domains satisfying 
a uniform exterior ball condition (which, informally speaking, 
means that a ball of fixed radius can be ``rolled'' along the
boundary). 
\item[$(iii)$] All open sets which are the image 
of a domain as in $(i),(ii)$ above under a $C^{1,1}$-diffeomorphism.  
\item[$(iv)$] All bounded domains of class $C^{1,r}$ for some $r>1/2$. 
\end{enumerate}
We note that being quasi-convex is a local property of the
boundary. The philosophy behind this concept
is that Lipschitz-type singularities are allowed in the boundary as long
as they are directed outwardly (see Figure\ 1 on p.\ \pageref{Pic}).
The key feature of this class of domains is the fact that the classical
elliptic regularity property
\begin{equation} \label{Df-H1}
\dom(-\Delta_{D,\Om})\subset H^2(\Om),\quad 
\dom(-\Delta_{N,\Om})\subset H^2(\Om)
\end{equation} 
remains valid. In this vein, it is worth recalling that the presence of a
single re-entrant corner for the domain $\Omega$ invalidates \eqref{Df-H1}.
All our results in this paper are actually valid for the class of
bounded Lipschitz domains for which \eqref{Df-H1} holds.
Condition \eqref{Df-H1} is, however, a regularity assumption on the
boundary of the Lipschitz domain $\Om$ and the class of quasi-convex domains
is the largest one for which we know \eqref{Df-H1} to hold.
Under the hypothesis of quasi-convexity, it has been shown
in \cite{GM10} that the Krein Laplacian $-\Delta_{K,\Om}$ (i.e., the Krein--von 
Neumann extension of the Laplacian $-\Delta$ defined on $C^\infty_0(\Omega)$)
in \eqref{Wey-11} is a well-defined self-adjoint operator which agrees
with the operator constructed using the recipe in \eqref{Gr-r2}.

The main issue of the current paper is the study of the
spectral properties of $H_{K,\Om}$, the Krein--von Neumann extension of the
perturbed Laplacian
\begin{equation} \label{Per-D}
-\Delta+V\, \mbox{ on }\,  C^\infty_0(\Omega),
\end{equation} 
in the case where both the potential $V$ and the domain $\Om$ are
nonsmooth. As regards the former, we shall assume that
$0\leq V\in L^\infty(\Omega;d^nx)$, and we shall assume that
$\Omega\subset\bbR^n$ is a quasi-convex domain (more on this shortly).
In particular, we wish to clarify the extent to which a Weyl asymptotic
formula continues to hold for this operator. For us, this undertaking was
originally inspired by the discussion by Alonso and Simon in \cite{AS80}.
At the end of that paper, the authors comment to the effect that
{\it ``It seems to us that the Krein extension of $-\Delta$, i.e.,
$-\Delta$ with the boundary condition $\eqref{B.A-1}$,  
is a natural object and therefore worthy of further study. For example: Are the 
asymptotics of its nonzero eigenvalues given by Weyl's formula?''}
Subsequently we have learned that when $\Omega$ is $C^\infty$-smooth
this has been shown to be the case by Grubb in \cite{Gr83}.
More specifically, in that paper Grubb has proved that if 
$N_{K,\Om}(\cA;\lambda)$ denotes the number of nonzero eigenvalues of 
$A_{K,\Om}$ (defined as in \eqref{Gr-r2}) not exceeding $\lambda$, then
\begin{equation} \label{Df-H2}
\Omega\in C^\infty\,\text{ implies }\,
N_{K,\Om}(\cA;\lambda)=C_{A,n}\lambda^{n/(2m)}+O\big(\lambda^{(n-\theta)/(2m)}\big)
\, \mbox{ as }\, \lambda\rightarrow\infty,
\end{equation} 
where, with $a^0(x,\xi)$ as in \eqref{Wey-5},
\begin{equation} \label{Df-H3}
C_{A,n}:=(2\pi)^{-n}\int_\Omega d^nx \int_{a^0(x,\xi)<1} d^n \xi
\end{equation} 
and
\begin{equation} \label{Df-H4}
\theta:=\max\,\Bigl\{\frac{1}{2}-\varepsilon\,,\,\frac{2m}{2m+n-1}\Bigl\},
\, \mbox{ with $\varepsilon>0$ arbitrary}.
\end{equation} 
See also \cite{Mik94} where the author announces a sharpening of the
remainder in \eqref{Df-H2} to any $\theta<1$ (but no proof is provided).
To show \eqref{Df-H2}--\eqref{Df-H4}, Grubb has reduced the eigenvalue problem
\begin{equation} \label{Df-H5}
\cA u=\lambda\,u,\quad u\in\dom(A_{K,\Om}),\; \lambda>0,
\end{equation} 
to the higher-order, elliptic system
\begin{equation}
\begin{cases}
\cA^2 v=\lambda\, \cA v\,\mbox{ in }\,\Om,
\\
\gamma^{2m}_D v =0\,\mbox{ on }\,\dOm,
\\
v\in C^\infty(\ol{\Omega}). 
\end{cases}   \label{Df-H6}
\end{equation}
Then the strategy is to use known asymptotics for the spectral distribution
function of regular elliptic boundary problems, along with perturbation
results due to Birman, Solomyak, and Grubb (see the
literature cited in \cite{Gr83} for precise references). It should be noted
that the fact that the boundary of $\Omega$ and the coefficients of
$\cA$ are smooth plays an important role in Grubb's proof. First, 
this is used to ensure that \eqref{Kre-DefY} holds which, in turn, allows for
the concrete representation \eqref{Gr-r2} (a formula which in effect
lies at the start of the entire theory, as Grubb adopts this as the {\it definition}
of the domains of the Krein--von Neumann extension). In addition, at a more technical level,
Lemma 3 in \cite{Gr83} is justified by making appeal to the
theory of pseudo-differential operators on $\partial\Omega$, assumed
to be an $(n-1)$-dimensional $C^\infty$ manifold.
In our case, that is, when dealing with the Krein--von Neumann extension
of the perturbed Laplacian \eqref{Per-D}, we establish the following theorem: 

%%%%%%%%%%%%%%%%%%%%%%%%%%%%%%%%%%%%%
\begin{theorem}\label{Th-InM}
Let $\Om\subset\bbR^n$ be a quasi-convex domain, assume that 
$0 \leq V\in L^\infty(\Om; d^nx)$, and denote by
$H_{K,\Om}$ the Krein--von Neumann extension of the perturbed Laplacian
\eqref{Per-D}. Then there exists a sequence of numbers
\begin{equation} \label{Xmam-1}
0<\lambda_{K,\Om,1} \leq \lambda_{K,\Om,2}\leq\cdots\leq\lambda_{K,\Om,j}
\leq\lambda_{K,\Om,j+1} \leq\cdots
\end{equation} 
converging to infinity, with the following properties.
\begin{enumerate}
\item[$(i)$] The spectrum of $H_{K,\Om}$ is given by
\begin{equation} \label{XMi-7}
\sigma(H_{K,\Om})=\{0\}\cup\{\lambda_{K,\Om,j}\}_{j\in\bbN},
\end{equation} 
and each number $\lambda_{K,\Om,j}$, $j\in\bbN$, is an eigenvalue for
$H_{K,\Om}$ of finite multiplicity.
\item[$(ii)$] There exists a countable family of orthonormal
eigenfunctions for $H_{K,\Om}$ which span the orthogonal
complement of the kernel of this operator. More precisely, there exists
a collection of functions $\{w_j\}_{j\in\bbN}$ with the following properties:
\begin{align} \label{Xmam-21}
& w_j\in\dom(H_{K,\Om})\, \mbox{ and }\, 
H_{K,\Om}w_j=\lambda_{K,\Om,j}w_j, \; j\in\bbN,
\\ 
& (w_j,w_k)_{L^2(\Om;d^nx)}=\delta_{j,k}, \; j,k\in\bbN,
\label{Xmam-22}\\
& L^2(\Omega;d^nx)=\ker(H_{K,\Om})\,\oplus\, 
\ol{{\rm lin. \, span} \{w_j\}_{j\in\bbN}},\, \mbox{ $($orthogonal direct sum$)$.}
\label{Xmam-23}
\end{align} 
If $V$ is Lipschitz then $w_j\in H^{1/2}(\Omega)$ for every $j$ and, in fact, 
$w_j\in C^\infty(\overline{\Omega})$ for every $j$ if $\Omega$ is $C^\infty$ and 
$V\in C^\infty(\overline{\Omega})$.  
\item[$(iii)$] The following min-max principle holds:
\begin{align} \label{Xmam-26}
\hspace*{6mm} 
& \lambda_{K,\Om,j}
=\min_{\stackrel{W_j\text{ subspace of }H^2_0(\Om)}{\dim(W_j)=j}}
\bigg(\max_{0\not=u\in W_j} \bigg(\frac{\|(-\Delta+V)u\|^2_{L^2(\Om;d^nx)}}
{\|\nabla u\|^2_{(L^2(\Om;d^nx))^n}+\|V^{1/2}u\|^2_{L^2(\Om;d^nx)}}\bigg)\bigg),  \no \\
& \hspace*{10.85cm} j\in\bbN.
\end{align} 
\item[$(iv)$] If
\begin{equation} \label{Ymam-1}
0<\lambda_{D,\Om,1} \leq\lambda_{D,\Om,2} \leq\cdots\leq\lambda_{D,\Om,j}
\leq\lambda_{D,\Om,j+1} \leq\cdots
\end{equation} 
are the eigenvalues of the perturbed Dirichlet Laplacian $-\Delta_{D,\Om}$
$($i.e., the Friedrichs extension of \eqref{Per-D} in $L^2(\Om;d^nx)$$)$, listed according to their multiplicities, then
\begin{equation} \label{Xmam-39}
0< \lambda_{D,\Om,j} \leq\lambda_{K,\Om,j}, \quad j\in\bbN,
\end{equation} 
Consequently introducing the spectral distribution functions
\begin{equation} \label{Xmam-44}
N_{X,\Om}(\lambda):=\#\{j\in\bbN\,|\,\lambda_{X,\Om,j} \leq\lambda\},
\quad X\in\{D,K\},
\end{equation} 
one has
\begin{equation} \label{Xmam-45}
N_{K,\Om}(\lambda)\leq N_{D,\Om}(\lambda).  
\end{equation} 
\item[$(v)$] Corresponding to the case $V\equiv 0$,
the first nonzero eigenvalue $\lambda^{(0)}_{K,\Om,1}$ of $-\Delta_{K,\Om}$ satisfies
\begin{equation} \label{Xmx-4}
\lambda^{(0)}_{D,\Om,2} \leq \lambda^{(0)}_{K,\Om,1}\, \mbox{ and }\, 
\lambda^{(0)}_{K,\Om,2} \leq\frac{n^2+8n+20}{(n+2)^2}\lambda^{(0)}_{K,\Om,1}.
\end{equation} 
In addition,
\begin{equation} \label{Xmx-1}
\sum_{j=1}^n\lambda^{(0)}_{K,\Om,j+1}
<(n+4)\lambda^{(0)}_{K,\Om,1}
-\frac{4}{n+4}(\lambda^{(0)}_{K,\Om,2}-\lambda^{(0)}_{K,\Om,1})
\le (n+4)\lambda^{(0)}_{K,\Om,1},
\end{equation} 
and
\begin{equation} \label{Xmx-3}
\sum_{j=1}^k \big(\lambda^{(0)}_{K,\Om,k+1}-\lambda^{(0)}_{K,\Om,j}\big)^2 
\leq\frac{4(n+2)}{n^2}
\sum_{j=1}^k \big(\lambda^{(0)}_{K,\Om,k+1}-\lambda^{(0)}_{K,\Om,j}\big)
\lambda^{(0)}_{K,\Om,j} 
\quad k\in\bbN.
\end{equation} 
Moreover, if $\Om$ is a bounded, convex domain in $\bbR^n$,
then the first two Dirichlet eigenvalues and the first nonzero eigenvalue
of the Krein Laplacian in $\Om$ satisfy
\begin{equation} \label{A-P.1U}
\lambda^{(0)}_{D,\Om,2} \leq \lambda^{(0)}_{K,\Om,1} \leq 4\,\lambda^{(0)}_{D,\Om,1}.
\end{equation} 
\item[$(vi)$] The following Weyl asymptotic formula holds:
\begin{equation} \label{Xkko-12}
N_{K,\Om}(\lambda)
=(2\pi)^{-n}v_n|\Omega|\,\lambda^{n/2}+O\big(\lambda^{(n-(1/2))/2}\big)
\, \mbox{ as }\, \lambda\to\infty,
\end{equation} 
where, as before, $v_n$ denotes the volume of the unit ball in $\bbR^n$,
and $|\Omega|$ stands for the $n$-dimensional Euclidean volume of $\Omega$.
\end{enumerate}
\end{theorem}
%%%%%%%%%%%%%%%%%%%%%%%%%%%%%%%%%%%

This theorem answers the question posed by Alonso and Simon in \cite{AS80}
(which corresponds to $V\equiv 0$), and further extends the work by Grubb
in \cite{Gr83} in the sense that we allow nonsmooth domains and coefficients. 
To prove this result, we adopt Grubb's strategy and show that the eigenvalue problem
\begin{equation} \label{Df-H7}
(-\Delta+V)u=\lambda\,u,\quad u\in\dom(H_{K,\Om}),\; \lambda>0,
\end{equation} 
is equivalent to the following fourth-order problem
\begin{equation}\label{Df-H8} 
\begin{cases}
(-\Delta+V)^2w=\lambda\,(-\Delta+V)w\, \mbox{ in }\, \Om,
\\\
\gamma_D w=\gamma_N w=0\, \mbox{ on } \,\dOm,
\\
w\in\dom(-\Delta_{max}).
\end{cases}
\end{equation}
This is closely related to the so-called problem of the {\it buckling of a 
clamped plate},
\begin{equation}\label{Df-H8F}
\begin{cases}
-\Delta^2 w=\lambda\,\Delta w \,\mbox{ in }\, \Om,
\\
\gamma_D w =\gamma_N w =0\, \mbox{ on }\, \dOm,
\\
w\in\dom(-\Delta_{max}), 
\end{cases}
\end{equation}
to which \eqref{Df-H8} reduces when $V\equiv 0$.  From a physical point of 
view, the nature of the later boundary value problem can be described as 
follows. In the two-dimensional setting, the bifurcation problem for a clamped,
homogeneous plate in the shape of $\Omega$, with uniform lateral compression
on its edges has the eigenvalues $\lambda$ of the problem \eqref{Df-H8}
as its critical points. In particular, the first eigenvalue of \eqref{Df-H8}
is proportional to the load compression at which the plate buckles.

One of the upshots of our work in this paper is establishing a definite
connection between the Krein--von Neumann extension of the Laplacian and the 
buckling problem \eqref{Df-H8F}. In contrast to the smooth case, since in our 
setting the solution $w$ of \eqref{Df-H8} does not exhibit any extra 
regularity on the Sobolev scale $H^s(\Om)$, $s\geq 0$, other than membership 
to $L^2(\Om;d^nx)$, a suitable interpretation of the boundary
conditions in \eqref{Df-H8} should be adopted. (Here we shall rely on the
recent progress from \cite{GM10} where this issue has been resolved
by introducing certain novel boundary Sobolev spaces, well-adapted to
the class of Lipschitz domains.) We nonetheless find this trade-off,
between the $2$nd-order boundary problem \eqref{Df-H7} which has
nonlocal boundary conditions, and the boundary problem \eqref{Df-H8}
which has local boundary conditions, but is of fourth-order, very useful.
The reason is that \eqref{Df-H8} can be rephrased, in view of
\eqref{Df-H1} and related regularity results developed in \cite{GM10},
in the form of
\begin{equation} \label{XMM-2}
(-\Delta+V)^2 u=\lambda\,(-\Delta+V)u \,\mbox{ in } \,\Omega,\quad
u\in H^2_0(\Omega).
\end{equation} 
In principle, this opens the door to bringing onto the stage the theory of
generalized eigenvalue problems, that is, operator pencil problems of the form
\begin{equation} \label{XMM-3}
Tu=\lambda\,Su,
\end{equation} 
where $T$ and $S$ are certain linear operators in a Hilbert space.
Abstract results of this nature can be found for instance, in \cite{LM08}, \cite{Pe68}, 
\cite{Tr00} (see also \cite{Le83}, \cite{Le85}, where this is applied to 
the asymptotic distribution of eigenvalues).  We, however, find it more 
convenient to appeal to a version of \eqref{XMM-3} which emphasizes the 
role of the symmetric forms
\begin{align} \label{Xkko-13}
& a(u,v):=\int_{\Om}d^nx\,\ol{(-\Delta+V)u}\,(-\Delta+V)v,
\quad u,v\in H^2_0(\Om),
\\
& b(u,v):=\int_{\Om}d^nx\,\ol{\nabla u}\cdot\nabla v
+\int_{\Om}d^nx\,\ol{V^{1/2}u}\,V^{1/2}v,\quad u,v\in H^2_0(\Om),
\label{Xkko-13F}
\end{align} 
and reformulate \eqref{XMM-2} as the problem of finding $u\in H^2_0(\Om)$
which satisfies
\begin{equation} \label{Xkko-14}
a(u,v)=\lambda\,b(u,v) \quad v\in H^2_0(\Om).
\end{equation} 
This type of eigenvalue problem, in the language of bilinear forms associated
with differential operators, has been studied by Kozlov in a series
of papers \cite{Ko79}, \cite{Ko83}, \cite{Ko84}. In particular, in \cite{Ko84},
Kozlov has obtained Weyl asymptotic formulas in the case where the underlying
domain $\Omega$ in \eqref{Xkko-13} is merely Lipschitz, and the lower-order
coefficients of the quadratic forms \eqref{Xkko-13}--\eqref{Xkko-13F}
are only measurable and bounded (see Theorem \ref{T-Koz}
for a precise formulation). Our demand that the potential $V$ is in
$L^\infty(\Om;d^nx)$ is therefore inherited from Kozlov's theorem.
Based on this result and the fact that the
problems \eqref{Xkko-13}--\eqref{Xkko-14} and \eqref{Df-H7} are
spectral-equivalent, we can then conclude that \eqref{Xkko-12} holds.
Formulas \eqref{Xmx-4}--\eqref{Xmx-3} are also a byproduct of the connection
between \eqref{Df-H7} and \eqref{Df-H8} and known spectral estimates for 
the buckling plate problem from \cite{As04}, \cite{As09}, \cite{AL96}, 
\cite{CY06}, \cite{HY84}, \cite{Pa55}, \cite{Pa67}, \cite{Pa91}. Similarly, 
\eqref{A-P.1U} for convex domains is based on the connection between 
\eqref{Df-H7} and \eqref{Df-H8} and the eigenvalue inequality relating 
the first eigenvalue of a fixed membrane and that of the buckling problem 
for the clamped plate as proven in \cite{Pa60} (see also \cite{Pa67}, 
\cite{Pa91}).  

In closing, we wish to point out that in the $C^\infty$-smooth setting,
Grubb's remainder in \eqref{Df-H2} could, in principle, be sharper than that
in \eqref{Xkko-12}. However, the main novel feature of our Theorem \ref{Th-InM}
is the low regularity assumptions on the underlying domain $\Omega$,
and the fact that we allow a nonsmooth potential $V$.
As was the case with the Weyl asymptotic formula for the classical
Dirichlet and Neumann Laplacians (briefly reviewed at the beginning of this
section), the issue of regularity (or lack thereof) has always been of
considerable importance in this line of work (as early as
1970, Birman and Solomyak noted in \cite{BS70} that
``{\it there has been recently some interest in obtaining the classical
asymptotic spectral formulas under the weakest possible hypotheses}.''). 
The interested reader may consult
the paper \cite{BS79} by Birman and Solomyak (see also \cite{BS72}, 
\cite{BS73}), as well as
the article \cite{Da97} by Davies for some very readable, highly
informative surveys underscoring this point (collectively, these papers
also contain more than 500 references concerning this circle of ideas).

Finally, a notational comment: For obvious reasons in connection 
with quantum mechanical applications, we will, with a slight abuse of 
notation, dub $-\Delta$ (rather than $\Delta$) as the ``Laplacian'' in 
this paper.

%%%%%%%%%%%%%%%%%%%%%%%%%%%%%%%%%%%%%%
%%%%%%%%%%%%%%%%%%%%%%%%%%%%%%%%%%%%%%
\section{The Abstract Krein--von Neumann Extension}
\label{s2}
%%%%%%%%%%%%%%%%%%%%%%%%%%%%%%%%%%%%%%
%%%%%%%%%%%%%%%%%%%%%%%%%%%%%%%%%%%%%%

To get started, we briefly elaborate on the notational conventions used
throughout this paper and especially throughout this section which collects abstract material on the Krein--von Neumann extension. Let $\cH$ be a separable complex Hilbert space, 
$(\dott,\dott)_{\cH}$ the scalar product in $\cH$ (linear in
the second factor), and $I_{\cH}$ the identity operator in $\cH$.
Next, let $T$ be a linear operator mapping (a subspace of) a
Banach space into another, with $\dom(T)$ and $\ran(T)$ denoting the
domain and range of $T$. The closure of a closable operator $S$ is
denoted by $\ol S$. The kernel (null space) of $T$ is denoted by
$\ker(T)$. The spectrum, essential spectrum, and resolvent set of a closed linear operator in $\cH$ will be denoted by $\sigma(\cdot)$, $\sigma_{\rm ess}(\cdot)$, 
and $\rho(\cdot)$, respectively. The
Banach spaces of bounded and compact linear operators on $\cH$ are
denoted by $\cB(\cH)$ and $\cB_\infty(\cH)$, respectively. Similarly,
the Schatten--von Neumann (trace) ideals will subsequently be denoted
by $\cB_p(\cH)$, $p\in (0,\infty)$. The analogous notation $\cB(\cX_1,\cX_2)$,
$\cB_\infty (\cX_1,\cX_2)$, etc., will be used for bounded, compact,
etc., operators between two Banach spaces $\cX_1$ and $\cX_2$. 
Moreover, $\cX_1\hookrightarrow \cX_2$ denotes the continuous embedding
of the Banach space $\cX_1$ into the Banach space $\cX_2$. 
In addition, $U_1 \dotplus U_2$ denotes the direct sum of the subspaces $U_1$ 
and $U_2$ of a Banach space $\cX$; and $V_1 \oplus V_2$ represents the orthogonal direct 
sum of the subspaces $V_j$, $j=1,2$, of a Hilbert space $\cH$. 

Throughout this manuscript, if $X$ denotes a Banach space, $X^*$ 
denotes the {\it adjoint space} of continuous conjugate linear functionals 
on $X$, that is, the {\it conjugate dual space} of $X$ (rather than the usual 
dual space of continuous linear functionals 
on $X$). This avoids the well-known awkward distinction between adjoint 
operators in Banach and Hilbert spaces (cf., e.g., the pertinent discussion 
in \cite[p.\ 3, 4]{EE89}). 

Given a reflexive Banach space $\cV$ and $T \in\cB(\cV,\cV^*)$,
the fact that $T$ is self-adjoint is defined by the requirement that
\begin{equation}\label{B.5}
{}_{\cV}\langle u,T v \rangle_{\cV^*}
= {}_{\cV^*}\langle T u, v \rangle_{\cV}
= \ol{{}_{\cV}\langle v, T u \rangle_{\cV^*}}, \quad u, v \in \cV,
\end{equation}
where in this context bar denotes complex conjugation, $\cV^*$ is the conjugate dual of $\cV$, and
${}_{\cV}\langle\dott,\dott\rangle_{\cV^*}$ stands for the $\cV, \cV^*$ pairing.

A linear operator $S:\dom(S)\subseteq\cH\to\cH$, is called {\it symmetric}, if
\begin{equation}\label{Pos-2}
(u,Sv)_\cH=(Su,v)_\cH, \quad u,v\in \dom (S).
\end{equation}
If $\dom(S)=\cH$, the classical Hellinger--Toeplitz theorem guarantees that $S\in\cB(\cH)$, in which situation $S$ is readily seen to be self-adjoint. In general, however, symmetry is a considerably weaker property than self-adjointness and a classical problem in functional analysis is that of determining all self-adjoint extensions in $\cH$ of a given unbounded symmetric operator of equal and nonzero deficiency indices. (Here self-adjointness of an operator $\wti S$ in $\cH$, is 
of course defined as usual by $\big(\wti S\big)^* = \wti S$.) In this manuscript we will be particularly interested in this question within the class of 
densely defined (i.e., $\ol{\dom(S)}=\cH$), nonnegative operators (in fact, in most instances $S$ will even turn out to be strictly positive) and we focus almost exclusively on self-adjoint extensions that are nonnegative operators. In the latter scenario, there are two distinguished constructions which we will briefly review next. 

To set the stage, we recall that a linear operator $S:\dom(S)\subseteq\cH\to \cH$
is called {\it nonnegative} provided
\begin{equation}\label{Pos-1}
(u,Su)_\cH\geq 0, \quad u\in \dom(S).
\end{equation}
(In particular, $S$ is symmetric in this case.) $S$ is called {\it strictly positive}, if for some 
$\varepsilon >0$, $(u,Su)_\cH\geq \varepsilon \|u\|_{\cH}^2$, $u\in \dom(S)$. 
Next, we recall that $A \leq B$ for two self-adjoint operators in $\cH$ if 
\begin{align}
\begin{split}
& \dom\big(|A|^{1/2}\big) \supseteq \dom\big(|B|^{1/2}\big) \, \text{ and } \\ 
& \big(|A|^{1/2}u, U_A |A|^{1/2}u\big)_{\cH} \leq \big(|B|^{1/2}u, U_B |B|^{1/2}u\big)_{\cH}, \quad  
u \in \dom\big(|B|^{1/2}\big),      \lb{AleqB} 
\end{split}
\end{align}
where $U_C$ denotes the partial isometry in $\cH$ in the polar decomposition of 
a densely defined closed operator $C$ in $\cH$, $C=U_C |C|$, $|C|=(C^* C)^{1/2}$. (If 
in addition, $C$ is self-adjoint, then $U_C$ and $|C|$ commute.) 
We also recall (\cite[Part II]{Fa75}, \cite[Theorem VI.2.21]{Ka80}) that if $A$ and $B$ are both self-adjoint and nonnegative in $\cH$, then 
\begin{align}
\begin{split}
& 0 \leq A\leq B  \, \text{ if and only if } \, 0 \leq A^{1/2 }\leq B^{1/2},     \label{PPa-1} \\
& \quad \text{equivalently, if and only if } \,  
(B + a I_\cH)^{-1} \leq (A + a I_\cH)^{-1} \, \text{ for all $a>0$,} 
\end{split}
\end{align}
and
\begin{equation}
\ker(A) =\ker\big(A^{1/2}\big)
\end{equation}
(with $C^{1/2}$ the unique nonnegative square root of a nonnegative self-adjoint operator $C$ in $\cH$).

For simplicity we will always adhere to the conventions that $S$ is a linear, unbounded, densely defined, nonnegative (i.e., $S\geq 0$) operator in $\cH$, and that $S$ has nonzero deficiency indices.  In particular,
\begin{equation}
{\rm def} (S) = \dim (\ker(S^*-z I_{\cH})) \in \bbN\cup\{\infty\}, 
\quad z\in \bbC\backslash [0,\infty), 
\lb{DEF}
\end{equation}
is well-known to be independent of $z$. 
Moreover, since $S$ and its closure $\ol{S}$ have the same self-adjoint extensions in $\cH$, we will without loss of generality assume that $S$ is closed in the remainder of this section.

The following is a fundamental result to be found in M.\ Krein's celebrated 1947 paper
 \cite{Kr47} (cf.\ also Theorems\ 2 and 5--7 in the English summary on page 492): 
 
%%%%%%%%%%%%%%%%%%%%%%%%%%%%%%%%%%%%%%%%%%%%
\begin{theorem}\label{T-kkrr}
Assume that $S$ is a densely defined, closed, nonnegative operator in $\cH$. Then, among all 
nonnegative self-adjoint extensions of $S$, there exist two distinguished ones, $S_K$ and $S_F$, which are, respectively, the smallest and largest
$($in the sense of order between self-adjoint operators, cf.\ \eqref{AleqB}$)$ such extension. Furthermore, a nonnegative self-adjoint operator $\widetilde{S}$ is a self-adjoint extension of $S$ if and only if $\widetilde{S}$ satisfies 
\begin{equation}\label{Fr-Sa}
S_K\leq\widetilde{S}\leq S_F.
\end{equation}
In particular, \eqref{Fr-Sa} determines $S_K$ and $S_F$ uniquely. \\

In addition,  if $S\geq \varepsilon I_{\cH}$ for some $\varepsilon >0$, one has 
$S_F \geq \varepsilon I_{\cH}$, and 
\begin{align}
\dom (S_F) &= \dom (S) \dotplus (S_F)^{-1} \ker (S^*),     \lb{SF}  \\
\dom (S_K) & = \dom (S) \dotplus \ker (S^*),    \lb{SK}   \\
\dom (S^*) & = \dom (S) \dotplus (S_F)^{-1} \ker (S^*) \dotplus \ker (S^*)  \no \\
& = \dom (S_F) \dotplus \ker (S^*),    \lb{S*} 
\end{align}
in particular, 
\begin{equation} \label{Fr-4Tf}
\ker(S_K)= \ker\big((S_K)^{1/2}\big)= \ker(S^*) = \ran(S)^{\bot}.
\end{equation} 
\end{theorem}
%%%%%%%%%%%%%%%%%%%%%%%%%%%%%%%%%%%%% 

Here the operator inequalities in \eqref{Fr-Sa} are understood in the sense of 
\eqref{AleqB} and hence they can  equivalently be written as
\begin{equation}
(S_F + a I_{\cH})^{-1} \le \big(\wti S + a I_{\cH}\big)^{-1} \le (S_K + a I_{\cH})^{-1} 
\, \text{ for some (and hence for all\,) $a > 0$.}    \lb{Res}
\end{equation}

We will call the operator $S_K$ the {\it Krein--von Neumann extension}
of $S$. See \cite{Kr47} and also the discussion in \cite{AS80}, \cite{AT03}, \cite{AT05}. It should be
noted that the Krein--von Neumann extension was first considered by von Neumann 
\cite{Ne29} in 1929 in the case where $S$ is strictly positive, that is, if 
$S \geq \varepsilon I_{\cH}$ for some $\varepsilon >0$. (His construction appears in the proof of Theorem 42 on pages 102--103.) However, von Neumann did not isolate the extremal property of this extension as described in \eqref{Fr-Sa} and \eqref{Res}. M.\ Krein \cite{Kr47}, \cite{Kr47a} was the first to systematically treat the general case $S\geq 0$ and to study all nonnegative self-adjoint extensions of $S$, illustrating the special role of the {\it Friedrichs extension} (i.e., the ``hard'' extension) $S_F$ of $S$ and the Krein--von Neumann (i.e., the ``soft'') extension $S_K$ of $S$ as extremal cases when considering all nonnegative extensions of $S$. For a recent exhaustive treatment of 
self-adjoint extensions of semibounded operators we refer to \cite{AT02}--\cite{AT09}. 

For classical references on the subject of self-adjoint extensions of semibounded operators (not necessarily restricted to the Krein--von Neumann extension) we refer to Birman \cite{Bi56}, \cite{Bi08}, Friedrichs \cite{Fr34}, Freudenthal \cite{Fr36}, Grubb \cite{Gr68}, \cite{Gr70},  
Krein \cite{Kr47a}, {\u S}traus \cite{St73}, and Vi{\u s}ik \cite{Vi63} (see also the monographs by Akhiezer and Glazman \cite[Sect. 109]{AG81a}, 
Faris \cite[Part III]{Fa75}, and the recent book by Grubb \cite[Sect.\ 13.2]{Gr09}).  

An intrinsic description of the Friedrichs extension $S_F$ of $S\geq 0$ due to 
Freudenthal \cite{Fr36} in 1936 describes $S_F$ as the operator 
$S_F:\dom(S_F)\subset\cH\to\cH$ given by   
\begin{align}
& S_F u:=S^*u,   \no \\
& u \in \dom(S_F):=\big\{v\in\dom(S^*)\,\big|\,  \mbox{there exists} \, 
\{v_j\}_{j\in\bbN}\subset \dom(S),    \label{Fr-2} \\
& \quad \mbox{with} \, \lim_{j\to\infty}\|v_j-v\|_{\cH}=0  
\mbox{ and } ((v_j-v_k),S(v_j-v_k))_\cH\to 0 \mbox{ as }  j,k\to\infty\big\}.     \no 
\end{align}
Then, as is well-known,  
\begin{align}
& S_F \geq 0,  \label{Fr-4}  \\
& \dom\big((S_F)^{1/2}\big)=\big\{v\in\cH\,\big|\, \mbox{there exists} \, 
\{v_j\}_{j\in\bbN}\subset \dom(S),   \label{Fr-4J} \\
& \quad \mbox{with} \lim_{j\to\infty}\|v_j-v\|_{\cH}=0  
\mbox{ and } ((v_j-v_k),S(v_j-v_k))_\cH\to 0\mbox{ as }
j,k\to\infty\big\},  \no 
\end{align}
and
\begin{equation}\label{Fr-4H}
S_F=S^*|_{\dom(S^*)\cap\dom((S_{F})^{1/2})}.
\end{equation}

Equations \eqref{Fr-4J} and \eqref{Fr-4H} are intimately related to the definition of $S_F$ via (the closure of) the sesquilinear form generated by $S$ as follows: One introduces the sesquilinear form
\begin{equation}
q_S(f,g)=(f,Sg)_{\cH}, \quad f, g \in \dom(q_S)=\dom(S). 
\end{equation}
Since $S\geq 0$, the form $q_S$ is closable and we denote by $Q_S$ the closure of 
$q_S$. Then $Q_S\geq 0$ is densely defined and closed. By the first and second representation theorem for forms (cf., e.g., \cite[Sect.\ 6.2]{Ka80}), $Q_S$ is uniquely associated with a nonnegative, self-adjoint operator in $\cH$. This operator is precisely the Friedrichs extension, $S_F \geq 0$, of $S$, and hence,
\begin{align}
\begin{split} 
& Q_S(f,g)=(f,S_F g)_{\cH}, \quad f \in \dom(Q_S), \, g \in \dom(S_F),  \lb{Fr-Q} \\ 
& \dom(Q_S) = \dom\big((S_F)^{1/2}\big).     
\end{split} 
\end{align}

An intrinsic description of the Krein--von Neumann extension $S_K$ of $S\geq 0$ has been given by Ando and Nishio \cite{AN70} in 1970, where $S_K$ has been characterized as the operator 
$S_K:\dom(S_K)\subset\cH\to\cH$ given by
\begin{align} 
& S_Ku:=S^*u,   \no \\
& u \in \dom(S_K):=\big\{v\in\dom(S^*)\,\big|\,\mbox{there exists} \, 
\{v_j\}_{j\in\bbN}\subset \dom(S),    \label{Fr-2X}  \\ 
& \quad \mbox{with} \, \lim_{j\to\infty} \|Sv_j-S^*v\|_{\cH}=0  
\mbox{ and } ((v_j-v_k),S(v_j-v_k))_\cH\to 0 \mbox{ as } j,k\to\infty\big\}.  \no
\end{align}

By \eqref{Fr-2} one observes that shifting $S$ by a constant commutes with the operation of taking the Friedrichs extension of $S$, that is, for any 
$c\in\bbR$,
\begin{equation}
(S + c I_{\cH})_{F} = S_F + c I_{\cH},    \lb{Fr-c}
\end{equation}
but by \eqref{Fr-2X}, the analog of \eqref{Fr-c} for the Krein--von Neumann extension 
$S_K$ fails.

At this point we recall a result due to Makarov and Tsekanovskii \cite{MT07}, concerning symmetries (e.g., the rotational symmetry exploited in Section \ref{s1vi}), and more generally, a scale invariance,  shared by $S$, $S^*$, $S_F$, and $S_K$ (see also \cite{HK09}). Actually, we will prove a slight extension of the principal result in \cite{MT07}:

%%%%%%%%%%%%%%%%%%%%%%%%%%%%%%%%%%%%%%%
\begin{proposition}  \lb{p2.2a}
Let $\mu > 0$, suppose that $V, V^{-1} \in \cB(\cH)$, and assume $S$ to be a densely defined, closed, nonnegative operator in $\cH$ satisfying  
\begin{equation}
V S V^{-1} = \mu S,    \lb{VS}
\end{equation}
and
\begin{equation}
V S V^{-1} = (V^*)^{-1} S V^* \, \text{ $($or equivalently, $(V^* V)^{-1} S (V^* V) = S$\,$)$.}  
\end{equation}
Then also $S^*$, $S_F$, and $S_K$ satisfy 
\begin{align}
(V^* V)^{-1} S^* (V^* V) &= S^*,  \,\,\,\quad  V S^* V^{-1} = \mu S^*,    \lb{VS*} \\
(V^* V)^{-1} S_F (V^* V) &= S_F,  \, \quad  V S_F V^{-1} = \mu S_F,  \lb{VSF}\\
(V^* V)^{-1} S_K (V^* V) &= S_K,  \quad V S_K V^{-1} = \mu S_K.  \lb{VSK} 
\end{align}
\end{proposition} 
%%%%%%%%%%%%%%%%%%%%%%%%%%%%%%%%%%%%%%%
\begin{proof}
Applying \cite[p.\ 73, 74]{We80}, \eqref{VS} yields $V S V^{-1} = (V^*)^{-1} S V^*$. The latter 
relation is equivalent to $(V^* V)^{-1} S (V^* V) = S$ and hence also equivalent to 
$(V^* V) S (V^* V)^{-1} = S$. Taking adjoints (and applying \cite[p.\ 73, 74]{We80} again) 
then yields $(V^*)^{-1} S^* V^* = V S^* V^{-1}$; the latter is equivalent to 
$(V^* V)^{-1} S^* (V^* V) = S^*$ and hence also equivalent to $(V^* V) S^* (V^* V)^{-1} = S$. 
Replacing $S$ and $S^*$ by $(V^* V)^{-1} S (V^* V)$ and $(V^* V)^{-1} S^* (V^* V)$, respectively, 
in \eqref{Fr-2}, and subsequently, in \eqref{Fr-2X}, then yields that 
\begin{equation}
(V^* V)^{-1} S_F (V^* V) = S_F  \, \text{ and } \,  
(V^* V)^{-1} S_K (V^* V) = S_K. 
\end{equation}
The latter are of course equivalent to 
\begin{equation}
(V^* V) S_F (V^* V)^{-1} = S_F  \, \text{ and } \,  
(V^* V) S_K (V^* V)^{-1} = S_K. 
\end{equation}
Finally, replacing $S$ by $V S V^{-1}$ and $S^*$ by $V S^* V^{-1}$ in  \eqref{Fr-2}  
then proves $V S_F V^{-1} = \mu S_F$. Performing the same replacement in  \eqref{Fr-2X} 
then yields $V S_K V^{-1} = \mu S_K$.  
\end{proof}
%%%%%%%%%%%%%%%%%%%%%%%%%%%%%%%%%%%%%%%

If in addition, $V$ is unitary (implying $V^* V = I_{\cH}$), Proposition \ref{p2.2a} immediately reduces to \cite[Theorem\ 2.2]{MT07}. In this special case one can also provide a quick alternative proof by directly invoking the inequalities \eqref{Res} and the fact that they are preserved under unitary equivalence. 

Similarly to Proposition \ref{p2.2a}, the following results also immediately follows from the characterizations \eqref{Fr-2} and \eqref{Fr-2X} of $S_F$ and $S_K$, respectively:

%%%%%%%%%%%%%%%%%%%%%%%%%%%%%%%%%%%%%%%
\begin{proposition}  \lb{p2.3}  
Let $U\colon\cH_1\to\cH_2$ be unitary from $\cH_1$ onto $\cH_2$ and assume $S$ to be a densely defined, closed, nonnegative operator in $\cH_1$ with adjoint 
$S^*$, Friedrichs extension $S_F$, and Krein--von Neumann extension $S_K$ in 
$\cH_1$, respectively. Then the adjoint, Friedrichs extension, and Krein--von Neumann extension of the nonnegative, closed, densely defined, symmetric operator $USU^{-1}$ in $\cH_2$ are given by 
\begin{align}
[USU^{-1}]^* = US^*U^{-1},  \quad 
[USU^{-1}]_F = US_F U^{-1},  \quad 
[USU^{-1}]_K = US_K U^{-1}
\end{align}
in $\cH_2$, respectively.
\end{proposition} 
%%%%%%%%%%%%%%%%%%%%%%%%%%%%%%%%%%%%%%%

%%%%%%%%%%%%%%%%%%%%%%%%%%%%%%%%%%%%%%%
\begin{proposition}  \lb{p2.4}  
Let $J\subseteq \bbN$ be some countable index set and consider 
$\cH = \bigoplus_{j\in J} \cH_j$ and $S=\bigoplus_{j\in J} S_j$, where each $S_j$ is a densely defined, closed, nonnegative operator in 
$\cH_j$, $j\in J$. Denoting by $(S_j)_F$ and $(S_j)_K$ the Friedrichs and 
Krein--von Neumann extension of $S_j$ in $\cH_j$, $j\in J$, one infers  
\begin{equation}
S^*=\bigoplus_{j\in J} \; (S_j)^*, \quad S_F=\bigoplus_{j\in J} \; (S_j)_F, 
\quad S_K = \bigoplus_{j\in J} \; (S_j)_K. 
\end{equation}
\end{proposition} 
%%%%%%%%%%%%%%%%%%%%%%%%%%%%%%%%%%%%%%%

The following is a consequence of a slightly more general result formulated
in \cite[Theorem 1]{AN70}: 

%%%%%%%%%%%%%%%%%%%%%%%%%%%%%%%%%%%%%%%
\begin{proposition}\label{Pr-an}
Let $S$ be a densely defined, closed, nonnegative operator in $\cH$.
Then $S_K$, the Krein--von Neumann extension of $S$, has the
property that
\begin{equation}\label{an-T1}
\dom\big((S_K)^{1/2}\big)=\biggl\{u\in\cH\,\bigg|\,\sup_{v\in\dom(S)}
\frac{|(u,Sv)_\cH|^2}{(v,Sv)_\cH}
<+\infty\biggr\},
\end{equation}
and
\begin{equation}\label{an-T2}
\big\|(S_K)^{1/2}u\big\|^2_\cH=\sup_{v\in\dom(S)}
\frac{|(u,Sv)_\cH|^2}{(v,Sv)_\cH}, \quad  u\in\dom\big((S_K)^{1/2}\big).  
\end{equation}
\end{proposition}
%%%%%%%%%%%%%%%%%%%%%%%%%%%%%%%%%%%%%%%

A word of explanation is in order here: Given $S\geq 0$ as in the statement of
Proposition \ref{Pr-an}, the Cauchy-Schwarz-type inequality
\begin{equation}\label{CS-I.1}
|(u,Sv)_{\cH}|^2\leq (u,Su)_{\cH} (v,Sv)_{\cH}, \quad u,v\in\dom(S), 
\end{equation}
shows (due to the fact that $\dom(S)\hookrightarrow\cH$ densely) that
\begin{equation}\label{CS-I.2}
u\in\dom(S) \,\mbox{ and } \, (u,S u)_{\cH} =0 \,
\text{ imply }\,Su=0.
\end{equation}
Thus, whenever the denominator of the fractions appearing
in \eqref{an-T1}, \eqref{an-T2} vanishes, so does the numerator, and
one interprets $0/0$ as being zero in \eqref{an-T1}, \eqref{an-T2}.

We continue by recording an abstract result regarding the
parametrization of all nonnegative self-adjoint extensions of a given
strictly positive, densely defined, symmetric operator. The following results were 
developed from Krein \cite{Kr47}, Vi{\u s}ik \cite{Vi63}, and Birman \cite{Bi56}, by 
Grubb \cite{Gr68}, \cite{Gr70}. Subsequent expositions are due to Faris 
\cite[Sect.\ 15]{Fa75}, Alonso and Simon \cite{AS80} (in the present form, the next theorem appears in \cite{GM10}), and Derkach and Malamud \cite{DM91}, \cite{Ma92}. We start by collecting our basic assumptions:

%%%%%%%%%%%%%
\begin{hypothesis}  \lb{h2.6}
Suppose that $S$ is a densely defined, symmetric, closed operator 
with nonzero deficiency indices in $\cH$ that satisfies 
\begin{equation}
S\geq \varepsilon I_{\cH} \, \text{ for some $\varepsilon >0$.}    \lb{3.1}
\end{equation}
\end{hypothesis}
%%%%%%%%%%%%%

%%%%%%%%%%%%%%%%%%%%%%%%%%%%%%%%%
\begin{theorem}\label{AS-th}
Suppose Hypothesis \ref{h2.6}. Then there exists a one-to-one correspondence between nonnegative self-adjoint operators 
$0 \leq B:\dom(B)\subseteq \cW\to \cW$, $\ol{\dom(B)}=\cW$, where $\cW$
is a closed subspace of $\cN_0 :=\ker(S^*)$, and nonnegative self-adjoint
extensions $S_{B,\cW}\geq 0$ of $S$. More specifically, $S_F$ is invertible, 
$S_F\geq \varepsilon I_{\cH}$, 
and one has
\begin{align} 
& \dom(S_{B,\cW})  %\no \\
%& \quad 
=\big\{f + (S_F)^{-1}(Bw + \eta)+ w \,\big|\,
f\in\dom(S),\, w\in\dom(B),\, \eta\in \cN_0 \cap \cW^{\bot}\big\},  \no \\
& S_{B,\cW} = S^*|_{\dom(S_{B,\cW})},       \label{AS-2}  
\end{align} 
where $\cW^{\bot}$ denotes the orthogonal complement of $\cW$ in $\cN_0$. In 
addition,
\begin{align}
&\dom\big((S_{B,\cW})^{1/2}\big) = \dom\big((S_F)^{1/2}\big) \dotplus 
\dom\big(B^{1/2}\big),   \\
&\big\|(S_{B,\cW})^{1/2}(u+g)\big\|_{\cH}^2 =\big\|(S_F)^{1/2} u\big\|_{\cH}^2 
+ \big\|B^{1/2} g\big\|_{\cH}^2, \\ 
& \hspace*{2.3cm} u \in \dom\big((S_F)^{1/2}\big), \; g \in \dom\big(B^{1/2}\big),  \no
\end{align}
implying,
\begin{equation}\label{K-ee}
\ker(S_{B,\cW})=\ker(B). 
\end{equation} 
Moreover, 
\begin{equation}
B \leq \wti B \, \text{ implies } \, S_{B,\cW} \leq S_{\wti B,\wti \cW},
\end{equation}
where 
\begin{align}
\begin{split}
& B\colon \dom(B) \subseteq \cW \to \cW, \quad 
 \wti B\colon \dom\big(\wti B\big) \subseteq \wti \cW \to \wti \cW,   \\
& \ol{\dom\big(\wti B\big)} = \wti\cW \subseteq \cW = \ol{\dom(B)}. 
\end{split}
\end{align}

In the above scheme, the Krein--von Neumann extension $S_K$
of $S$ corresponds to the choice $\cW=\cN_0$ and $B=0$ $($with 
$\dom(B)=\dom\big(B^{1/2}\big)=\cN_0=\ker (S^*)$$)$.
In particular, one thus recovers \eqref{SK}, and \eqref{Fr-4Tf}, and also obtains
\begin{align}
&\dom\big((S_{K})^{1/2}\big) = \dom\big((S_F)^{1/2}\big) \dotplus \ker (S^*),  
\lb{SKform1}  \\
&\big\|(S_{K})^{1/2}(u+g)\big\|_{\cH}^2 =\big\|(S_F)^{1/2} u\big\|_{\cH}^2, 
\quad u \in \dom\big((S_F)^{1/2}\big), \; g \in \ker (S^*).   \lb{SKform2}
\end{align}
Finally, the Friedrichs extension $S_F$ corresponds to the choice $\dom(B)=\{0\}$ $($i.e., formally, 
$B\equiv\infty$$)$, in which case one recovers \eqref{SF}. 
\end{theorem}
%%%%%%%%%%%%%%%%%%%%%%%%%%%%%%%%%%

The relation $B \leq \wti B$ in the case where $\wti \cW \subsetneqq \cW$ requires an explanation: In analogy to \eqref{AleqB} we mean 
\begin{equation}
\big(|B|^{1/2}u, U_B |B|^{1/2}u\big)_{\cW} \leq \big(|\wti B|^{1/2}u, U_{\wti B} |\wti B|^{1/2}u\big)_{\cW}, 
\quad u \in \dom\big(|\wti B|^{1/2}\big) 
\end{equation}
and (following \cite{AS80}) we put
\begin{equation}
\big(|\wti B|^{1/2}u, U_{\wti B} |\wti B|^{1/2}u)_{\cW} = \infty \, \text{ for } \, 
u \in \cW \backslash \dom\big(|\wti B|^{1/2}\big).
\end{equation}

For subsequent purposes we also note that under the assumptions on $S$ in 
Hypothesis \ref{h2.6}, one has 
\begin{equation}
\dim(\ker (S^*-z I_{\cH})) = \dim(\ker(S^*)) = \dim (\cN_0) = {\rm def} (S), 
\quad z\in \bbC\backslash [\varepsilon,\infty).   \lb{dim}
\end{equation}

The following result is a simple consequence of \eqref{SK}, \eqref{SF}, and 
\eqref{Fr-2X}, but since it seems not to have been explicitly stated in \cite{Kr47}, we provide the short proof for completeness (see also \cite[Remark\ 3]{Ma92}). First we recall that two self-adjoint extensions $S_1$ and $S_2$ of $S$ are called 
{\it relatively prime} if $\dom (S_1) \cap \dom (S_2) = \dom (S)$.

%%%%%%%%%%%%%%%%%%%%%%%%%%%%%%%%%%
\begin{lemma}  \lb{lKF}
Suppose Hypothesis \ref{h2.6}. Then $S_F$ and $S_K$ are 
relatively prime, that is,
\begin{equation}
\dom (S_F) \cap \dom (S_K) = \dom (S).    \lb{RP}
\end{equation}
\end{lemma}
%%%%%%%%%%%%%%%%%%%%%%%%%%%%%%%%%%
\begin{proof}
By \eqref{SF} and \eqref{SK} it suffices to prove that 
$\ker (S^*) \cap (S_F)^{-1}\ker (S^*) = \{0\}$. Let 
$f_0 \in \ker (S^*) \cap (S_F)^{-1}\ker (S^*)$. Then $S^* f_0 =0$ and 
$f_0=(S_F)^{-1}g_0$ for some $g_0 \in \ker (S^*)$. Thus one concludes that 
$f_0 \in\dom (S_F)$ and $S_F f_0 =g_0$. But $S_F = S^*|_{\dom (S_F)}$ and hence 
$g_0 =S_F f_0 = S^* f_0 = 0$. Since $g_0 =0$ one finally obtains $f_0 =0$. 
\end{proof}
%%%%%%%%%%%%%%%%%%%%%%%%%%%%%%%%%%

Next, we consider a self-adjoint operator
\begin{equation} \label{Barr-1}
T:\dom(T)\subseteq \cH\to\cH,\quad T=T^*,
\end{equation} 
which is bounded from below, that is, there exists $\alpha\in\bbR$ such that
\begin{equation} \label{Barr-2}
T\geq \alpha I_{\cH}.
\end{equation} 
We denote by $\{E_T(\lambda)\}_{\lambda\in\bbR}$ the family of strongly right-continuous spectral projections of $T$, and introduce, as usual, $E_T((a,b))=E_T(b_-) - E_T(a)$, 
$E_T(b_-) = \slim_{\varepsilon\downarrow 0}E_T(b-\varepsilon)$, $-\infty \leq a < b$. In addition, we set 
\begin{equation} \label{Barr-3}
\mu_{T,j}:=\inf\,\bigl\{\lambda\in\bbR\,|\,
\dim (\ran (E_T((-\infty,\lambda)))) \geq j\bigr\},\quad j\in\bbN.
\end{equation} 
Then, for fixed $k\in\bbN$, either: \\
$(i)$ $\mu_{T,k}$ is the $k$th eigenvalue of $T$ counting multiplicity below the bottom of the essential spectrum, $\sigma_{\rm ess}(T)$, of $T$, \\
or, \\
$(ii)$ $\mu_{T,k}$ is the bottom of the essential spectrum of $T$, 
\begin{equation}
\mu_{T,k} = \inf \{\lambda \in \bbR \,|\, \lambda \in \sigma_{\rm ess}(T)\}, 
\end{equation}
and in that case $\mu_{T,k+\ell} = \mu_{T,k}$, $\ell\in\bbN$, and there are at most $k-1$ eigenvalues (counting multiplicity) of $T$ below $\mu_{T,k}$. 

We now record a basic result of M.\ Krein \cite{Kr47} with an important extension due 
to Alonso and Simon \cite{AS80} and some additional results recently derived in 
\cite{AGMST09}. For this purpose we introduce the {\it reduced  
Krein--von Neumann operator} $\hatt S_K$ in the Hilbert space (cf.\ \eqref{Fr-4Tf})
\begin{equation}
\hatt \cH = [\ker (S^*)]^{\bot} = \big[I_{\cH} - P_{\ker(S^*)}\big] \cH 
= \big[I_{\cH} - P_{\ker(S_K)}\big] \cH = [\ker (S_K)]^{\bot},    \lb{hattH}
\end{equation}
by 
\begin{align}
\hatt S_K:&=S_K|_{[\ker(S_K)]^{\bot}}    \label{Barr-4} \\
\begin{split}
& = S_K[I_{\cH} - P_{\ker(S_K)}]   \lb{SKP} \, \text{ in $[I_{\cH} - P_{\ker(S_K)}]\cH$} \\
&= [I_{\cH} - P_{\ker(S_K)}]S_K[I_{\cH} - P_{\ker(S_K)}] 
\, \text{ in $[I_{\cH} - P_{\ker(S_K)}]\cH$},  
\end{split}
\end{align} 
where $P_{\ker(S_K)}$ denotes the orthogonal projection onto $\ker(S_K)$ and we are alluding to the orthogonal direct sum decomposition of $\cH$ into 
\begin{equation}
\cH = P_{\ker(S_K)}\cH \oplus [I_{\cH} - P_{\ker(S_K)}]\cH.
\end{equation}
Assuming Hypothesis \ref{h2.6}, we recall that Krein \cite{Kr47} (see also 
\cite[Corollary\ 5]{Ma92} for a generalization to the case $S\geq 0$) proved the formula
\begin{equation}
\big(\hatt S_K\big)^{-1} = [I_{\cH} - P_{\ker(S_K)}] (S_F)^{-1} [I_{\cH} - P_{\ker(S_K)}].   
\lb{SKinv}
\end{equation}

%%%%%%%%%%%%%%%%%%%%%%%%%%%%%%%%%%%%%%
\begin{theorem}  \lb{AS-thK}
Suppose Hypothesis \ref{h2.6}. Then, 
\begin{equation}\label{Barr-5}
\varepsilon \leq \mu_{S_F,j} \leq \mu_{\hatt S_K,j}, \quad j\in\bbN.
\end{equation} 
In particular, if the Friedrichs extension $S_F$ of $S$ has purely discrete
spectrum, then, except possibly for $\lambda=0$, the Krein--von Neumann extension
$S_K$ of $S$ also has purely discrete spectrum in $(0,\infty)$, that is, 
\begin{equation}
\sigma_{\rm ess}(S_F) = \emptyset \, \text{ implies } \, 
\sigma_{\rm ess}(S_K) \backslash\{0\} = \emptyset.      \lb{ESSK}
\end{equation}
In addition, let $p\in (0,\infty)\cup\{\infty\}$, then
\begin{align}
\begin{split}
& (S_F - z_0 I_{\cH})^{-1} \in \cB_p(\cH) 
\, \text{ for some $z_0\in \bbC\backslash [\varepsilon,\infty)$}   \\ 
& \quad \text{implies } \, 
(S_K - zI_{\cH})^{-1}[I_{\cH} - P_{\ker(S_K)}] \in \cB_p(\cH) 
 \, \text{ for all $z\in \bbC\backslash [\varepsilon,\infty)$}. 
\lb{CPK}
\end{split} 
\end{align}
In fact, the $\ell^p(\bbN)$-based trace ideals $\cB_p(\cH)$ of $\cB(\cH)$ can be replaced by any two-sided symmetrically normed ideals of $\cB(\cH)$.
\end{theorem}
%%%%%%%%%%%%%%%%%%%%%%%%%%%%%%%%%%%%%%

We note that \eqref{ESSK} is a classical result of Krein \cite{Kr47}, the more general fact \eqref{Barr-5} has not been mentioned explicitly in Krein's paper \cite{Kr47}, although it immediately follows from the minimax principle and Krein's formula \eqref{SKinv}. On the other hand, in the special case ${\rm def}(S)<\infty$, Krein states an extension of 
\eqref{Barr-5} in his Remark 8.1 in the sense that he also considers self-adjoint extensions different from the Krein extension. Apparently, \eqref{Barr-5} in the context of infinite deficiency indices has first been proven by Alonso and Simon \cite{AS80} by a somewhat different method. Relation \eqref{CPK} was 
recently proved in \cite{AGMST09} for $p \in (0, \infty)$.

Finally, we very briefly mention some new results on the Krein--von Neumann extension which were developed when working on this paper. These results exhibit the Krein--von Neumann extension as a natural object in elasticity theory by relating it to an abstract buckling problem as follows:

We start by introducing an abstract version of Proposition\ 1 in Grubb's paper  
\cite{Gr83} devoted to Krein--von Neumann extensions of even order elliptic differential operators on bounded smooth domains. We recall that Proposition\ 1 in 
\cite{Gr83} describes an  intimate connection between the nonzero eigenvalues of the Krein--von Neumann extension of an appropriate minimal elliptic differential operator of order $2m$, $m\in\bbN$, and nonzero eigenvalues of a suitable higher-order buckling problem (cf.\ \eqref{Df-H6}). The abstract version of this remarkable connection reads as follows:

%%%%%%%%%%%
\begin{lemma}  \lb{l3.3}
Assume Hypothesis \ref{h2.6} and let $\lambda \neq 0$. Then there exists 
$0 \neq v \in \dom(S_K)$ with
\begin{equation}
S_K v = \lambda v   \lb{3.1b}
\end{equation}
if and only if there exists $0 \neq u \in \dom(S^* S)$ such that
\begin{equation}
S^* S u = \lambda S u.   \lb{3.1c}
\end{equation}
In particular, the solutions $v$ of \eqref{3.1b} are in one-to-one correspondence with the solutions $u$ of \eqref{3.1c} given by the formulas
\begin{align}
u & = (S_F)^{-1} S_K v,    \lb{3.1d}  \\
v & = \lambda^{-1} S u.   \lb{3.1e}
\end{align}
Of course, since $S_K \geq 0$, any $\lambda \neq 0$ in \eqref{3.1b} and \eqref{3.1c}  necessarily satisfies $\lambda > 0$.
\end{lemma}
%%%%%%%%%%%

We refer to \cite{AGMST09} for the proof of Lemma \ref{l3.3}. Due to the generalized buckling problem \eqref{MM-1}, respectively, \eqref{MM-2}, we will call the linear pencil eigenvalue problem $S^* Su = \lambda S u$ in \eqref{3.1c} the {\it abstract buckling problem} associated with the Krein--von Neumann extension $S_K$ of $S$.

Next, we turn to a variational formulation of the correspondence between the inverse of the reduced Krein extension $\hatt S_K$ and the abstract buckling problem in terms of appropriate sesquilinear forms by following the treatment of Kozlov 
\cite{Ko79}--\cite{Ko84} in the context of elliptic partial differential operators. This will then lead to an even stronger connection between the Krein--von Neumann extension $S_K$ of $S$ and the associated abstract buckling eigenvalue problem \eqref{3.1c}, culminating in a unitary equivalence result in Theorem \ref{t3.3}. 

Given the operator $S$, we introduce the following sesquilinear forms in $\cH$,
\begin{align}
a(u,v) & = (Su,Sv)_{\cH}, \quad u, v \in \dom(a) = \dom(S),    \lb{3.2} \\
b(u,v) & = (u,Sv)_{\cH}, \quad u, v \in  \dom(b) = \dom(S).    \lb{3.3} 
\end{align}
Then $S$ being densely defined and closed implies that the sesquilinear form $a$ shares these properties and  \eqref{3.1} implies its boundedness from below,
\begin{equation}
a(u,u) \geq \varepsilon^2 \|u\|_{\cH}^2, \quad u \in \dom(S).    \lb{3.4}
\end{equation} 
Thus, one can introduce the Hilbert space 
$\cW=(\dom(S), (\dott,\dott)_{\cW})$ with associated scalar product 
\begin{equation}
(u,v)_{\cW}=a(u,v) = (Su,Sv)_{\cH}, \quad u, v \in \dom(S).   \lb{3.5}
\end{equation}
In addition, we denote by $\iota_{\cW}$ the continuous embedding operator of $\cW$ 
into $\cH$,
\begin{equation}
\iota_{\cW} : \cW \hookrightarrow \cH.    \lb{3.6}
\end{equation}
Hence, we will use the notation
\begin{equation}
(w_1,w_2)_{\cW} =a(\iota_{\cW} w_1,\iota_{\cW} w_2) 
= (S\iota_{\cW} w_1, S\iota_{\cW} w_2)_{\cH}, \quad w_1, w_2 \in \cW,   \lb{3.7}
\end{equation}
in the following. 

Given the sesquilinear forms $a$ and $b$ and the Hilbert space $\cW$, we next define the operator $T$ in $\cW$ by
\begin{align}
\begin{split}
(w_1,T w_2)_{\cW} & = a(\iota_{\cW} w_1,\iota_{\cW} T w_2) 
= (S \iota_{\cW} w_1,S\iota_{\cW} T w_2)_{\cH}   \\
& = b(\iota_{\cW} w_1,\iota_{\cW} w_2) = (\iota_{\cW} w_1,S \iota_{\cW} w_2)_{\cH},
\quad w_1, w_2 \in \cW.    \lb{3.8}
\end{split}
\end{align}
One verifies that $T$ is well-defined and that 
\begin{equation}
|(w_1,T w_2)_{\cW}| \leq \|\iota_{\cW} w_1\|_{\cH} \|S \iota_{\cW} w_2\|_{\cH} 
\leq \varepsilon^{-1} \|w_1\|_{\cW} \|w_2\|_{\cW}, \quad w_1, w_2 \in \cW,   \lb{3.9}
\end{equation}
and hence that
\begin{equation}
0 \leq T = T^* \in \cB(\cW), \quad \|T\|_{\cB(\cW)} \leq \varepsilon^{-1}.  \lb{3.10}
\end{equation}
For reasons to become clear in connection with \eqref{3.20a}--\eqref{3.39}, we called 
$T$ the {\it abstract buckling problem operator} associated with the Krein--von Neumann extension $S_K$ of $S$ in \cite{AGMST09}.  

Next, recalling the notation 
$\hatt \cH = [\ker (S^*)]^{\bot} = \big[I_{\cH} - P_{\ker(S^*)}\big] \cH$ (cf.\ \eqref{hattH}), 
we introduce the operator
\begin{equation}
\hatt S: \begin{cases} \cW \to \hatt \cH,  \\
w \mapsto S \iota_{\cW} w,  \end{cases}      \lb{3.12}
\end{equation}
and note that 
\begin{equation} 
\ran\big(\hatt S\big) = \ran (S) = \hatt \cH,    \lb{3.12aa}
\end{equation}
since $S\geq \varepsilon I_{\cH}$ for some $\varepsilon > 0$ and $S$ is closed in $\cH$ 
(see, e.g., \cite[Theorem\ 5.32]{We80}). In fact, 
\begin{equation}
\hatt S\in\cB(\cW,\hatt \cH) \, \text{ maps $\cW$ unitarily onto $\hatt \cH$.} 
\end{equation}

Continuing, we briefly recall the polar decomposition of $S$, 
\begin{equation}
S = U_S |S|,   \lb{3.19a}
\end{equation} 
with
\begin{equation}
|S| = (S^* S)^{1/2} \geq \varepsilon I_{\cH}, \; \varepsilon > 0, \quad 
U_S \in \cB\big(\cH,\hatt \cH\big) \, \text{ unitary,}    \lb{3.19b}
\end{equation}
and state the principal unitary equivalence result proven in \cite{AGMST09}:

%%%%%%%%%%%%%
\begin{theorem} \lb{t3.3}
Assume Hypothesis \ref{h2.6}. Then the inverse of the reduced Krein--von Neumann extension $\hatt S_K$ in $\hatt \cH = \big[I_{\cH} - P_{\ker(S^*)}\big] \cH$ and the abstract buckling problem operator $T$ in $\cW$ are unitarily equivalent, in particular,
\begin{equation}
\big(\hatt S_K\big)^{-1} = \hatt S T (\hatt S)^{-1}.    \lb{3.20}
\end{equation}
Moreover, one has
\begin{equation}
\big(\hatt S_K\big)^{-1} = U_S \big[|S|^{-1} S |S|^{-1}\big] (U_S)^{-1},    \lb{3.20a}
\end{equation}
where $U_S\in \cB\big(\cH,\hatt \cH\big)$ is the unitary operator in the polar decomposition \eqref{3.19a} of $S$ and the operator $|S|^{-1} S |S|^{-1}\in\cB(\cH)$ is self-adjoint in $\cH$. 
\end{theorem}
%%%%%%%%%%%%%

Equation \eqref{3.20a} is of course motivated by rewriting the abstract linear pencil buckling eigenvalue problem \eqref{3.1c}, $S^* S u = \lambda S u$, $\lambda \neq 0$, in the form
\begin{equation}
\lambda^{-1} S^* S u = \lambda^{-1} (S^* S)^{1/2} \big[(S^* S)^{1/2} u\big]  
= S (S^* S)^{-1/2} \big[(S^* S)^{1/2} u\big]    \lb{3.38}
\end{equation}
and hence in the form of a standard eigenvalue problem
\begin{equation}
|S|^{-1} S |S|^{-1} w = \lambda^{-1} w, \quad \lambda \neq 0, \quad w = |S| u.  \lb{3.39}
\end{equation}

Concluding this section, we point out that a great variety of additional results 
for the Krein--von Neumann extension can be found, for instance, in 
\cite[Sect.\ 109]{AG81a}, \cite{AS80}, \cite{AN70}, \cite{Ar98}, \cite{Ar00}, \cite{AHSD01}, 
\cite{AT02}, \cite{AT03}, \cite{AT05}, \cite{AT09}, \cite{AGMST09}, \cite{BC05}, 
\cite{DM91}, \cite{DM95}, \cite[Part III]{Fa75}, 
\cite[Sect.\ 3.3]{FOT94}, \cite{GM10}, \cite{Gr83}, \cite{Ha57}, 
\cite{HK09}, \cite{HMD04}, \cite{HSDW07}, \cite{KO77}, \cite{KO78}, \cite{Ne83}, 
\cite{PS96}, \cite{SS03}, \cite{Si98}, \cite{Sk79}, \cite{St96}, \cite{Ts80}, \cite{Ts81}, 
\cite{Ts92}, and the references therein. We also mention the references \cite{EM05}, 
\cite{EMP04}, \cite{EMMP07} (these authors, apparently unaware of the work of von 
Neumann, Krein, Vi{\u s}hik, Birman, Grubb, {\u S}trauss, etc., in this context, 
introduced the Krein Laplacian and called it the harmonic operator, see also 
\cite{Gr06}).

%%%%%%%%%%%%%%%%%%%%%%%%%%%%%%%%%%%%%%%%
%%%%%%%%%%%%%%%%%%%%%%%%%%%%%%%%%%%%%%%%
\section{Trace Theory in Lipschitz Domains}
\label{s3}
%%%%%%%%%%%%%%%%%%%%%%%%%%%%%%%%%%%%%%%%
%%%%%%%%%%%%%%%%%%%%%%%%%%%%%%%%%%%%%%%%

In this section we shall review material pertaining to analysis
in Lipschitz domains, starting with Dirichlet and Neumann boundary traces
in Subsection \ref{s3X}, and then continuing with a brief survey of
perturbed Dirichlet and Neumann Laplacians in Subsection \ref{s4X}.

%%%%%%%%%%%%%%%%%%%%%%%%%%%%%%%%%%%%%%%%
\subsection{Dirichlet and Neumann Traces in Lipschitz Domains}
\label{s3X}
%%%%%%%%%%%%%%%%%%%%%%%%%%%%%%%%%%%%%%%%

The goal of this subsection is to introduce the relevant material pertaining
to Sobolev spaces $H^s(\Omega)$ and $H^r(\partial\Omega)$ corresponding to subdomains $\Om$ of $\bbR^n$, $n\in\bbN$, and discuss various trace results. 

Before we focus primarily on bounded Lipschitz domains, we briefly recall some basic facts in connection with Sobolev spaces corresponding to open sets 
$\Om\subseteq\bbR^n$, $n\in\bbN$: For an arbitrary $m\in\bbN\cup\{0\}$, we follow the customary way of defining $L^2$-Sobolev spaces of order $\pm m$ in $\Om$ as
\begin{align}\label{hGi-1}
H^m(\Om) &:=\big\{u\in L^2(\Om;d^nx)\,\big|\,\partial^\alpha u\in L^2(\Om;d^nx), 
\, 0 \leq |\alpha|\leq m\big\}, \\
H^{-m}(\Om) &:=\biggl\{u\in\cD^{\prime}(\Om)\,\bigg|\,u=\sum_{0 \leq |\alpha|\leq m}
\partial^\alpha u_{\alpha}, \mbox{ with }u_\alpha\in L^2(\Om;d^nx), 
\,  0 \leq |\alpha|\leq m\biggr\},
\label{hGi-2}
\end{align}
equipped with natural norms (cf., e.g., \cite[Ch.\ 3]{AF03}, \cite[Ch.\ 1]{Ma85}). 
Here $\cD^\prime(\Om)$ denotes the usual set of 
distributions on $\Omega\subseteq \bbR^n$. Then we set
\begin{equation}\label{hGi-3}
H^m_0(\Om):=\,\mbox{the closure of $C^\infty_0(\Om)$ in $H^m(\Om)$}, 
\quad m \in \bbN\cup\{0\}.
\end{equation}
As is well-known, all three spaces above are Banach, reflexive and,
in addition,
\begin{equation} \label{hGi-4}
\bigl(H^m_0(\Om)\bigr)^*=H^{-m}(\Om).
\end{equation} 
Again, see, for instance, \cite[Ch.\ 3]{AF03}, \cite[Ch.\ 1]{Ma85}.

We recall that an open, nonempty set $\Omega\subseteq\bbR^n$ is
called a {\it Lipschitz domain} if the following property holds: 
There exists an open covering $\{{\mathcal O}_j\}_{1\leq j\leq N}$
of the boundary $\partial\Omega$ of $\Om$ such that for every
$j\in\{1,...,N\}$, ${\mathcal O}_j\cap\Omega$ coincides with the portion
of ${\mathcal O}_j$ lying in the over-graph of a Lipschitz function
$\varphi_j:\bbR^{n-1}\to\bbR$ $($considered in a new system of coordinates
obtained from the original one via a rigid motion$)$. The number
$\max\,\{\|\nabla\varphi_j\|_{L^\infty (\bbR^{n-1};d^{n-1}x')^{n-1}}\,|\,1\leq j\leq N\}$
is said to represent the {\it Lipschitz character} of $\Omega$.

The classical theorem of Rademacher on almost everywhere differentiability of Lipschitz
functions ensures that for any  Lipschitz domain $\Omega$, the
surface measure $d^{n-1} \omega$ is well-defined on  $\partial\Omega$ and
that there exists an outward  pointing normal vector $\nu$ at
almost every point of $\partial\Omega$. 

As regards $L^2$-based Sobolev spaces of fractional order $s\in\bbR$,
on arbitrary Lipschitz domains $\Om\subseteq\bbR^n$, we introduce
\begin{align}\label{HH-h1}
H^{s}(\bbR^n) &:=\bigg\{U\in \cS^\prime(\bbR^n)\,\bigg|\,
\norm{U}_{H^{s}(\bbR^n)}^2 = \int_{\bbR^n}d^n\xi\,
\big|\hatt U(\xi)\big|^2\big(1+\abs{\xi}^{2s}\big)<\infty \bigg\},
\\
H^{s}(\Om) &:=\big\{u\in \cD^\prime(\Om)\,\big|\,u=U|_\Om\text{ for some }
U\in H^{s}(\bbR^n)\big\} = R_{\Om} \, H^s(\bbR^n),
\label{HH-h2}
\end{align} 
where $R_{\Om}$ denotes the restriction operator (i.e., $R_{\Om} \, U=U|_{\Om}$, 
$U\in H^{s}(\bbR^n)$),   
$\cS^\prime(\bbR^n)$ is the space of tempered distributions on $\bbR^n$,
and $\hatt U$ denotes the Fourier transform of $U\in\cS^\prime(\bbR^n)$.
These definitions are consistent with \eqref{hGi-1}, \eqref{hGi-2}. 
Next, retaining that $\Om\subseteq \bbR^n$ is an arbitrary Lipschitz domain, we introduce
\begin{equation}\label{incl-xxx}
H^{s}_0(\Omega):=\big\{u\in H^{s}(\bbR^n)\,\big|\, {\rm supp} (u)\subseteq\ol{\Omega}\big\}, 
\quad s\in\bbR,
\end{equation} 
equipped with the natural norm induced by $H^{s}(\bbR^n)$. The space 
$H^{s}_0(\Omega)$ is reflexive, being a closed subspace of $H^{s}(\bbR^n)$. 
Finally, we introduce for all $s\in\bbR$,
\begin{align} 
\mathring{H}^{s} (\Omega) &= \mbox{the closure of $C^\infty_0(\Omega)$ in $H^s(\Omega)$},   \\
H^{s}_{z} (\Om) &= R_{\Om} \, H^{s}_0(\Omega).
\end{align}

Assuming from now on that $\Om\subset\bbR^n$ is a Lipschitz domain with a compact boundary, we recall the existence of a universal linear extension operator 
$E_{\Om}:\cD^\prime (\Om) \to \cS^\prime (\bbR^n)$ such that 
$E_{\Om}: H^s(\Om) \to H^s(\bbR^n)$ is bounded for all $s\in\bbR$, and 
$R_{\Om}  E_{\Om}=I_{H^s(\Om)}$ (cf.\ \cite{Ry99}). If $\widetilde{C_0^\infty(\Om)}$ denotes the set of $C_0^\infty(\Om)$-functions extended to all of $\bbR^n$ by setting functions zero outside of $\Omega$, then for all $s\in\bbR$, 
$\widetilde{C_0^\infty(\Om)} \hookrightarrow H^s_0(\Om)$ densely.  

Moreover, one has
\begin{equation}\label{incl-Ya}
\big(H^{s}_0(\Omega)\big)^*=H^{-s}(\Omega),  \quad s\in\bbR.
\end{equation}
(cf., e.g., \cite{JK95}) consistent with \eqref{hGi-3}, and also, 
\begin{equation}
\big(H^s(\Om)\big)^* = H^{-s}_0(\Om),  \quad s\in\bbR, 
\end{equation}
in particular, $H^s(\Om)$ is a reflexive Banach space. We shall also use the fact that  
for a Lipschitz domain $\Om\subset\bbR^n$ with compact boundary, the space 
$\mathring{H}^{s} (\Omega)$ 
satisfies 
\begin{equation}\label{incl-Yb}
\mathring{H}^{s} (\Omega) = H^s_{z}(\Om)
\, \mbox{ if } \, s > -1/2,\,\,s\notin{\textstyle\big\{{\frac12}}+\bbN_0\big\}. 
\end{equation}
For a Lipschitz domain $\Omega\subseteq\bbR^n$ with compact boundary it is also known that
\begin{equation}\label{dual-xxx}
\bigl(H^{s}(\Omega)\bigr)^*=H^{-s}(\Omega), \quad - 1/2 <s< 1/2.
\end{equation}
See \cite{Tr02} for this and other related properties. Throughout this paper,
we agree to use the {\it adjoint} (rather than the dual) space $X^*$ of a Banach space 
$X$.

From this point on we will always make the following assumption (unless explicitly stated otherwise):

%%%%%%%%%%%%%%%%%%%%%%%%%%%%%%%%%%%%%%%
\begin{hypothesis}\label{h2.1}
Let $n\in\bbN$, $n\geq 2$, and assume that $\emptyset \neq \Om\subset{\bbR}^n$ is
a bounded Lipschitz domain.
\end{hypothesis}
%%%%%%%%%%%%%%%%%%%%%%%%%%%%%%%%%%%%%%% 

To discuss Sobolev spaces on the boundary of a Lipschitz domains, consider
first the case where $\Omega\subset\bbR^n$ is the domain lying above the graph
of a Lipschitz function $\varphi\colon\bbR^{n-1}\to\bbR$. In this setting,
we define the Sobolev space $H^s(\partial\Omega)$ for $0\leq s\leq 1$,
as the space of functions $f\in L^2(\partial\Omega;d^{n-1}\omega)$ with the
property that $f(x',\varphi(x'))$, as a function of $x'\in\bbR^{n-1}$,
belongs to $H^s(\bbR^{n-1})$. This definition is easily adapted to the case
when $\Omega$ is a Lipschitz domain whose boundary is compact,
by using a smooth partition of unity. Finally, for $-1\leq s\leq 0$, we set
\begin{equation}\label{A.6}
H^s(\dOm) = \big(H^{-s}(\dOm)\big)^*, \quad -1 \le s \le 0.
\end{equation}
 From the above characterization of $H^s(\partial\Omega)$ it follows that
any property of Sobolev spaces (of order $s\in[-1,1]$) defined in Euclidean
domains, which are invariant under multiplication by smooth, compactly
supported functions as well as composition by bi-Lipschitz diffeomorphisms,
readily extends to the setting of $H^s(\partial\Omega)$ (via localization and
pullback). For additional background
information in this context we refer, for instance, to
\cite[Chs.\ V, VI]{EE89}, \cite[Ch.\ 1]{Gr85}.

Assuming Hypothesis \ref{h2.1}, we introduce the boundary trace
operator $\ga_D^0$ (the Dirichlet trace) by
\begin{equation}
\ga_D^0\colon C(\ol{\Om})\to C(\dOm), \quad \ga_D^0 u = u|_\dOm.   \label{2.5}
\end{equation}
Then there exists a bounded, linear operator $\gamma_D$
\begin{align}
\begin{split}
& \ga_D\colon H^{s}(\Om)\to H^{s-(1/2)}(\dOm) \hookrightarrow \LdOm,
\quad 1/2<s<3/2, \label{2.6}  \\
& \ga_D\colon H^{3/2}(\Om)\to H^{1-\varepsilon}(\dOm) \hookrightarrow \LdOm,
\quad \varepsilon \in (0,1) 
\end{split}
\end{align}
(cf., e.g., \cite[Theorem 3.38]{Mc00}), whose action is compatible with that of 
$\ga_D^0$. That is, the two Dirichlet trace operators coincide on the intersection 
of their domains. Moreover, we recall that
\begin{equation}\label{2.6a}
\ga_D\colon H^{s}(\Om)\to H^{s-(1/2)}(\dOm) \, \text{ is onto for $1/2<s<3/2$}.
\end{equation}

Next, retaining Hypothesis \ref{h2.1}, we introduce the operator
$\ga_N$ (the strong Neumann trace) by
\begin{equation} \label{2.7}
\ga_N = \nu\cdot\ga_D\nabla \colon H^{s+1}(\Om)\to \LdOm, \quad 1/2<s<3/2,
\end{equation} 
where $\nu$ denotes the outward pointing normal unit vector to
$\partial\Om$. It follows from \eqref{2.6} that $\ga_N$ is also a
bounded operator. We seek to extend the action of the Neumann trace
operator \eqref{2.7} to other (related) settings. To set the stage,
assume Hypothesis \ref{h2.1} and observe that the inclusion
\begin{equation}\label{inc-1}
\iota:H^{s_0}(\Omega)\hookrightarrow \bigl(H^r(\Omega)\bigr)^*, \quad
s_0>-1/2,\; r>1/2,
\end{equation}
is well-defined and bounded. We then introduce the weak Neumann trace
operator
\begin{equation}\label{2.8}
\wti\ga_N\colon\big\{u\in H^{s+1/2}(\Om)\,\big|\,\Delta u\in H^{s_0}(\Om)\big\}
\to H^{s-1}(\dOm),\quad s\in(0,1),\; s_0>-1/2,
\end{equation}
as follows: Given $u\in H^{s+1/2}(\Om)$ with $\Delta u \in H^{s_0}(\Om)$
for some $s\in(0,1)$ and $s_0>-1/2$, we set (with $\iota$ as in
\eqref{inc-1} for $r:=3/2-s>1/2$)
\begin{equation} \label{2.9}
\langle\phi,\wti\ga_N u \rangle_{1-s}
={}_{H^{1/2-s}(\Om)}\langle\nabla\Phi,\nabla u\rangle_{(H^{1/2-s}(\Om))^*}
+ {}_{H^{3/2-s}(\Om)}\langle\Phi,\iota(\Delta u)\rangle_{(H^{3/2-s}(\Om))^*},
\end{equation} 
for all $\phi\in H^{1-s}(\dOm)$ and $\Phi\in H^{3/2-s}(\Om)$ such that
$\ga_D\Phi=\phi$. We note that the first pairing in the right-hand side
above is meaningful since
\begin{equation} \label{2.9JJ}
\bigl(H^{1/2-s}(\Om)\bigr)^*=H^{s-1/2}(\Om),\quad s\in (0,1),
\end{equation}  
that the definition \eqref{2.9} is independent of the particular
extension $\Phi$ of $\phi$, and that $\wti\ga_N$ is a bounded extension
of the Neumann trace operator $\ga_N$ defined in \eqref{2.7}.

For further reference, let us also point out here that if $\Omega\subset\bbR^n$
is a bounded Lipschitz domain then for any $j,k\in\{1,...,n\}$ the
(tangential first-order differential) operator
\begin{equation}\label{Pf-2}
\partial/\partial\tau_{j,k}:=\nu_j\partial_k-\nu_k\partial_j:
H^s(\partial\Omega)\to H^{s-1}(\partial\Omega),\quad 0\leq s\leq 1,
\end{equation}
is well-defined, linear and bounded. Assuming Hypothesis \ref{h2.1},
we can then define the tangential gradient operator
\begin{equation} \label{Tan-C1}
\nabla_{tan}: \begin{cases}H^1(\partial\Omega)\to 
\big(L^2(\partial\Omega;d^{n-1}\omega)\big)^n   \\
\hspace*{1cm} f \mapsto 
\nabla_{tan}f:=\Big(\sum_{k=1}^n\nu_k\frac{\partial f}{\partial\tau_{kj}}
\Big)_{1\leq j\leq n} \end{cases} 
,\quad f\in H^1(\partial\Omega).
\end{equation} 
The following result has been proved in \cite{MMS05}.

%%%%%%%%%%%%%%%%%%%%%%%%%%%%%%%%%%%%%%%%%%%%%
\begin{theorem}\label{T-MMS}
Assume Hypothesis \ref{h2.1} and denote by $\nu$ the outward unit normal
to $\partial\Omega$. Then the operator
\begin{equation}\label{Tan-C2} 
\gamma_2: \begin{cases} H^2(\Omega)\to \bigl\{(g_0,g_1)\in H^1(\partial\Omega) \times L^2(\partial\Omega;d^{n-1}\omega)\,\big|\,   \nabla_{tan}g_0  % \\ 
% \hspace*{7.45cm}  
+g_1\nu\in \bigl(H^{1/2}(\partial\Omega)\bigr)^n\bigl\}  \\
\hspace*{8mm} 
u \mapsto \gamma_2 u=(\gamma_D u\,,\,\gamma_N u), 
\end{cases}
\end{equation}
is well-defined, linear, bounded, onto, and has a linear, bounded
right-inverse. The space $\bigl\{(g_0,g_1)\in H^1(\partial\Omega)
\times L^2(\partial\Omega;d^{n-1}\omega)\,\big|\,
\nabla_{tan}g_0+g_1\nu\in \bigl(H^{1/2}(\partial\Omega)\bigr)^n\bigl\}$ in \eqref{Tan-C2} 
is equipped with the natural norm
\begin{equation} \label{NoRw-1}
(g_0,g_1)\mapsto \|g_0\|_{H^1(\partial\Omega)}
+\|g_1\|_{L^2(\partial\Omega;d^{n-1}\omega)}
+\|\nabla_{tan}g_0+g_1\nu\|_{(H^{1/2}(\partial\Omega))^n}.
\end{equation} 
Furthermore, the null space of the operator \eqref{Tan-C2}
is given by
\begin{equation} \label{Tan-C3}
\ker(\gamma_2):= \big\{u\in H^2(\Omega)\,\big|\,\gamma_D u =\gamma_N u=0\big\}
=H^2_0(\Omega),
\end{equation}
with the latter space denoting the closure of $C^\infty_0(\Omega)$
in $H^2(\Omega)$.
\end{theorem}
%%%%%%%%%%%%%%%%%%%%%%%%%%%%%%%%%%%%%

Continuing to assume Hypothesis \ref{h2.1}, we now introduce
\begin{equation} \label{Tan-C4} 
N^{1/2}(\partial\Omega):=\big\{g\in L^2(\partial\Omega;d^{n-1}\omega)\,\big|\,
g\nu_j\in H^{1/2}(\partial\Omega),\,\,1\leq j\leq n\big\},
\end{equation} 
where the $\nu_j$'s are the components of $\nu$.  We equip this space with 
the natural norm
\begin{equation} \label{Tan-C4B}
\|g\|_{N^{1/2}(\partial\Omega)}
:=\sum_{j=1}^n\|g\nu_j\|_{H^{1/2}(\partial\Omega)}.
\end{equation} 

Then $N^{1/2}(\partial\Omega)$ is a reflexive Banach space which embeds
continuously into $L^2(\partial\Omega;d^{n-1}\omega)$. Furthermore,
\begin{equation} \label{Tan-C5}
N^{1/2}(\partial\Omega)=H^{1/2}(\partial\Omega)\, 
\mbox{ whenever $\Omega$ is a bounded $C^{1,r}$ domain with $r>1/2$}.
\end{equation} 

It should be mentioned that the spaces
$H^{1/2}(\partial\Omega)$ and $N^{1/2}(\partial\Omega)$ can be quite
different for an arbitrary Lipschitz domain $\Omega$.
Our interest in the latter space stems from the fact that this
arises naturally when considering the Neumann trace operator acting on 
\begin{equation} \label{Tan-C6}
\big\{u\in H^2(\Omega)\,\big|\,\gamma_D u =0\big\}=H^2(\Omega)\cap H^1_0(\Omega),
\end{equation} 
considered as a closed subspace of $H^2(\Omega)$ (hence, a Banach space
when equipped with the $H^2$-norm). More specifically, we have
(cf.\ \cite{GM10} for a proof):

%%%%%%%%%%%%%%%%%%%%%%%%%%%%%%%%%%%%%%%%
\begin{lemma}\label{Lo-Tx}
Assume Hypothesis \ref{h2.1}. Then the Neumann trace operator $\gamma_N$
considered in the context
\begin{equation} \label{Tan-C7}
\gamma_N:H^2(\Omega)\cap H^1_0(\Omega)\to N^{1/2}(\partial\Omega)
\end{equation} 
is well-defined, linear, bounded, onto and with a linear, bounded 
right-inverse. In addition, the null space of $\gamma_N$ in \eqref{Tan-C7} is
precisely $H^2_0(\Omega)$, the closure of $C^\infty_0(\Omega)$ in
$H^2(\Omega)$.
\end{lemma}
%%%%%%%%%%%%%%%%%%%%%%%%%%%%%%%%%%%%%%%%

Most importantly for us here is the fact that one can use the above Neumann
trace result in order to extend the action of the Dirichlet trace operator
\eqref{2.6} to $\dom(-\Delta_{max,\Om})$, the domain of the maximal
Laplacian, that is, $\{u\in L^2(\Omega;d^nx)\,|\,\Delta u\in L^2(\Omega;d^nx)\}$,
which we consider equipped with the graph norm
$u\mapsto \|u\|_{L^2(\Omega;d^nx)}+\|\Delta u\|_{L^2(\Omega;d^nx)}$.
Specifically, with $\bigl(N^{1/2}(\partial\Omega)\bigr)^*$ denoting the
conjugate dual space of $N^{1/2}(\partial\Omega)$, we have the
following result from \cite{GM10}:

%%%%%%%%%%%%%%%%%%%%%%%%%
\begin{theorem}\label{New-T-tr}
Assume Hypothesis \ref{h2.1}. Then there exists a unique linear, bounded
operator
\begin{equation} \label{Tan-C10}
\widehat{\gamma}_D:\big\{u\in L^2(\Omega;d^nx)\,\big|\,\Delta u\in L^2(\Omega;d^nx)\big\}
\to \bigl(N^{1/2}(\partial\Omega)\bigr)^*
\end{equation} 
which is compatible with the Dirichlet trace introduced in \eqref{2.6},
in the sense that, for each $s>1/2$, one has
\begin{equation} \label{Tan-C11}
\widehat{\gamma}_D u =\gamma_D u \, \mbox{ for every $u\in H^s(\Omega)$
with $\Delta u\in L^2(\Omega;d^nx)$}.
\end{equation} 
Furthermore, this extension of the Dirichlet trace operator in \eqref{2.6}
allows for the following generalized integration by parts formula
\begin{equation} \label{Tan-C12}
{}_{N^{1/2}(\partial\Omega)}\langle\gamma_N w,\widehat{\gamma}_D u 
\rangle_{(N^{1/2}(\partial\Omega))^*}
=(\Delta w,u)_{L^2(\Om;d^nx)}
- (w,\Delta u)_{L^2(\Om;d^nx)},
\end{equation} 
valid for every $u\in L^2(\Omega;d^nx)$ with $\Delta u\in L^2(\Omega;d^nx)$
and every $w\in H^2(\Omega)\cap H^1_0(\Omega)$.
\end{theorem}
%%%%%%%%%%%%%%%%%%%%%%%%%

We next review the case of the Neumann trace, whose action
is extended to $\dom(-\Delta_{max,\Om})$. To this end, we need
to address a number of preliminary matters. First, assuming
Hypothesis \ref{h2.1}, we make the following definition
(compare with \eqref{Tan-C4}):
\begin{equation} \label{3an-C4} 
N^{3/2}(\partial\Omega):=\bigl\{g\in H^1(\partial\Omega)\,\big|\,
\nabla_{tan}g\in \bigl(H^{1/2}(\partial\Omega)\bigr)^n\bigl\},
\end{equation} 
equipped with the natural norm
\begin{equation} \label{3an-C4B}
\|g\|_{N^{3/2}(\partial\Omega)}
:=\|g\|_{L^2(\partial\Omega;d^{n-1}\omega)}+
\|\nabla_{tan}g\|_{(H^{1/2}(\partial\Omega))^n}.
\end{equation} 
Assuming Hypothesis \ref{h2.1}, $N^{3/2}(\partial\Omega)$ is a
reflexive Banach space which embeds continuously into the space 
$H^1(\partial\Omega;d^{n-1}\omega)$. In addition, this turns out
to be a natural substitute for the more familiar space
$H^{3/2}(\partial\Omega)$ in the case where $\Omega$
is sufficiently smooth. Concretely, one has 
\begin{equation} \label{3an-C5}
N^{3/2}(\partial\Omega)=H^{3/2}(\partial\Omega),
\end{equation} 
(as vector spaces with equivalent norms),
whenever $\Omega$ is a bounded $C^{1,r}$ domain with $r>1/2$.
The primary reason we are interested in $N^{3/2}(\partial\Omega)$ is that
this space arises naturally when considering the Dirichlet trace operator
acting on 
\begin{equation} \label{3an-C6N}
\big\{u\in H^2(\Omega)\,\big|\,\gamma_N u =0\big\},
\end{equation} 
considered as a closed subspace of $H^2(\Omega)$ (thus, a Banach space
when equipped with the norm inherited from $H^2(\Omega)$).
Concretely, the following result has been established in \cite{GM10}.

%%%%%%%%%%%%%%%%%%%%%%%%%%%%%%%%%%%%%%%%
\begin{lemma}\label{3o-TxD}
Assume Hypothesis \ref{h2.1}. Then the Dirichlet trace operator $\gamma_D$
considered in the context
\begin{equation} \label{3an-C7D}
\gamma_D:\big\{u\in H^2(\Omega)\,\big|\,\gamma_N u =0\big\}
\to N^{3/2}(\partial\Omega)
\end{equation} 
is well-defined, linear, bounded, onto and with a linear, bounded 
right-inverse. In addition, the null space of $\gamma_D$ in \eqref{3an-C7D} is
precisely $H^2_0(\Omega)$, the closure of $C^\infty_0(\Omega)$ in
$H^2(\Omega)$.
\end{lemma}
%%%%%%%%%%%%%%%%%%%%%%%%%%%%%%%%%%%%%%%%

It is then possible to use the Neumann trace result from Lemma \ref{3o-TxD}
in order to extend the action of the Neumann trace operator \eqref{2.7}
to $\dom(-\Delta_{max,\Om})=\big\{u\in L^2(\Omega;d^nx)\,\big|\,
\Delta u\in L^2(\Omega;d^nx)\big\}$.
As before, this space is equipped with the natural graph norm.
Let $\bigl(N^{3/2}(\partial\Omega)\bigr)^*$ denote the
conjugate dual space of $N^{3/2}(\partial\Omega)$. The following result holds:

%%%%%%%%%%%%%%%%%%%%%%%%%
\begin{theorem}\label{3ew-T-tr}
Assume Hypothesis \ref{h2.1}. Then there exists a unique linear, bounded
operator
\begin{equation} \label{3an-C10}
\widehat{\gamma}_N:\big\{u\in L^2(\Omega;d^nx)\,\big|\,\Delta u\in L^2(\Omega;d^nx)\big\}
\to \bigl(N^{3/2}(\partial\Omega)\bigr)^*
\end{equation} 
which is compatible with the Neumann trace introduced in \eqref{2.7},
in the sense that, for each $s>3/2$, one has
\begin{equation} \label{3an-C11}
\widehat{\gamma}_N u =\gamma_N u \, \mbox{ for every $u\in H^s(\Omega)$
with $\Delta u\in L^2(\Omega;d^nx)$}.
\end{equation} 
Furthermore, this extension of the Neumann trace operator from \eqref{2.7}
allows for the following generalized integration by parts formula
\begin{equation} \label{3an-C12}
{}_{N^{3/2}(\partial\Omega)}\langle\gamma_D w,\widehat\gamma_N u 
\rangle_{(N^{3/2}(\partial\Omega))^*}
= ( w,\Delta u)_{L^2(\Om;d^nx)}
- (\Delta w,u)_{L^2(\Om;d^nx)},
\end{equation} 
valid for every $u\in L^2(\Omega;d^nx)$ with $\Delta u\in L^2(\Omega;d^nx)$
and every $w\in H^2(\Omega)$ with $\gamma_N w =0$.
\end{theorem}
%%%%%%%%%%%%%%%%%%%%%%%%%

A proof of Theorem \ref{3ew-T-tr} can be found in \cite{GM10}.

%%%%%%%%%%%%%%%%%%%%%%%%%%%%%%%%%%%%%%%%
\subsection{Perturbed Dirichlet and Neumann Laplacians}
\label{s4X}
%%%%%%%%%%%%%%%%%%%%%%%%%%%%%%%%%%%%%%%%

Here we shall discuss operators of the form $-\Delta+V$ equipped with
Dirichlet and Neumann boundary conditions. Temporarily, we will employ
the following assumptions:

%%%%%%%%%%%%%%%%%%%%%%%%%%%%%%%%%%%%%%%
\begin{hypothesis}\label{h.V}
Let $n\in\bbN$, $n\geq 2$, assume that $\Om\subset{\bbR}^n$ is
an open, bounded, nonempty set, and suppose that
\begin{equation} \label{VV-W}
V\in L^\infty(\Om;d^nx)
\, \mbox{ and }\, V \,\mbox{is real-valued a.e.\ on } \,\Omega.
\end{equation} 
\end{hypothesis}
%%%%%%%%%%%%%%%%%%%%%%%%%%%%%%%%%%%%%%%

We start by reviewing the perturbed Dirichlet and Neumann Laplacians $H_{D,\Om}$ 
and $H_{N,\Om}$ associated with an open set $\Om$ in $\bbR^n$ and a potential $V$ satisfying Hypothesis \ref{h.V}: Consider the sesquilinear forms in $L^2(\Om;d^n x)$, 
\begin{equation}
Q_{D,\Om} (u,v) = (\nabla u, \nabla v) + (u,Vv), \quad u,v \in \dom (Q_{D,\Om}) 
= H^1_0(\Om),   \lb{3.QD}
\end{equation}
and 
\begin{equation}
Q_{N,\Om} (u,v) = (\nabla u, \nabla v) + (u,Vv), \quad u,v \in \dom (Q_{N,\Om}) 
= H^1(\Om).    \lb{3.QN}
\end{equation}
Then both forms in \eqref{3.QD} and \eqref{3.QN} are densely, defined, closed, and bounded from below in $L^2(\Om;d^n x)$. Thus, by the first and second representation theorems for forms (cf., e.g., \cite[Sect.\ VI.2]{Ka80}), one concludes that there exist unique self-adjoint operators $H_{D,\Om}$ and $H_{N,\Om}$ in $L^2(\Om;d^n x)$, both bounded from below, associated with the forms $Q_{D,\Om}$ and $Q_{N,\Om}$, respectively, which satisfy
\begin{align}
& Q_{D,\Om} (u,v) = (u,H_{D,\Om} v), \quad u \in \dom (Q_{D,\Om}), \, 
v \in \dom(H_{D,\Om}),  \lb{3.QHD} \\
& \dom(H_{D,\Om}) \subset \dom\big(|H_{D,\Om}|^{1/2}\big) = \dom (Q_{D,\Om}) 
= H^1_0(\Om)   \lb{3.HD}
\end{align}
and 
\begin{align}
& Q_{N,\Om} (u,v) = (u,H_{N,\Om} v), \quad u \in \dom (Q_{N,\Om}), \, 
v \in \dom(H_{N,\Om}),  \lb{3.QHN} \\
& \dom(H_{N,\Om}) \subset \dom\big(|H_{N,\Om}|^{1/2}\big) = \dom (Q_{N,\Om}) 
= H^1(\Om).   \lb{3.HN}
\end{align}
In the case of the perturbed Dirichlet Laplacian, $H_{D,\Om}$, one actually can say a bit more: Indeed, $H_{D,\Om}$ coincides with the Friedrichs extension of the operator 
\begin{equation}
H_{c,\Om} u = (-\Delta + V) u, 
\quad u \in \dom(H_{c,\Om}):=C^\infty_0(\Omega)
\end{equation}
in $L^2(\Om;d^nx)$, 
\begin{equation}
(H_{c,\Om})_F = H_{D,\Om},   \lb{3.cFD}
\end{equation}
and one obtains as an immediate consequence of \eqref{Fr-Q} and \eqref{3.QD}
\begin{equation}
H_{D,\Om}u= (-\Delta+V)u,  \quad 
u\in \dom(H_{D,\Om}) 
= \big\{v\in H_0^1(\Om)\,\big|\,\Delta v\in L^2(\Om;d^n x)\big\}.     \lb{3.HDF}
\end{equation}
We also refer to \cite[Sect.\ IV.2, Theorem VII.1.4]{EE89}). In addition, $H_{D,\Om}$ is known to have a compact resolvent and hence purely discrete spectrum bounded from below.

In the case of the perturbed Neumann Laplacian, $H_{N,\Om}$, it is not possible to be more specific under this general hypothesis on $\Om$ just being open. However, under the additional assumptions on the domain $\Om$ in Hypothesis \ref{h2.1} one can be more explicit about the domain of $H_{N,\Om}$ and also characterize its spectrum as follows. In addition, we also record an improvement of \eqref{3.HDF} under the additional Lipschitz hypothesis on $\Om$:

%%%%%%%%%%%%%%%%%%%%
\begin{theorem} \label{t2.5} \hspace*{-1mm} 
Assume Hypotheses \ref{h2.1} and \ref{h.V}.\ 
Then the perturbed Dirichlet Laplacian, $H_{D,\Om}$, given by
\begin{align} 
& H_{D,\Om}u= (-\Delta+V)u,\no \\
& u\in \dom(H_{D,\Om}) =
\big\{v\in H^1(\Om)\,\big|\,\Delta v\in L^2(\Om;d^n x),\, 
\gamma_D v=0\text{ in $H^{1/2}(\dOm)$}\big\}   \label{2.39} \\
& \hspace*{2.46cm} =\big\{v\in H_0^1(\Om)\,\big|\,\Delta v\in L^2(\Om;d^n x)\big\}, \no 
\end{align}
is self-adjoint and bounded from below in $L^2(\Om;d^nx)$. Moreover,
\begin{equation}\label{2.40}
\dom\big(|H_{D,\Om}|^{1/2}\big)=H^1_0(\Om), 
\end{equation}
and the spectrum of $H_{D,\Om}$, is purely discrete $($i.e., it consists of eigenvalues of finite multiplicity$)$, 
\begin{equation}
\sigma_{\rm ess}(H_{D,\Om}) = \emptyset. 
\end{equation}
If, in addition, $V\geq 0$ a.e.\ in $\Omega$, then $H_{D,\Om}$ is strictly positive in 
$L^2(\Om;d^nx)$. 
\end{theorem}
%%%%%%%%%%%%%%%%%%%%

The corresponding result for the perturbed Neumann Laplacian $H_{N,\Om}$
reads as follows:

%%%%%%%%%%%%%%%%%%%%
\begin{theorem}\label{t2.3}
Assume Hypotheses \ref{h2.1} and \ref{h.V}.
Then the perturbed Neumann Laplacian, $H_{N,\Om}$, given by
\begin{align}\label{2.20}
& H_{N,\Om}u = (-\Delta+V)u,  \\
& u \in \dom(H_{N,\Om}) =
\big\{v\in H^1(\Om)\,\big|\,\Delta v\in L^2(\Om;d^nx),\, 
\wti\gamma_N v =0\text{ in }H^{-1/2}(\dOm)\big\},  \no 
\end{align}
is self-adjoint and bounded from below in $L^2(\Om;d^nx)$. Moreover,
\begin{equation}\label{2.40a}
\dom\big(|H_{N,\Om}|^{1/2}\big)=H^1(\Om), 
\end{equation}
and the spectrum of $H_{N,\Om}$, is purely discrete $($i.e., it consists of eigenvalues of finite multiplicity$)$, 
\begin{equation}
\sigma_{\rm ess}(H_{N,\Om}) = \emptyset. 
\end{equation}
If, in addition, $V\geq 0$ a.e.\ in $\Omega$, then $H_{N,\Om}$ is nonnegative in 
$L^2(\Om;d^nx)$. 
\end{theorem}
%%%%%%%%%%%%%%%%%%%%

In the sequel, corresponding to the case where $V\equiv 0$, we shall abbreviate
\begin{equation} \label{V-Df1}
-\Delta_{D,\Om}\, \mbox{ and }\, -\Delta_{N,\Om},
\end{equation} 
for $H_{D,\Om}$ and $H_{N,\Om}$, respectively, and simply refer to these operators as, the Dirichlet and Neumann Laplacians. The above results have been proved in 
\cite[App.\ A]{GLMZ05}, \cite{GMZ07} for considerably more general potentials than assumed in Hypothesis \ref{h.V}.

Next, we shall now consider the minimal and maximal perturbed Laplacians.
Concretely, given an open set $\Omega\subset{\mathbb{R}}^n$
and a potential $0\leq V\in L^\infty(\Om;d^nx)$,
we introduce the maximal perturbed Laplacian in $L^2(\Omega;d^nx)$
\begin{align}\label{Yan-1}
\begin{split}
& H_{max,\Om} u:=(-\Delta+V)u,  \\
& u\in \dom(H_{max,\Om} ):=\big\{v\in L^2(\Omega;d^nx)\,\big|\,
\Delta v \in L^2(\Omega;d^nx)\big\}.
\end{split}
\end{align}

We pause for a moment to dwell on the notation used in connection with the symbol 
$\Delta$:

%%%%%%%%%%%%%%%%%%%%%%%%%%%%%%%%%%%% 
\begin{remark}
Throughout this manuscript the symbol $\Delta$ 
alone indicates that the Laplacian acts in the sense of distributions, 
\begin{equation}
\Delta\colon \cD^\prime(\Om) \to \cD^\prime(\Om).    \lb{DELTA} 
\end{equation}
In some cases, when it is necessary to interpret $\Delta$ as a bounded operator 
acting between Sobolev spaces, we write 
$\Delta \in \cB\big(H^s(\Om),H^{s-2}(\Om)\big)$ for various ranges of $s\in\bbR$ 
(which is of course compatible with \eqref{DELTA}). 
In addition, as a consequence of standard interior elliptic regularity (cf.\ Weyl's classical lemma) it is not difficult to see that if $\Om\subseteq\bbR$ is open, $u \in \cD'(\Om)$ and 
$\Delta u \in L^2_{\rm loc}(\Om; d^nx)$ then actually $u \in H^2_{\rm loc}(\Om)$. In particular, this comment applies to $u\in \dom(H_{max,\Om} )$ in \eqref{Yan-1}.
\end{remark}
%%%%%%%%%%%%%%%%%%%%%%%%%%%%%%%%%%%%  

In the remainder of this subsection we shall collect a number of
results, originally proved in \cite{GM10} when $V\equiv 0$,
but which are easily seen to hold in the more general setting considered here.

%%%%%%%%%%%%%%%%%%%%%%%%%%%%%%%%%%%%%%
\begin{lemma}\label{Max-M1}
Assume Hypotheses \ref{h2.1} and \ref{h.V}.
Then the maximal perturbed Laplacian associated with
$\Omega$ and the potential $V$ is a closed, densely defined operator for which
\begin{equation} \label{Yan-2}
H^2_0(\Omega)\subseteq \dom((H_{max,\Om})^*)
\subseteq \big\{u\in L^2(\Omega;d^nx)\,\big|\,
\Delta u\in L^2(\Omega;d^nx), \, 
\widehat{\gamma}_D u=\widehat{\gamma}_N u =0\big\}.
\end{equation} 
\end{lemma}
%%%%%%%%%%%%%%%%%%%%%%%%%%%%%%%%%%%%%

For an open set $\Omega\subset{\mathbb{R}}^n$ and a potential
$0\leq V\in L^\infty(\Om;d^nx)$, we also bring in the minimal
perturbed Laplacian in $L^2(\Omega;d^nx)$, that is, 
\begin{equation}\label{Yan-6} 
H_{min,\Om} u:=(-\Delta+V)u,\quad u\in \dom(H_{min,\Om}):=H^2_0(\Omega).
\end{equation}

%%%%%%%%%%%%%%%%%%%%%%%%%%%%%%%%%%%%%
\begin{corollary}\label{Max-M2}
Assume Hypotheses \ref{h2.1} and \ref{h.V}.
Then $H_{min,\Om} $ is a densely defined, symmetric operator which satisfies
\begin{equation} \label{Yan-7}
H_{min,\Om} \subseteq (H_{max,\Om})^*\, \mbox{ and }\, 
H_{max,\Om} \subseteq (H_{min,\Om})^*.
\end{equation}
Equality occurs in one $($and hence, both$)$ inclusions in \eqref{Yan-7} if and only if
\begin{equation} \label{Yan-8}
H^2_0(\Omega) \, \text{ equals } \, \big\{u\in L^2(\Omega;d^nx)\,\big|\,\Delta u\in L^2(\Omega;d^nx), \, 
\widehat{\gamma}_D u =\widehat{\gamma}_N u =0\big\}.
\end{equation} 
\end{corollary}
%%%%%%%%%%%%%%%%%%%%%%%%%%%%%%%%%%%%%

%%%%%%%%%%%%%%%%%%%%%%%%%%%%%%%%%%%%%%
%%%%%%%%%%%%%%%%%%%%%%%%%%%%%%%%%%%%%%
\section{Boundary Value Problems in Quasi-Convex Domains}
\label{s5}
%%%%%%%%%%%%%%%%%%%%%%%%%%%%%%%%%%%%%%
%%%%%%%%%%%%%%%%%%%%%%%%%%%%%%%%%%%%%%

This section is divided into three parts. In Subsection \ref{s5X} we introduce
a distinguished category of the family of Lipschitz domains in $\bbR^n$,
called quasi-convex domains, which is particularly well-suited for the
kind of analysis we have in mind. In Subsection \ref{s6X} and Subsection \ref{s7X},
we then proceed to review, respectively, trace operators and boundary problems,
and Dirichlet-to-Neumann operators in quasi-convex domains.

%%%%%%%%%%%%%%%%%%%%%%%%%%%%%%%%%%%%%%
\subsection{The Class of Quasi-Convex Domains}
\label{s5X}
%%%%%%%%%%%%%%%%%%%%%%%%%%%%%%%%%%%%%%

In the class of Lipschitz domains, the two spaces appearing in \eqref{Yan-8}
are not necessarily equal (although, obviously, the left-to-right inclusion
always holds). The question now arises: What extra properties of the
Lipschitz domain will guarantee equality in \eqref{Yan-8}?
This issue has been addressed in \cite{GM10}, where a class of domains
(which is in the nature of best possible) has been identified.

To describe this class, we need some preparations. Given $n\geq 1$, denote by
$MH^{1/2}(\bbR^n)$ the class of pointwise multipliers of the Sobolev
space $H^{1/2}(\bbR^n)$. That is,
\begin{equation} \label{MaS-1}
MH^{1/2}(\bbR^n):=\bigl\{f\in L^1_{\loc}(\bbR^n)\,\big|\,
M_f\in\cB\bigl(H^{1/2}(\bbR^n)\bigr)\bigr\},
\end{equation} 
where $M_f$ is the operator of pointwise multiplication by $f$. This
space is equipped with the natural norm, that is, 
\begin{equation} \label{MaS-2}
\|f\|_{MH^{1/2}(\bbR^n)}:=\|M_f\|_{\cB(H^{1/2}(\bbR^n))}.
\end{equation} 
For a comprehensive and systematic treatment of spaces of multipliers,
the reader is referred to the 1985 monograph of Maz'ya and Shaposhnikova 
\cite{MS85}. Following \cite{MS85}, \cite{MS05}, we now
introduce a special class of domains, whose boundary regularity properties
are expressed in terms of spaces of multipliers.

%%%%%%%%%%%%%%%%%%%%%%%%%%%%%%%%%%%%
\begin{definition}\label{Def-MS}
Given $\delta>0$, call a bounded, Lipschitz domain $\Omega\subset\bbR^n$
to be of class $MH^{1/2}_\delta$, and write
\begin{equation} \label{MaS-3}
\dOm\in MH^{1/2}_\delta,
\end{equation} 
provided the following holds: There exists a finite open covering
$\{{\mathcal O}_j\}_{1\leq j\leq N}$ of the boundary $\partial\Omega$ of
$\Om$ such that for every $j\in\{1,...,N\}$, ${\mathcal O}_j\cap\Omega$
coincides with the portion of ${\mathcal O}_j$ lying in the over-graph of
a Lipschitz function $\varphi_j:\bbR^{n-1}\to\bbR$ $($considered in a new
system of coordinates obtained from the original one via a rigid motion$)$
which, additionally, has the property that
\begin{equation} \label{MaS-4}
\nabla\varphi_j\in \big(MH^{1/2}(\bbR^{n-1})\big)^n\, \mbox{ and }\, 
\|\varphi_j\|_{(MH^{1/2}(\bbR^{n-1}))^n}\leq\delta.
\end{equation} 
\end{definition}
%%%%%%%%%%%%%%%%%%%%%%%%%%%%%%%%%%%%

Going further, we consider the classes of domains
\begin{equation} \label{MaS-5}
MH^{1/2}_\infty:=\bigcup_{\delta>0}MH^{1/2}_\delta,\quad
MH^{1/2}_0:=\bigcap_{\delta>0}MH^{1/2}_\delta,
\end{equation} 
and also introduce the following definition: 

%%%%%%%%%%%%%%%%%%%%%%%%%%%%%%%%%%%%
\begin{definition}\label{Def-MS2}
We call a bounded Lipschitz domain $\Omega\subset\bbR^n$ to be
{\it square-Dini}, and write
\begin{equation} \label{MaS-6}
\dOm\in {\rm SD},
\end{equation} 
provided the following holds: There exists a finite open covering
$\{{\mathcal O}_j\}_{1\leq j\leq N}$ of the boundary $\partial\Omega$ of
$\Om$ such that for every $j\in\{1,...,N\}$, ${\mathcal O}_j\cap\Omega$
coincides with the portion of ${\mathcal O}_j$ lying in the over-graph of
a Lipschitz function $\varphi_j:\bbR^{n-1}\to\bbR$ $($considered in a new
system of coordinates obtained from the original one via a rigid motion$)$
which, additionally, has the property that the following square-Dini
condition holds,
\begin{equation} \label{MaS-7}
\int_0^1 \frac{dt}{t} \bigg(\frac{\omega(\nabla\varphi_j;t)}{t^{1/2}}\bigg)^2 
<\infty.
\end{equation} 
Here, given a $($possibly vector-valued\,$)$ function $f$ in $\bbR^{n-1}$,
\begin{equation} \label{MaS-8}
\omega(f;t):=\sup\,\{|f(x)-f(y)|\,|\,x,y\in\bbR^{n-1},\,\,|x-y|\leq t\},
\quad t\in(0,1),
\end{equation} 
is the modulus of continuity of $f$, at scale $t$.
\end{definition}
%%%%%%%%%%%%%%%%%%%%%%%%%%%%%%%%%%%

 From the work of Maz'ya and Shaposhnikova \cite{MS85} \cite{MS05},
it is known that if $r>1/2$, then
\begin{equation} \label{MaS-9}
\Om\in C^{1,r}\Longrightarrow
\Om\in{\rm SD}\Longrightarrow
\Om\in MH^{1/2}_0\Longrightarrow
\Om\in MH^{1/2}_\infty.
\end{equation} 
As pointed out in \cite{MS05}, domains of class $MH^{1/2}_\infty$ can have
certain types of vertices and edges when $n\geq 3$. Thus, the domains
in this class can be nonsmooth.

Next, we recall that a domain is said to satisfy a uniform exterior
ball condition (UEBC) provided there exists a number $r>0$ with the property that
\begin{align} \label{UEBC}
\begin{split}
& \mbox{for every $x\in\dOm$, there exists $y\in\bbR^n$, such that 
$B(y,r)\cap\Om=\emptyset$} \\
& \quad \mbox{and $x\in\partial B(y,r)\cap\dOm$}.
\end{split} 
\end{align} 
Heuristically, \eqref{UEBC} should be interpreted as a lower bound on
the curvature of $\partial\Omega$. Next, we review the class of almost-convex 
domains introduced in \cite{MTV}. 

%%%%%%%%%%%%%%%%%
\begin{definition}\label{Def-AC}
A bounded Lipschitz domain $\Omega\subset{\mathbb{R}}^n$ is called 
an almost-convex domain provided there exists a family 
$\{\Omega_\ell\}_{\ell\in{\mathbb{N}}}$
of open sets in ${\mathbb{R}}^n$ with the following properties:  
\begin{enumerate}
\item[$(i)$] $\partial\Omega_\ell\in C^2$ and 
$\overline{\Omega_{\ell}}\subset\Omega$ for every $\ell\in{\mathbb{N}}$. 
\item[$(ii)$] $\Omega_\ell\nearrow\Omega$ as $\ell\to\infty$, in the sense
that $\overline{\Omega_{\ell}}\subset\Omega_{\ell+1}$ for each 
$\ell\in{\mathbb{N}}$ and $\bigcup_{\ell\in{\mathbb{N}}}\Omega_{\ell}=\Omega$. 
\item[$(iii)$] There exists a neighborhood $U$ of $\partial\Omega$ and, 
for each $\ell\in{\mathbb{N}}$, a $C^2$ real-valued function $\rho_{\ell}$ 
defined in $U$ with the property that $\rho_{\ell}<0$ on $U\cap\Omega_{\ell}$, 
$\rho_{\ell}>0$ in $U \backslash \overline{\Omega_{\ell}}$, and $\rho_{\ell}$  
vanishes on $\partial\Omega_\ell$. In addition, it is assumed that 
there exists some constant $C_1\in (1,\infty)$ such that 
\begin{eqnarray}\label{MTV3.1}
C_1^{-1}\leq |\nabla\rho_\ell(x)|\leq C_1,
\quad x\in \partial\Omega_\ell,\; \ell\in{\mathbb{N}}. 
\end{eqnarray} 
\item[$(iv)$] There exists $C_2\geq 0$ such that for every number
$\ell\in{\mathbb{N}}$, every point $x\in\partial\Omega_{\ell}$, 
and every vector $\xi\in{\mathbb{R}}^n$ which is tangent to 
$\partial\Omega_{\ell}$ at $x$, there holds 
\begin{eqnarray}\label{MTV3.2}
\big\langle{\rm Hess}\,(\rho_\ell)\xi\,,\,\xi\big\rangle\geq -C_2|\xi|^2, 
\end{eqnarray}
\noindent where $\langle\dott,\dott\rangle$ is the standard Euclidean inner 
product in ${\mathbb{R}}^n$ and  
\begin{eqnarray}\label{MTV3.3}
{\rm Hess}\,(\rho_\ell):=\left(\frac{\partial^2\rho_\ell}
{\partial x_j\partial x_k}\right)_{1\leq j,k \leq n},
\end{eqnarray} 
\noindent is the Hessian of $\rho_{\ell}$. 
\end{enumerate}
\end{definition}
%%%%%%%%%%%%%%%%%

\noindent A few remarks are in order: First, it is not difficult to see
that \eqref{MTV3.1} ensures that each domain $\Omega_\ell$ is Lipschitz, 
with Lipschitz constant bounded uniformly in $\ell$. Second, \eqref{MTV3.2}
simply says that, as quadratic forms on the tangent bundle 
$T\partial\Omega_\ell$ to $\partial\Omega_{\ell}$, one has  
\begin{eqnarray}\label{MT-SR}
{\rm Hess}\,(\rho_\ell)\geq -C_2\,I_n,
\end{eqnarray}
\noindent where $I_n$ is the $n\times n$ identity matrix. Hence, another equivalent 
formulation of \eqref{MTV3.2} is the following requirement: 
\begin{eqnarray}\label{MTV3.4}
\sum\limits_{j,k=1}^n\frac{\partial^2\rho_\ell}{\partial x_j \partial x_k}
\xi_j \xi_k \geq -C_2 \sum\limits_{j=1}^n\xi_j^2, \, \mbox{ whenever }\,
\rho_\ell=0\,\mbox{ and }\,\sum\limits_{j=1}^n
\frac{\partial\rho_\ell}{\partial x_j}\xi_j =0.
\end{eqnarray}
\noindent We note that, since the second fundamental form $II_{\ell}$ on 
$\partial\Omega_{\ell}$ is 
$II_\ell={{\rm Hess}\,\rho_\ell}/{|\nabla\rho_\ell|}$,
almost-convexity is, in view of \eqref{MTV3.1}, equivalent to 
requiring that $II_\ell$ be bounded from below, uniformly in $\ell$.

We now discuss some important special classes of almost-convex
domains. 

%%%%%%%%%%%%%%%%%%%%%
\begin{definition}\label{eu-RF}
A bounded Lipschitz domain $\Omega\subset{\mathbb{R}}^n$ satisfies 
a local exterior ball condition, henceforth referred to as LEBC,
if every boundary point $x_0\in\partial\Omega$ has an open 
neighborhood ${\mathcal{O}}$ which satisfies the following two conditions: 
\begin{enumerate}
\item[$(i)$] There exists a Lipschitz function
$\varphi:{{\mathbb{R}}}^{n-1}\to{{\mathbb{R}}}$ with
$\varphi(0)=0$ such that if $D$ is the domain above the graph of $\varphi$, 
then $D$ satisfies a UEBC.  
\item[$(ii)$] There exists a $C^{1,1}$ diffeomorphism $\Upsilon$ mapping 
${\mathcal{O}}$ onto the unit ball $B(0,1)$ in ${{\mathbb{R}}}^n$ such 
that $\Upsilon(x_0)=0$, $\Upsilon({\mathcal{O}}\cap\Omega)=B(0,1)\cap D$,
$\Upsilon({\mathcal{O}} \backslash {\ol\Omega})=B(0,1) \backslash \overline{D}$.
\end{enumerate}
\end{definition}
%%%%%%%%%%%%%%%

\noindent It is clear from Definition \ref{eu-RF} that the class of 
bounded domains satisfying a LEBC is invariant under $C^{1,1}$ diffeomorphisms.
This makes this class of domains amenable to working on manifolds. 
This is the point of view adopted in \cite{MTV}, where the following 
result is also proved:  
%%%%%%%%%%%%%%%%
\begin{lemma}\label{MTVp3.1}
If the bounded Lipschitz domain $\Omega\subset{\mathbb{R}}^n$ satisfies 
a LEBC then it is almost-convex.
\end{lemma}
%%%%%%%%%%%%%%%%
\noindent Hence, in the class of bounded Lipschitz domains 
in ${\mathbb{R}}^n$, we have
\begin{eqnarray}\label{ewT-1}
\mbox{convex}\,\Longrightarrow\,
\mbox{UEBC}\,\Longrightarrow\,
\mbox{LEBC}\,\Longrightarrow\,
\mbox{almost-convex}.
\end{eqnarray}
We are now in a position to specify the class of domains in which most
of our subsequent analysis will be carried out. 

%%%%%%%%%%%%%%%%%%%%%%
\begin{definition}\lb{d.Conv}
Let $n\in\bbN$, $n\geq 2$, and assume that $\Omega\subset{\bbR}^n$ is
a bounded Lipschitz domain. Then $\Om$ is called a quasi-convex domain if 
there exists $\delta>0$ 
sufficiently small $($relative to $n$ and the Lipschitz character of $\Om$$)$, 
with the following property that for every $x\in\dOm$ there exists an open 
subset $\Om_x$ of $\Omega$ such that $\dOm\cap\dOm_x$ is an open neighborhood 
of $x$ in $\dOm$, and for which one of the following two conditions holds: \\
$(i)$ \, $\Omega_x$ is of class $MH^{1/2}_\delta$ if $n\geq 3$, and 
of class $C^{1,r}$ for some $1/2<r<1$ if $n=2$.   \\
$(ii)$ $\Omega_x$ is an almost-convex domain. 
\end{definition}
%%%%%%%%%%%%%%%%%%%%%% 

Given Definition \ref{d.Conv}, we thus introduce the following basic assumption:

%%%%%%%%%%%%%%%%%%%%%%
\begin{hypothesis}\lb{h.Conv}
Let $n\in\bbN$, $n\geq 2$, and assume that $\Omega\subset{\bbR}^n$ is
a quasi-convex domain. 
\end{hypothesis}
%%%%%%%%%%%%%%%%%%%%%% 

Informally speaking, the above definition ensures that the boundary
singularities are directed outwardly. A typical example of such a domain
is shown in Fig.\ \ref{Pic} below.  
%\begin{equation} \label{Pic}
%\begin{center}
%\includegraphics[scale=0.5]{strange2}
%\end{center}
%\end{equation} 
\begin{figure}[th]
\centering
\begin{picture}(4,4)
\put(2,1.8){$\Omega$}
\curve(0.1,2, 1.1,2.2, 1.3,3.1, 0.9,3.6)
\curve(0.9,3.6, 1.4,3.5, 1.8,3.1)
\curve(1.8,3.1, 2.3,2.9, 2.7,3.1)
\curve(2.7,3.1, 3.2,3.45, 3.9,3.5)
\curve(3.9,3.5, 3,3.1, 3.05,2.9, 3.2,2.85, 3.3,2.6, 3.4,2.6, 3.6,2.7, 4.1,2.8)
\curve(4.1,2.8, 3.5,1.8, 4.2,0.9)
\curve(4.2,0.9, 3,0.8, 2.15,0.65, 1.6,0.1)
\curve(1.6,0.1, 1.7,0.65, 1.5,0.9, 1.25,1.15, 1.1,1.35, 0.8,1.7, 0.1,2)
\end{picture}
\caption{A quasi-convex domain.}\label{Pic}
\end{figure}
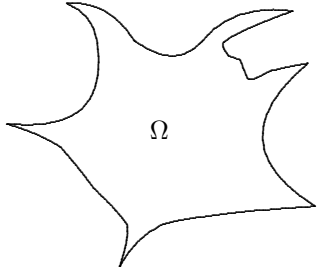

Being quasi-convex is a certain type of regularity condition of the boundary
of a Lipschitz domain. The only way we are going to utilize this
property is via the following elliptic regularity result proved in
\cite{GM10}.

%%%%%%%%%%%%%%%%%%%%%%%%%%%%%%%%%%%%%%%
\begin{proposition}\label{Bjk}
Assume Hypotheses \ref{h.V} and \ref{h.Conv}. Then
\begin{equation}\label{Yan-9}
\dom\big(H_{D,\Om}\big)\subset H^{2}(\Omega), \quad
\dom\big(H_{N,\Om}\big)\subset H^{2}(\Omega).
\end{equation}
\end{proposition}
%%%%%%%%%%%%%%%%%%%%%%%%%%%%%%%%%%%%%%%
%
\noindent In fact, all of our results in this paper hold in the class of Lipschitz
domains for which the two inclusions in \eqref{Yan-9} hold.

The following theorem addresses the issue raised at
the beginning of this subsection. Its proof is similar to the
special case $V\equiv 0$, treated in \cite{GM10}.

%%%%%%%%%%%%%%%%%%%%%%%%%%%%%%%%%%%%%%%
\begin{theorem}\label{T-DD1}
Assume Hypotheses \ref{h.V} and \ref{h.Conv}.
Then \eqref{Yan-8} holds. In particular,
\begin{align}
\dom(H_{min,\Om}) & = H^2_0(\Omega)   \no \\
 & =\big\{u\in L^2(\Omega;d^nx)\,\big|\,\Delta u\in L^2(\Omega;d^nx), \, 
\widehat{\gamma}_D u =\widehat{\gamma}_N u =0\big\},     \lb{dmin} \\
\dom(H_{max,\Om} ) & = \big\{u\in L^2(\Omega;d^nx)\,\big|\,
\Delta u \in L^2(\Omega;d^nx)\big\},    \lb{dmax}
\end{align}
and
\begin{equation} \label{Yan-10}
H_{min,\Om} 
= (H_{max,\Om})^*\, \mbox{ and }\, H_{max,\Om} = (H_{min,\Om})^*.
\end{equation} 
\end{theorem}
%%%%%%%%%%%%%%%%%%%%%%%%%%%%%%%%%%%%%%%

We conclude this subsection with the following result which is essentially
contained in \cite{GM10}.

%%%%%%%%%%%%%%%%%%%%%%%%%%%%%%%%%%%%%%
\begin{proposition}\label{L-Fri1}
Assume Hypotheses \ref{h2.1} and \ref{h.V}.
Then the Friedrichs extension of $(-\Delta+V)|_{C^\infty_0(\Om)}$ in 
$L^2(\Om;d^nx)$ is precisely the perturbed Dirichlet Laplacian $H_{D,\Om}$.
Consequently, if Hypothesis \ref{h.Conv} is assumed in place
of Hypothesis \ref{h2.1}, then the Friedrichs extension
of $H_{min,\Om} $ in \eqref{Yan-6} is the perturbed Dirichlet Laplacian
$H_{D,\Om}$.
\end{proposition}
%%%%%%%%%%%%%%%%%%%%%%%%%%%%%%%%%%%%%

%%%%%%%%%%%%%%%%% appendix E %%%%%%%%%%%%%%
\subsection{Trace Operators and Boundary Problems on Quasi-Convex Domains}
\label{s6X}
%%%%%%%%%%%%%%%%%%%%%%%%%%%%%%%%%%%%%%

Here we revisit the issue of traces, originally taken up in Section \ref{s2},
and extend the scope of this theory. The goal is to extend our earlier
results to a context that is well-suited for the treatment of the
perturbed Krein Laplacian in quasi-convex domains, later on.
All results in this subsection are direct generalizations of similar results
proved in the case where $V\equiv 0$ in \cite{GM10}.

%%%%%%%%%%%%%%%%%%%%%%%%%%%%%%%%%
\begin{theorem}\label{tH.A}
Assume Hypotheses \ref{h.V} and \ref{h.Conv},
and suppose that $z\in\bbC\backslash\si(H_{D,\Om})$. Then for any
functions $f\in L^2(\Om;d^nx)$ and $g\in (N^{1/2}(\partial\Omega))^*$
the following inhomogeneous Dirichlet boundary value problem
\begin{equation}\label{Yan-14}
\begin{cases}
(-\Delta+V-z)u=f\text{ in }\,\Om,
\\
u\in L^2(\Om;d^nx),
\\
\widehat\ga_D u =g\text{ on }\,\dOm,
\end{cases}
\end{equation}
has a unique solution $u=u_D$. This solution satisfies
\begin{equation}\label{Hh.3X}
\|u_D\|_{L^2(\Om;d^nx)}
+\|\widehat\ga_N u_D\|_{(N^{3/2}(\partial\Omega))^*}\leq C_D
(\|f\|_{L^2(\Om;d^nx)}+\|g\|_{(N^{1/2}(\partial\Omega))^*})
\end{equation}
for some constant $C_D=C_D(\Omega,V,z)>0$, and the following regularity
results hold:
\begin{align} \label{3.3Y}
& g\in H^1(\partial\Omega) \,\text{ implies }\, u_D\in H^{3/2}(\Omega),
\\
& g\in\gamma_D\bigl(H^2(\Omega)\bigr)
\, \text{ implies }\, u_D\in H^2(\Omega).
\label{3.3Ys}
\end{align} 
In particular,
\begin{equation} \label{3.3Ybis}
g=0\, \text{ implies } \, u_D\in H^2(\Omega)\cap H^1_0(\Omega).
\end{equation} 
Natural estimates are valid in each case.

Moreover, the solution operator for \eqref{Yan-14} with $f=0$
$($i.e., $P_{D,\Om,V,z}:g\mapsto u_D$$)$ satisfies
\begin{equation}\label{3.34Y}
P_{D,\Om,V,z}=\big[\ga_N(H_{D,\Om}-{\ol z}I_\Om)^{-1}\big]^*
\in\cB\big((N^{1/2}(\partial\Omega))^*,L^2(\Om;d^nx)\big),
\end{equation}
and the solution of \eqref{Yan-14} is given by the formula
\begin{equation}\label{3.35Y}
u_D=(H_{D,\Om}-zI_\Om)^{-1}f
-\big[\ga_N(H_{D,\Om}-\ol{z}I_\Om)^{-1}\big]^*g.
\end{equation}
\end{theorem}
%%%%%%%%%%%%%%%%%%%%%

%%%%%%%%%%%%%%%%%%%%%%%%%
\begin{corollary}\label{New-CV22}
Assume Hypotheses \ref{h.V} and \ref{h.Conv}.
Then for every $z\in\bbC\backslash \si(H_{D,\Om})$ the map  
\begin{equation} \label{Tan-Bq2}
\widehat{\gamma}_D: \big\{u\in L^2(\Omega;d^nx)\,\big|\,(-\Delta+V-z)u=0
\,\mbox{in}\, \Omega\}\to \bigl(N^{1/2}(\partial\Omega)\bigr)^*
\end{equation} 
is an isomorphism $($i.e., bijective and bicontinuous\,$)$. 
\end{corollary}
%%%%%%%%%%%%%%%%%%%%%%%%%

%%%%%%%%%%%%%%%%%%%%%%%%%%%%%%%%%
\begin{theorem}\label{tH.G2}
Assume Hypotheses \ref{h.V} and \ref{h.Conv}   
and suppose that $z\in\bbC\backslash\si(H_{N,\Om})$. Then for any functions
$f\in L^2(\Om;d^nx)$ and $g\in (N^{3/2}(\partial\Omega))^*$ the
following inhomogeneous Neumann boundary value problem
\begin{equation}\label{n-1H}
\begin{cases}
(-\Delta+V-z)u=f\text{ in }\,\Om,
\\[4pt]
u\in L^2(\Om;d^nx),
\\[4pt]
\widehat\ga_N u=g\text{ on }\,\dOm,
\end{cases}
\end{equation}
has a unique solution $u=u_N$. This solution satisfies
\begin{equation}\label{Hh.3f}
\|u_N\|_{L^2(\Om;d^nx)}
+\|\widehat\ga_D u_N\|_{(N^{1/2}(\partial\Omega))^*}\leq C_N
(\|f\|_{L^2(\Om;d^nx)}+\|g\|_{(N^{3/2}(\partial\Omega))^*})
\end{equation}
for some constant $C_N=C_N(\Omega,V,z)>0$, and the following regularity
results hold:
\begin{align} \label{3.3f}
& g\in L^2(\partial\Omega;d^{n-1}\omega)\, \text{ implies }\, u_N\in H^{3/2}(\Omega),
\\
& g\in\gamma_N\bigl(H^2(\Om)\bigr)\, \text{ implies }\, u_N\in H^2(\Omega).
\label{3.3fbis}
\end{align} 
Natural estimates are valid in each case.

Moreover, the solution operator for \eqref{n-1H} with $f=0$
$($i.e., $P_{N,\Om,V,z}:g\mapsto u_N$$)$ satisfies
\begin{equation}\label{3.34f}
P_{N,\Om,V,z}=\big[\ga_D(H_{N,\Om}-{\ol z}I_\Om)^{-1}\big]^*
\in\cB\big((N^{3/2}(\partial\Omega))^*,L^2(\Om;d^nx)\big),
\end{equation}
and the solution of \eqref{n-1H} is given by the formula
\begin{equation}\label{3.a5Y}
u_N=(H_{N,\Om}-zI_\Om)^{-1}f
+\big[\ga_D(H_{N,\Om}-\ol{z}I_\Om)^{-1}\big]^*g.
\end{equation}
\end{theorem}
%%%%%%%%%%%%%%%%%%%%%

%%%%%%%%%%%%%%%%%%%%%%%%%
\begin{corollary}\label{New-CV33}
Assume Hypotheses \ref{h.V} and \ref{h.Conv}.
Then, for every $z\in\bbC\backslash \sigma(H_{N,\Om})$, the map 
\begin{equation} \label{Tan-FFF}
\widehat{\gamma}_N:\big\{u\in L^2(\Omega;d^nx)\,\big|\,(-\Delta+V-z)u=0
\,\mbox{ in }\Omega\big\}\to \bigl(N^{3/2}(\partial\Omega)\bigr)^*
\end{equation} 
is an isomorphism $($i.e., bijective and bicontinuous\,$)$.
\end{corollary}
%%%%%%%%%%%%%%%%%%%%%%%%%

%%%%%%%%%%%%%%%%%%%%%%%%%%%%%%%
\subsection{Dirichlet-to-Neumann Operators on Quasi-Convex Domains}
\label{s7X}
%%%%%%%%%%%%%%%%%%%%%%%%%%%%%%%%%%%%%%

In this subsection we review spectral parameter dependent Dirichlet-to-Neumann maps, also known in the literature as Weyl--Titchmarsh and Poincar\'e--Steklov operators. Assuming Hypotheses \ref{h.V} and \ref{h.Conv}, introduce the 
Dirichlet-to-Neumann map
$M_{D,N,\Om,V}(z)$ associated with $-\Delta+V-z$ on $\Om$, as follows:
\begin{equation} \label{3.44v}
M_{D,N,\Om,V}(z) \colon
\begin{cases}
\bigl(N^{1/2}(\dOm)\bigr)^*\to \bigl(N^{3/2}(\dOm)\bigr)^*,  \\
\hspace*{1.8cm}
f\mapsto -\widehat\ga_N u_D,
\end{cases}
\; z\in\bbC\backslash\si(H_{D,\Om}),
\end{equation} 
where $u_D$ is the unique solution of
\begin{equation} \label{3.45v}
(-\Delta+V-z)u=0\,\text{ in }\Om,\quad u\in L^2(\Om;d^nx),
\;\; \widehat\ga_D u=f\,\text{ on }\dOm.
\end{equation} 
Retaining Hypotheses \ref{h.V} and \ref{h.Conv},
we next introduce the Neumann-to-Dirichlet map $M_{N,D,\Om,V}(z)$
associated with $-\Delta+V-z$ on $\Om$, as follows:
\begin{equation} \label{3.48v}
M_{N,D,\Om,V}(z)\colon
\begin{cases}
\bigl(N^{3/2}(\dOm)\bigr)^*\to \bigl(N^{1/2}(\dOm)\bigr)^*,
\\
\hspace*{1.8cm}
g\mapsto\widehat\ga_D u_N,
\end{cases}
\; z\in\bbC\backslash\si(H_{N,\Om}),
\end{equation} 
where $u_{N}$ is the unique solution of
\begin{equation} \label{3.49v}
(-\Delta+V-z)u=0\,\text{ in }\Om,\quad u\in L^2(\Om;d^nx),
\;\; \widehat\ga_Nu=g\,\text{ on }\dOm.
\end{equation} 
As in \cite{GM10}, where the case $V\equiv 0$ has been treated,
we then have the following result:

%%%%%%%%%%%%%%
\begin{theorem}\label{t3.5v}
Assume Hypotheses \ref{h.V} and \ref{h.Conv}.
Then, with the above notation,
\begin{equation} \label{3.46v}
M_{D,N,\Om,V}(z)\in\cB\big((N^{1/2}(\dOm))^*\,,\,(N^{3/2}(\dOm))^*\big),
\quad z\in\bbC\backslash\si(H_{D,\Om}),
\end{equation}
and
\begin{equation}\label{3.47v}
M_{D,N,\Om,V}(z)=\widehat\gamma_N
\big[\gamma_N(H_{D,\Om}-\ol{z}I_\Om)^{-1}\big]^*,
\quad z\in\bbC\backslash\si(H_{D,\Om}).
\end{equation}
Similarly,
\begin{equation}\label{3.50v}
M_{N,D,\Om,V}(z)\in\cB\big((N^{3/2}(\dOm))^*\,,\,(N^{1/2}(\dOm))^*\big),
\quad z\in\bbC\backslash\si(H_{N,\Om}),
\end{equation}
and
\begin{equation}\label{3.52v}
M_{N,D,\Om,V}(z)
= \widehat \gamma_D\big[\gamma_D(H_{N,\Om}-\ol{z}I_\Om)^{-1}\big]^*,
\quad z\in\bbC\backslash\si(H_{N,\Om}).
\end{equation}
Moreover,
\begin{equation}\label{3.53v}
M_{N,D,\Om,V}(z)=-M_{D,N,\Om,V}(z)^{-1},\quad
z\in\bbC\backslash(\si(H_{D,\Om})\cup\si(H_{N,\Om})),
\end{equation}
and
\begin{equation}\label{NaLa}
\big[M_{D,N,\Om,V}(z)\big]^*=M_{D,N,\Om,V}(\ol{z}),\quad
\big[M_{N,D,\Om,V}(z)\big]^*=M_{N,D,\Om,V}(\ol{z}).
\end{equation}
As a consequence, one also has
\begin{align} \label{3.TTa}
& M_{D,N,\Om,V}(z)\in\cB\big(N^{3/2}(\dOm)\,,\,N^{1/2}(\dOm)\big),
\quad z\in\bbC\backslash\si(H_{D,\Om}),
\\
& M_{N,D,\Om,V}(z)\in\cB\big(N^{1/2}(\dOm)\,,\,N^{3/2}(\dOm)\big),
\quad z\in\bbC\backslash\si(H_{N,\Om}).
\label{3.TTb}
\end{align} 
\end{theorem}
%%%%%%%%%%%%%%

For closely related recent work on Weyl--Titchmarsh operators associated with nonsmooth domains we refer to \cite{GM08}, \cite{GM09a}, \cite{GM09b},  \cite{GM10}, and \cite{GMZ07}. 
For an extensive list of references on $z$-dependent Dirichlet-to-Neumann maps we also refer, for instance, to \cite{Ag03}, \cite{ABMN05}, \cite{AP04}, \cite{BL07}, \cite{BMN02}, \cite{BGW09}, \cite{BHMNW09}, \cite{BMNW08}, \cite{BGP08}, 
\cite{DM91}, \cite{DM95}, \cite{GLMZ05}--\cite{GMZ07}, \cite{Gr08a}, \cite{Po08}, 
\cite{Ry07}, \cite{Ry09}, \cite{Ry10}.

%%%%%%%%%%%%%%%%%%%%%%%%%%%%%%%%%%%%%%
%%%%%%%%%%%%%%%%%%%%%%%%%%%%%%%%%%%%%%
\section{Regularized Neumann Traces and Perturbed Krein Laplacians}
\label{s8}
%%%%%%%%%%%%%%%%%%%%%%%%%%%%%%%%%%%%%%
%%%%%%%%%%%%%%%%%%%%%%%%%%%%%%%%%%%%%%

This section is structured into two parts dealing, respectively, with
the regularized Neumann trace operator (Subsection \ref{s8X}), and the
perturbed Krein Laplacian in quasi-convex domains (Subsection \ref{s9X}). 

%%%%%%%%%%%%%%%%%%%%%%%%%%%%%%%
\subsection{The Regularized Neumann Trace Operator on Quasi-Convex Domains}
\label{s8X}
%%%%%%%%%%%%%%%%%%%%%%%%%%%%%%%%%%%%%%

Following earlier work in \cite{GM10}, we now consider a version of the
Neumann trace operator which is suitably normalized to permit
the familiar version of Green's formula (cf.\ \eqref{T-Green} below)
to work in the context in which the functions involved are only known
to belong to $\dom(-\Delta_{\max,\Om})$. The following theorem is a
slight extension of a similar result proved in \cite{GM10} when $V\equiv 0$.

%%%%%%%%%%%%%%
\begin{theorem}\label{LL.w}
Assume Hypotheses \ref{h.V} and \ref{h.Conv}.
Then, for every $z\in\bbC\backslash\si(H_{D,\Om})$, the map 
\begin{equation} \label{3.Aw1}
\tau_{N,V,z}:\bigl\{u\in L^2(\Om;d^nx);\,\Delta u\in L^2(\Om;d^nx)\bigr\}
\to N^{1/2}(\partial\Omega)
\end{equation} 
given by
\begin{equation} \label{3.Aw2}
\tau_{N,V,z} u:=\widehat\gamma_N u 
+M_{D,N,\Om,V}(z)\bigl(\widehat\gamma_D u \bigr),
\quad u\in L^2(\Om;d^nx),\,\,\Delta u\in L^2(\Om;d^nx),
\end{equation} 
is well-defined, linear and bounded, where the space
\begin{equation} 
\big\{u\in L^2(\Om;d^nx)\,\big|\, \Delta u\in L^2(\Om;d^nx)\big\}  
\end{equation}
is endowed with the natural graph norm
$u\mapsto\|u\|_{L^2(\Om;d^nx)}+\|\Delta u\|_{L^2(\Om;d^nx)}$.
Moreover, this operator satisfies the following additional properties:

\begin{enumerate}
\item[$(i)$] The map $\tau_{N,V,z}$ in \eqref{3.Aw1}, \eqref{3.Aw2} is onto
$($i.e., $\tau_{N,V,z}(\dom(H_{max,\Om} ))=N^{1/2}(\partial\Omega)$$)$,
for each $z\in\bbC\backslash\si(H_{D,\Om})$. In fact,
\begin{equation} \label{3.ON}
\tau_{N,V,z}\bigl(H^2(\Om)\cap H^1_0(\Om)\bigr)=N^{1/2}(\partial\Omega)
\, \mbox{ for each } \, z\in\bbC\backslash\si(H_{D,\Om}).
\end{equation} 
\item[$(ii)$] One has 
\begin{equation} \label{3.Aw9}
\tau_{N,V,z}=\gamma_N(H_{D,\Om}-zI_{\Om})^{-1}(-\Delta-z),\quad
z\in\bbC\backslash \si(H_{D,\Om}).
\end{equation} 
\item[$(iii)$] For each $z\in\bbC\backslash\si(H_{D,\Om})$, the
kernel of the map $\tau_{N,V,z}$ in \eqref{3.Aw1}, \eqref{3.Aw2} is
\begin{equation} \label{3.AKe}
\ker(\tau_{N,V,z})=H^2_0(\Omega)\dot{+}\{u\in L^2(\Om;d^nx)\,|\,
(-\Delta+V-z)u=0\,\mbox{ in }\,\Omega\}.
\end{equation} 
In particular, if $z\in\bbC\backslash\si(H_{D,\Om})$, then
\begin{equation}\label{Sim-Gr}
\tau_{N,V,z} u =0\, \mbox{ for every }\, u\in\ker(H_{max,\Om} -zI_{\Om}).
\end{equation} 
\item[$(iv)$] The following Green formula holds for every
$u,v\in\dom(H_{max,\Om} )$ and every complex number
$z\in\bbC\backslash \si(H_{D,\Om})$:
\begin{align}\label{T-Green}
&  ((-\Delta+V-z)u\,,\,v)_{L^2(\Omega;d^nx)}
- (u\,,\,(-\Delta+V-\ol{z})v)_{L^2(\Omega;d^nx)}
\nonumber\\
& \quad
=-{}_{N^{1/2}(\partial\Omega)}\langle\tau_{N,V,z} u,\widehat\gamma_D v
\rangle_{(N^{1/2}(\partial\Omega))^*}  % \nonumber\\
% & \qquad
+\,\ol{{}_{N^{1/2}(\partial\Omega)}\langle\tau_{N,V,\ol{z}} v,
\widehat{\gamma}_D u \rangle_{(N^{1/2}(\partial\Omega))^*}}.
\end{align}
\end{enumerate}
\end{theorem}
%%%%%%%%%%%%%%

%%%%%%%%%%%%%%%%%%%%%%%%%%%%%%%%%%%%%%
\subsection{The Perturbed Krein Laplacian in Quasi-Convex Domains}
\label{s9X}
%%%%%%%%%%%%%%%%%%%%%%%%%%%%%%%%%%%%%%

We now discuss the Krein--von Neumann extension of the Laplacian 
$-\Delta\big|_{C^\infty_0(\Omega)}$ perturbed by a nonnegative, bounded potential $V$ in 
$L^2(\Om; d^n x)$. We will conveniently call this operator the {\it perturbed Krein Laplacian} and introduce the following basic assumption:

%%%%%%%%%%%%%%%%%%%%%%%%%%%%%%%%%%%%%%%
\begin{hypothesis}\label{h.VK}
$(i)$ Let $n\in\bbN$, $n\geq 2$, and assume that $\emptyset \neq \Om\subset{\bbR}^n$ is
a bounded Lipschitz domain satisfying Hypothesis \ref{h.Conv}. \\
$(ii)$ Assume that
\begin{equation} \label{VV-WW}
V\in L^\infty(\Om;d^nx)
\, \mbox{ and }\, V \geq 0 \mbox{ a.e.\ in } \,\Omega.
\end{equation} 
\end{hypothesis}
%%%%%%%%%%%%%%%%%%%%%%%%%%%%%%%%%%%%%%% 

Denoting by $\ol{T}$ the closure of a linear operator $T$ in a Hilbert space $\cH$, we have the following result: 

%%%%%%%%%%%%%%%%%%%%%%%%%%%%%%%%%%%%%%%
\begin{lemma}\label{C-Da}
Assume Hypothesis \ref{h.VK}.
Then $H_{min,\Om}$ is a densely defined, closed,
nonnegative $($in particular, symmetric$)$ operator in $L^2(\Om; d^n x)$. Moreover,
\begin{equation}\label{Pos-3}
\ol{(-\Delta+V)\big|_{C^\infty_0(\Om)}} = H_{min,\Om}.  
\end{equation}  
\end{lemma}
%%%%%%%%%%%%%%%%%%%%%%%%%%%%%%%%%%%
\begin{proof}
The first claim in the statement is a direct consequence of
Theorem \ref{T-DD1}. As for \eqref{Pos-3}, let us temporarily denote by $H_0$
the closure of $-\Delta+V$ defined on $C^\infty_0(\Om)$. Then
\begin{equation} \label{Pos-4}
u\in\dom(H_0) \, \text{ if and only if }  
\begin{cases}
\mbox{there exist }v\in L^2(\Om;d^nx)\mbox{ and }
u_j\in C^\infty_0(\Om),\,j\in\bbN,\mbox{ such that }
\\
u_j\to u \,\mbox{ and } \,(-\Delta+V)u_j\to v \,\mbox{ in }\,L^2(\Om;d^nx)
\,\mbox{ as }\, j\to\infty.
\end{cases} 
\end{equation}
Thus, if $u\in \dom(H_0)$ and $v$, $\{u_j\}_{j\in\bbN}$ are
as in the right-hand side of \eqref{Pos-4}, then
$(-\Delta+V)u=v$ in the sense of distributions in $\Omega$, and
\begin{align} \label{Pos-5}
\begin{split} 
0&=\widehat\gamma_D u_j \to \widehat\gamma_D u 
\, \mbox{ in }\, \bigl(N^{1/2}(\dOm)\bigr)^* \, \mbox{ as }\, j\to\infty,
\\
0&=\widehat\gamma_N u_j \to \widehat\gamma_N u 
\, \mbox{ in }\, \bigl(N^{1/2}(\dOm)\bigr)^* \, \mbox{ as }\, j\to\infty,
\end{split}
\end{align}
by Theorem \ref{New-T-tr} and Theorem \ref{3ew-T-tr}.
Consequently, $u\in\dom(H_{max,\Om} )$ satisfies
$\widehat\gamma_D u =0$ and $\widehat\gamma_N u =0$.
Hence, $u\in H^2_0(\Om)=\dom(H_{min,\Om} )$ by Theorem \ref{T-DD1}
and the current assumptions on $\Omega$. This shows that
$H_0\subseteq H_{min,\Om} $. The converse inclusion readily follows
from the fact that any $u\in H^2_0(\Om)$ is the limit in
$H^2(\Om)$ of a sequence of test functions in $\Omega$.
\end{proof}
%%%%%%%%%%%%%%%%%%%%%%%%%%%%%%%%%%%%%%

%%%%%%%%%%%%%%%%%%%%%%%%%%%%%%%%%%%%%%
\begin{lemma}\label{C-DaW}
Assume Hypothesis \ref{h.VK}.
Then the Krein--von Neumann extension $H_{K,\Om}$ of 
$(-\Delta+V)\big|_{C^\infty_0(\Om)}$ in $L^2(\Om;d^nx)$ is the $L^2$-realization of 
$-\Delta+V$ with domain
\begin{align} \label{Kre-Frq1}
\begin{split} 
\dom(H_{K,\Om}) &= \dom(H_{min,\Om})\,\dot{+}\ker(H_{max,\Om})  \\
& =H^2_0(\Om)\,\dot{+}\,
\big\{u\in L^2(\Om;d^nx)\,\big|\,(-\Delta+V)u=0\mbox{ in }\Omega\big\}.
\end{split} 
\end{align} 
\end{lemma}
%%%%%%%%%%%%%%%%%%%%%%%%%%%%%%%%%%%%%%
\begin{proof}
By virtue of \eqref{SK}, \eqref{Yan-10}, and the fact that 
$(-\Delta+V)|_{C^\infty_0(\Om)}$ and its closure, $H_{min,\Om}$ (cf.\ \eqref{Pos-3}) 
have the same self-adjoint extensions, one obtains 
\begin{align} \label{Kre-Def}
\dom(H_{K,\Om}) & = \dom(H_{min,\Om})\,\dot{+}\ker((H_{min,\Om})^*)
\nonumber\\
& = \dom(H_{min,\Om} )\,\dot{+}\ker(H_{max,\Om})  \no  \\
& = H^2_0(\Om)\,\dot{+}\,
\big\{u\in L^2(\Om;d^nx)\,\big|\,(-\Delta+V)u=0\mbox{ in }\Omega\big\}, 
\end{align}
as desired.
\end{proof}
%%%%%%%%%%%%%%%%%%%%%%%%%%%%%%%%%%%%%

Nonetheless, we shall adopt a different point of view which better elucidates
the nature of the boundary condition associated with this
perturbed Krein Laplacian. More specifically, following the same
pattern as in \cite{GM10}, the following result can be proved.

%%%%%%%%%%%%%%
\begin{theorem}\label{T-Kr}
Assume Hypothesis \ref{h.VK} and 
fix $z\in\bbC\backslash \si(H_{D,\Om})$.
Then $H_{K,\Om,z}$ in $L^2(\Omega;d^nx)$, given by
\begin{align} \label{A-zz.1}
\begin{split}
& H_{K,\Om,z} u:=(-\Delta+V-z)u, \\
& u\in \dom(H_{K,\Om,z}):=\{v\in\dom(H_{max,\Om} )\,|\, \tau_{N,V,z} v =0\},
\end{split} 
\end{align}
satisfies
\begin{equation} \label{A-zz.W}
(H_{K,\Om,z})^*=H_{K,\Om,\ol{z}}, 
\end{equation} 
and agrees with the self-adjoint perturbed Krein Laplacian $H_{K,\Om}=H_{K,\Om,0}$ 
when taking $z=0$. In particular, if $z\in\bbR\backslash \si(H_{D,\Om})$ then
$H_{K,\Om,z}$ is self-adjoint. Moreover, if $z \leq 0$, then $H_{K,\Om,z}$ is
nonnegative. Hence, the perturbed Krein Laplacian
$H_{K,\Om}$ is a self-adjoint operator in $L^2(\Om;d^nx)$
which admits the description given in \eqref{A-zz.1} when $z=0$, and which
satisfies
\begin{equation} \label{A-zz.b}
H_{K,\Om}\geq 0 \,\mbox{ and }\, 
H_{min,\Om} \subseteq H_{K,\Om}\subseteq H_{max,\Om} .
\end{equation} 
Furthermore, 
\begin{align}
& \ker(H_{K,\Om})=\big\{u\in L^2(\Om;d^nx)\,\big|\,(-\Delta+V)u=0\big\}, \\ 
& \dim(\ker(H_{K,\Om})) = {\rm def} (H_{min,\Om}) 
=  {\rm def} \big(\ol{(-\Delta+V)\big|_{C^\infty_0(\Om)}}\big) =\infty, \\
& \ran(H_{K,\Om})=(-\Delta+V) H^2_0(\Om),  \\
& \text{$H_{K,\Om}$ has a purely discrete spectrum in $(0,\infty)$},  
\quad \sigma_{\rm ess}(H_{K,\Om}) = \{0\},       \label{spec-1} 
\end{align}
and for any nonnegative self-adjoint extension $\wti S$ of 
$(-\Delta+V)|_{C^\infty_0(\Om)}$ one has $($cf.\ \eqref{PPa-1}$)$, 
\begin{equation}\label{Ok.1}
H_{K,\Om}\leq \wti S\leq H_{D,\Om}.
\end{equation} 
\end{theorem}
%%%%%%%%%%%%%%

The nonlocal boundary condition 
\begin{equation}
\tau_{N,V,0} v = \hatt \gamma_N v + M_{D,N,\Om,V} (0) v = 0, \quad 
v \in \dom(H_{K,\Om})
\end{equation} 
(cf.\ \eqref{A-zz.1}) in connection with the Krein--von Neumann extension $H_{K,\Om}$, 
in the special one-dimensional half-line case $\Om= [a,\infty)$ has first been established 
in \cite{Ts87}. In terms of abstract boundary conditions in connection with the theory of 
boundary value spaces, such a condition has been derived in \cite{DMT88} and \cite{DMT89}. 
However, we emphasize that this abstract boundary value space approach, while applicable 
to ordinary differential operators, is not applicable to partial differential operators even in the 
case of smooth boundaries $\partial\Om$ (see, e.g., the discussion in \cite{BL07}). In 
particular, it does not apply to the nonsmooth 
domains $\Om$ studied in this paper. In fact, only very recently, appropriate modifications of 
the theory of boundary value spaces have successfully been applied to partial differential 
operators in smooth domains in \cite{BL07}, \cite{BGW09}, \cite{BHMNW09}, \cite{BMNW08}, 
\cite{Po08}, \cite{PR09}, \cite{Ry07}, \cite{Ry09}, and \cite{Ry10}. With the exception of the following 
short discussions: Subsection\ 4.1 in \cite{BL07} (which treat the special case where $\Om$ 
equals the unit ball in $\bbR^2$), Remark\ 3.8 in \cite{BGW09}, Section\ 2 in \cite{Ry07}, 
Subsection\ 2.4 in \cite{Ry09}, and Remark\ 5.12 in \cite{Ry10}, these investigations did not enter a detailed discussion of the Krein-von Neumann extension. In particular, none of these references 
applies to the case of nonsmooth domains $\Om$.

%%%%%%%%%%%%%%%%%%%%%%%%%%%%%%%%%%%%%%
%%%%%%%%%%%%%%%%%%%%%%%%%%%%%%%%%%%%%%
\section{Connections with the Problem of the Buckling of a Clamped Plate}
\label{s10}
%%%%%%%%%%%%%%%%%%%%%%%%%%%%%%%%%%%%%%
%%%%%%%%%%%%%%%%%%%%%%%%%%%%%%%%%%%%%%

In this section we proceed to study a fourth-order problem, which is a
perturbation of the classical problem for the buckling of a clamped plate,
and which turns out to be essentially spectrally equivalent to the
perturbed Krein Laplacian $H_{K,\Om}:=H_{K,\Om,0}$.

For now, let us assume Hypotheses \ref{h2.1} and \ref{h.V}.
Given $\lambda\in\bbC$, consider the eigenvalue problem for the
generalized buckling of a clamped plate in the domain $\Om\subset\bbR^n$
\begin{equation}\label{MM-1}
\begin{cases}
u\in\dom(-\Delta_{max,\Om}),
\\
(-\Delta+V)^2u=\lambda\,(-\Delta+V)u\,\mbox{ in }\, \Omega,
\\
\widehat\ga_D u =0 \,\mbox{ in } \,\big(N^{1/2}(\dOm)\big)^*,
\\
\widehat\ga_N u =0 \,\mbox{ in } \,\big(N^{3/2}(\dOm)\big)^*,
\end{cases}
\end{equation}
where $(-\Delta+V)^2u:=(-\Delta+V)(-\Delta u+Vu)$ in the sense of
distributions in $\Om$. Due to the trace theory developed in Sections \ref{s3} 
and \ref{s5}, this formulation is meaningful. In addition, if Hypothesis \ref{h.Conv} is assumed
in place of Hypothesis \ref{h2.1} then, by \eqref{Yan-8}, this problem
can be equivalently rephrased as
\begin{equation}\label{MM-2} 
\begin{cases}
u\in H^2_0(\Omega),
\\
(-\Delta+V)^2u=\lambda\,(-\Delta+V)u \,\mbox{ in } \,\Omega.
\end{cases}
\end{equation}

%%%%%%%%%%%%%%%%%%%%%%%%%%%%%%%%%%
\begin{lemma}\label{L-MM-1}
Assume Hypothesis \ref{h.VK} and suppose
that $u\not=0$ solves \eqref{MM-1} for some $\lambda\in\bbC$. Then necessarily
$\lambda\in (0,\infty)$.
\end{lemma}
%%%%%%%%%%%%%%%%%%%%%%%%%%%%%%%%%%%
\begin{proof}
Let $u,\lambda$ be as in the statement of the lemma. Then,
as already pointed out above, $u\in H^2_0(\Omega)$. Based on this,
the fact that $\Delta u\in\dom(-\Delta_{max,\Om})$, and the integration
by parts formulas \eqref{2.9} and \eqref{Tan-C12}, we may then write
(we recall that our $L^2$ pairing is conjugate linear in the {\it first}
argument):
\begin{align}\label{MM-3}
& \lambda\bigl[\|\nabla u\|^2_{(L^2(\Om;d^nx))^n}
+\|V^{1/2}u\|^2_{(L^2(\Om;d^nx))^n}\bigr]
=\lambda\, (u,(-\Delta+V)u)_{L^2(\Om;d^nx)}
\nonumber\\
& \quad
=(u\,,\,\lambda\,(-\Delta+V)u)_{L^2(\Om;d^nx)}
= \big(u,(-\Delta+V)^2 u\big)_{L^2(\Om;d^nx)}
\nonumber\\[4pt]
& \quad
=(u,(-\Delta+V)(-\Delta u+Vu))_{L^2(\Om;d^nx)}
=((-\Delta+V)u,(-\Delta+V)u)_{L^2(\Om;d^nx)}
\nonumber\\[4pt]
& \quad
=\|(-\Delta+V)u\|^2_{L^2(\Om;d^nx)}.
\end{align}
Since, according to Theorem \ref{tH.A},
$L^2(\Om;d^nx)\ni u\not=0$ and $\widehat\ga_D u =0$ prevent $u$ from being a
constant function, \eqref{MM-3} entails
\begin{equation} \label{MM-4}
\lambda=\frac{\|(-\Delta+V)u\|^2_{L^2(\Om;d^nx)}}
{\|\nabla u\|^2_{(L^2(\Om;d^nx))^n}+\|V^{1/2}u\|^2_{(L^2(\Om;d^nx))^n}}>0,
\end{equation} 
as desired.
\end{proof}
%%%%%%%%%%%%%%%%%%%%%%%%%%%%%%%%%%%%%

Next, we recall the operator $P_{D,\Om,V,z}$ introduced just above \eqref{3.34Y}
and agree to simplify notation by abbreviating $P_{D,\Om,V}:=P_{D,\Om,V,0}$.
That is,
\begin{equation} \label{3.34Yz}
P_{D,\Om,V}=\big[\ga_N (H_{D,\Om})^{-1}\big]^*
\in\cB\big((N^{1/2}(\partial\Omega))^*,L^2(\Om;d^nx)\big)
\end{equation}
is such that if $u:=P_{D,\Om,V} g$ for some
$g\in\bigl(N^{1/2}(\partial\Omega)\bigr)^*$, then
\begin{equation}\label{Yan-14z}
\begin{cases}
(-\Delta+V)u=0\text{ in }\,\Om,
\\[4pt]
u\in L^2(\Om;d^nx),
\\[4pt]
\widehat\ga_D u =g\text{ on }\,\dOm.
\end{cases}
\end{equation}
Hence,
\begin{align}\label{3.Gv}
\begin{split} 
& (-\Delta+V) P_{D,\Om,V}=0,  \\ 
& \widehat\ga_N P_{D,\Om,V}=-M_{D,N,\Omega,V}(0)
\, \mbox{ and }\, 
\widehat\ga_D P_{D,\Om,V}=I_{(N^{1/2}(\dOm))^*},
\end{split} 
\end{align}
with $I_{(N^{1/2}(\dOm))^*}$ the identity operator, on $\bigl(N^{1/2}(\dOm)\bigr)^*$.

%%%%%%%%%%%%%%%%%%%%%%%%%%%%%%%%%%%% 
\begin{theorem}\label{T-MM-1}
Assume Hypothesis \ref{h.VK}.
If $0\not=v\in L^2(\Om;d^nx)$ is an eigenfunction of the
perturbed Krein Laplacian $H_{K,\Om}$ corresponding to
the eigenvalue $0\not=\lambda\in\bbC$ $($hence $\lambda>0$$)$, then
\begin{equation} \label{MM-5}
u:=v-P_{D,\Om,V}(\widehat\ga_D v)
\end{equation} 
is a nontrivial solution of \eqref{MM-1}. Conversely, if
$0\not=u\in L^2(\Om;d^nx)$ solves \eqref{MM-1} for some $\lambda\in\bbC$ then
$\lambda$ is a $($strictly$)$ positive eigenvalue of the perturbed Krein Laplacian
$H_{K,\Om}$, and
\begin{equation} \label{MM-6}
v:=\lambda^{-1}(-\Delta+V)u
\end{equation} 
is a nonzero eigenfunction of the
perturbed Krein Laplacian, corresponding to this eigenvalue.
\end{theorem}
%%%%%%%%%%%%%%%%%%%%%%%%%%%%%%%%%%%% 
\begin{proof}
In one direction, assume that $0\not=v\in L^2(\Om;d^nx)$ is an
eigenfunction of the perturbed Krein Laplacian $H_{K,\Om}$
corresponding to the eigenvalue $0\not=\lambda\in\bbC$
(since $H_{K,\Om}\geq 0$
--\,cf.\ Theorem \ref{T-Kr}-- it follows that $\lambda>0$).
Thus, $v$ satisfies
\begin{equation} \label{MM-7}
v\in\dom(H_{max,\Om} ),\quad
(-\Delta+V)v=\lambda\,v,\; 
\tau_{N,V,0} v =0.
\end{equation} 
In particular, $\widehat\ga_D v \in\bigl(N^{1/2}(\dOm)\bigr)^*$ by
Theorem \ref{New-T-tr}. Hence, by \eqref{3.34Yz}, $u$ in \eqref{MM-5} is
a well-defined function which belongs to $L^2(\Om;d^nx)$. In fact,
since also $(-\Delta+V)u=(-\Delta+V)v\in L^2(\Om;d^nx)$, it follows that
$u\in\dom(H_{max,\Om} )$. Going further, we note that
\begin{align} \label{MM-8}
\begin{split} 
(-\Delta+V)^2u &=(-\Delta+V)(-\Delta+V)u
= (-\Delta+V)(-\Delta+V)v   \\
&= \lambda\,(-\Delta+V)v=\lambda\,(-\Delta+V)u.
\end{split} 
\end{align} 
Hence, $(-\Delta+V)^2u=\lambda\,(-\Delta+V)u$ in $\Om$. In addition, by \eqref{3.Gv},
\begin{equation} \label{MM-9}
\widehat\ga_D u =\widehat\ga_D v
-\widehat\ga_D(P_{D,\Om,V}(\widehat\ga_D v )
=\widehat\ga_D v -\widehat\ga_D v =0,
\end{equation} 
whereas
\begin{equation} \label{MM-10}
\widehat\ga_N u =\widehat\ga_N v 
-\widehat\ga_N(P_{D,\Om,V}(\widehat\ga_D v )
=\widehat\ga_N v +M_{D,N,\Om,V}(0)(\widehat\ga_D v )
=\tau_{N,V,0} v=0,
\end{equation} 
by the last condition in \eqref{MM-7}.
Next, to see that $u$ cannot vanish identically,
we note that $u=0$ would imply $v=P_{D,\Om,V}(\widehat\ga_D v)$
which further entails $\lambda\,v=(-\Delta+V)v
=(-\Delta+V)P_{D,\Om,V}(\widehat\ga_D v)=0$, that is, 
$v=0$ (since $\lambda\not=0$). This contradicts the original assumption
on $v$ and shows that $u$ is a nontrivial solution of \eqref{MM-1}.
This completes the proof of the first half of the theorem.

Turning to the second half, suppose that $\lambda\in\bbC$ and
$0\not=u\in L^2(\Om;d^nx)$ is a solution of \eqref{MM-1}. Lemma \ref{L-MM-1}
then yields $\lambda>0$, so that $v:=\lambda^{-1}(-\Delta+V)u$ is a
well-defined function satisfying
\begin{equation} \label{MM-11}
v\in\dom(H_{max,\Om} )\, \mbox{ and }\, 
(-\Delta+V)v=\lambda^{-1}\,(-\Delta+V)^2u=(-\Delta+V)u=\lambda\,v.
\end{equation} 
If we now set $w:=v-u\in L^2(\Omega;d^nx)$ it follows that
\begin{equation} \label{MM-12}
(-\Delta+V)w=(-\Delta+V)v-(-\Delta+V)u=\lambda\,v-\lambda\,v=0,
\end{equation} 
and
\begin{equation} \label{MM-13}
\widehat\ga_N w =\widehat\ga_N v,\quad
\widehat\ga_D w =\widehat\ga_D v.
\end{equation} 
In particular, by the uniqueness in the Dirichlet problem \eqref{Yan-14z},
\begin{equation} \label{MM-14}
w=P_{D,\Om,V}(\widehat\ga_D v).
\end{equation} 
Consequently,
\begin{equation} \label{MM-15}
\widehat\ga_N v=\widehat\ga_N w
=\widehat\ga_N(P_{D,\Om,V}(\widehat\ga_D v)
=-M_{D,N,\Om,V}(0)(\widehat\ga_D v),
\end{equation} 
which shows that
\begin{equation} \label{MM-16}
\tau_{N,V,0} v=\widehat\ga_N v +M_{D,N,\Om,V}(0)(\widehat\ga_D v)=0.
\end{equation} 
Hence $v\in\dom(H_{K,\Om})$. We note that $v=0$ would entail that the
function $u\in H^2_0(\Omega)$ is a null solution of $-\Delta+V$, hence
identically zero which, by assumption, is not the case. Therefore, $v$ does
not vanish identically. Altogether, the above reasoning shows that $v$ is
a nonzero eigenfunction of the perturbed Krein Laplacian,
corresponding to the positive eigenvalue $\lambda >0$, completing the proof. 
\end{proof}
%%%%%%%%%%%%%%%%%%%%%%%%%%%%%%%%%%%%%

%%%%%%%%%%%%%%%%%%%%%%%%%%%%%%%%%%%%%
\begin{proposition} \lb{pHKv} 
$(i)$ Assume Hypothesis \ref{h.VK} and let $0\neq v$ be any eigenfunction of 
$H_{K,\Om}$ corresponding to the eigenvalue 
$0 \neq \lambda \in \sigma(H_{K,\Om})$. In addition suppose that the operator of multiplication 
by $V$ satisfies 
\begin{equation}\label{MUL}
M_V\in\cB\bigl(H^2(\Om),H^s(\Om)\bigr) \, \mbox{ for some } \, 1/2<s\leq 2.
\end{equation}
Then $u$ defined in \eqref{MM-5} satisfies 
\begin{equation}
u \in H^{5/2}(\Om), \, \text{ implying } \, v \in H^{1/2}(\Om).   \lb{Kv}
\end{equation}
$(ii)$ Assume the smooth case, that is, $\partial\Omega$ is $C^\infty$ and 
$V\in C^\infty(\ol\Om)$, and let $0 \neq v$ be any eigenfunction of $H_{K,\Om}$ corresponding to the eigenvalue $0 \neq \lambda \in \sigma(H_{K,\Om})$. 
Then $u$ defined in 
\eqref{MM-5} satisfies 
\begin{equation}
u\in C^\infty(\ol \Om), \, \text{ implying } \, v\in C^\infty(\ol \Om).  \lb{KvC}
\end{equation}
\end{proposition}
%%%%%%%%%%%%%%%%%%%%%%%%%%%%%%%%%%%%%
\begin{proof}
$(i)$ We note that $u\in L^2(\Om;d^nx)$ satisfies $\widehat\gamma_D(u)=0$, 
$\widehat\gamma_N(u)=0$, and 
$(-\Delta+V)u=(-\Delta+V)v=\lambda v\in L^2(\Omega;d^nx)$. Hence, by Theorems 
\ref{T-DD1} and \ref{tH.A}, we obtain that $u\in H^2_0(\Omega)$. Next, observe that 
$(-\Delta+V)^2u=\lambda^2 v\in L^2(\Omega;d^nx)$ which therefore entails 
$\Delta^2u\in H^{s-2}(\Om)$ by \eqref{MUL}. With this at hand, the regularity results
in \cite{PV95} (cf.\ also \cite{AP98} for related results) yield that 
$u\in H^{5/2}(\Om)$. 

$(ii)$ Given the eigenfunction $0\neq v$ of $H_{K,\Om}$, \eqref{MM-5} yields that $u$ satisfies the 
generalized buckling problem \eqref{MM-1}, so that by elliptic regularity 
$u\in C^\infty(\ol \Om)$. By \eqref{MM-6} and \eqref{MM-7} one thus obtains 
\begin{equation}
\lambda v = (-\Delta +V) v = (-\Delta + V) u, \, \text{ with } \, u\in C^\infty(\ol \Om), 
\end{equation} 
proving \eqref{KvC}. 
\end{proof}
%%%%%%%%%%%%%%%%%%%%%%%%%%%%%%%%%%%%%

In passing, we note that the multiplier condition \eqref{MUL} 
is satisfied, for instance, if $V$ is Lipschitz. 

We next wish to prove that the perturbed Krein Laplacian has
only point spectrum (which, as the previous theorem shows, is directly
related to the eigenvalues of the generalized buckling of the clamped
plate problem). This requires some preparations, and we proceed by
first establishing the following.

%%%%%%%%%%%%%%%%%%%%%%%%%%%%%%%%%%%%% 
\begin{lemma}\label{L-MM-2}
Assume Hypothesis \ref{h.VK}.
Then there exists a discrete subset $\Lambda_{\Om}$ of $(0,\infty)$ without
any finite accumulation points which has the following significance: 
For every $z\in\bbC\backslash \Lambda_{\Om}$ and every
$f\in H^{-2}(\Om)$, the problem
\begin{equation} \label{MM-17}
\begin{cases}
u\in H^2_0(\Omega),
\\
(-\Delta+V)(-\Delta+V-z)u=f \,\mbox{ in } \,\Omega,
\end{cases}
\end{equation}
has a unique solution. In addition, there exists $C=C(\Omega,z)>0$
such that the solution satisfies
\begin{equation} \label{MM-18}
\|u\|_{H^2(\Omega)}\leq C\|f\|_{H^{-2}(\Omega)}.
\end{equation} 

Finally, if $z\in\Lambda_{\Om}$, then there exists $u\not=0$ satisfying
\eqref{MM-2}. In fact, the space of solutions for the
problem \eqref{MM-2} is, in this case, finite-dimensional and nontrivial.
\end{lemma}
%%%%%%%%%%%%%%%%%%%%%%%%%%%%%%%%%%
\begin{proof}
In a first stage, fix $z\in\bbC$ with $\Re(z)\leq -M$,
where $M=M(\Om,V)>0$ is a large constant to be specified later, and
consider the bounded sesquilinear form
\begin{align}\label{MM-19} 
& a_{V,z}(\dott,\dott):H^2_0(\Om)\times H^2_0(\Om)\to \bbC,  \no
\\
& a_{V,z}(u,v):= ((-\Delta+V)u,(-\Delta+V)v)_{L^2(\Om;d^nx)}
+ \big(V^{1/2}u,V^{1/2}v\big)_{L^2(\Om;d^nx)}
\\
& \hskip 0.77in
-z\, (\nabla u,\nabla v)_{(L^2(\Om;d^nx))^n},\quad
u,v\in H^2_0(\Om).  \no 
\end{align}
Then, since $f\in H^{-2}(\Om)=\bigl(H^2_0(\Om)\bigr)^*$, the well-posedness
of \eqref{MM-17} will follow with the help of the Lax-Milgram lemma as soon as
we show that \eqref{MM-19} is coercive. To this end, observe that
via repeated integrations by parts
\begin{align}\label{MM-20}
\begin{split} 
a_{V,z}(u,u) &= \sum_{j,k=1}^n\int_{\Omega}d^nx\,\Big|
\frac{\partial^2 u}{\partial x_j\partial x_k}\Big|^2
-z \sum_{j=1}^n\int_{\Omega}d^nx\,
\Big|\frac{\partial u}{\partial x_j}\Big|^2
\\
& \quad +\int_{\Omega}d^nx\,\big|V^{1/2}u\bigr| 
+2 \Re\bigg(\int_{\Omega}d^nx\,\Delta u\,V\ol{u}\bigg), \quad u\in C^\infty_0(\Om).
\end{split} 
\end{align}
We note that the last term is of the order
\begin{equation} \label{MM-20U}
O\bigl(\|V\|_{L^\infty(\Om;d^nx)}\|\Delta u\|_{L^2(\Om;d^nx)}
\|u\|_{L^2(\Om;d^nx)}\bigr)
\end{equation} 
and hence, can be dominated by
\begin{equation} \label{MM-21U}
C\|V\|_{L^\infty(\Om;d^nx)}\big[\varepsilon\|u\|^2_{H^2(\Om)}
+(4\varepsilon)^{-1}\|u\|^2_{L^2(\Om;d^nx)}\big],
\end{equation} 
for every $\varepsilon>0$. Thus, based on this and Poincar\'e's inequality,
we eventually obtain, by taking $\varepsilon>0$ sufficiently small, and
$M$ (introduced in the beginning of the proof) sufficiently large, that
\begin{equation} \label{MM-21}
\Re (a_{V,z}(u,u)) \geq C\|u\|^2_{H^2(\Om)},\quad 
u\in C^\infty_0(\Om).
\end{equation} 
Hence,
\begin{equation} \label{MM-22}
\Re (a_{V,z}(u,u)) \geq C\|u\|^2_{H^2(\Om)},\quad 
u\in H^2_0(\Om),
\end{equation} 
by the density of $C^\infty_0(\Om)$ in $H^2_0(\Om)$. Thus, the form
\eqref{MM-20} is coercive and hence, the problem \eqref{MM-17} is well-posed
whenever $z\in\bbC$ has $\Re(z)\leq -M$.

We now wish to extend this type of conclusion to a larger set of
$z$'s. With this in mind, set
\begin{equation} \label{Mi-1}
A_{V,z}:=(-\Delta+V)(-\Delta+V-z I_{\Om})\in
\cB\bigl(H^2_0(\Om),H^{-2}(\Om)\bigr),\quad z\in\bbC.
\end{equation} 
The well-posedness of \eqref{MM-17} is equivalent to the fact that the
above operator is invertible. In this vein, we note that if
we fix $z_0 \in\bbC$ with $\Re(z_0 )\leq -M$, then, from what
we have shown so far,
\begin{equation} \label{Mi-2}
A_{V,z_0 }^{-1}\in\cB\bigl(H^{-2}(\Om),H^2_0(\Om)\bigr)
\end{equation} 
is a well-defined operator. For an arbitrary $z\in\bbC$ we then write
\begin{equation} \label{Mi-3}
A_{V,z}=A_{V,z_0 }[I_{H^2_0(\Om)}+B_{V,z}],
\end{equation} 
where $I_{H^2_0(\Om)}$ is the identity operator on $H^2_0(\Om)$ and we have set
\begin{equation} \label{Mi-4}
B_{V,z}:=A_{V,z_0 }^{-1}(A_{V,z}-A_{V,z_0 })
=(z_0 -z)A_{V,z_0 }^{-1}(-\Delta+V)
\in\cB_\infty\bigl(H^2_0(\Om)\bigr).
\end{equation} 
Since $\bbC\ni z\mapsto B_{V,z}\in\cB\bigl(H^2_0(\Om)\bigr)$
is an analytic, compact operator-valued mapping, which vanishes for
$z=z_0 $, the Analytic Fredholm Theorem yields the existence
of an exceptional, discrete set $\Lambda_{\Om}\subset\bbC$, without
any finite accumulation points such that
\begin{equation} \label{Mi-5}
(I_{H^2_0(\Om)}+B_{V,z})^{-1}\in\cB\bigl(H^2_0(\Om)\bigr),\quad 
z\in\bbC\backslash \Lambda_{\Om}.
\end{equation} 
As a consequence of this, \eqref{Mi-2}, and \eqref{Mi-3}, we therefore have
\begin{equation} \label{Mi-6}
A_{V,z}^{-1}\in\cB\bigl(H^{-2}(\Om),H^2_0(\Om)\bigr),\quad 
z\in\bbC\backslash \Lambda_{\Om}.
\end{equation} 
We now proceed to show that, in fact, $\Lambda_{\Om}\subset(0,\infty)$.
To justify this inclusion, we observe that
\begin{equation} \label{Mi-6X}
\text{$A_{V,z}$ in \eqref{Mi-1} is a Fredholm operator,
with Fredholm index zero, for every $z\in\bbC$},
\end{equation} 
due to \eqref{Mi-2}, \eqref{Mi-3}, and \eqref{Mi-4}. Thus, if for some
$z\in\bbC$ the operator $A_{V,z}$ fails to be invertible, then
there exists $0\not=u\in L^2(\Om;d^nx)$ such that $A_{V,z}u=0$.
In view of \eqref{Mi-1} and Lemma \ref{L-MM-1}, the latter condition
forces $z\in(0,\infty)$. Thus, $\Lambda_{\Om}$ consists of positive
numbers. At this stage, it remains to justify the very last claim in
the statement of the lemma. This, however, readily follows from \eqref{Mi-6X},
completing the proof.
\end{proof}
%%%%%%%%%%%%%%%%%%%%%%%%%%%%%%%%%%%%%%%%

%%%%%%%%%%%%%%%%%%%%%%%%%%%%%%%%%%%%%%%%
\begin{theorem}\label{T-MM-2}
Assume Hypothesis \ref{h.VK} and recall the
exceptional set $\Lambda_{\Om}\subset(0,\infty)$ from Lemma \ref{L-MM-2},
which is discrete with only accumulation point at infinity. Then
\begin{equation} \label{Mi-7}
\sigma(H_{K,\Om})=\Lambda_{\Om}\cup\{0\}.
\end{equation} 
Furthermore, for every $0\not=z \in\bbC\backslash \Lambda_{\Om}$, the
action of the resolvent $(H_{K,\Om}-z I_{\Om})^{-1}$
on an arbitrary element $f\in L^2(\Om;d^nx)$ can be described as follows: 
Let $v$ solve
\begin{equation}\label{MM-23} 
\begin{cases}
v\in H^2_0(\Omega),
\\
(-\Delta+V)(-\Delta+V-z)v=(-\Delta+V)f\in H^{-2}(\Omega),
\end{cases} 
\end{equation}
and consider
\begin{equation} \label{MM-24}
w:=z^{-1}[(- \Delta+V-z)v-f]\in L^2(\Om;d^nx).
\end{equation} 
Then
\begin{equation} \label{MM-24X}
(H_{K,\Om}-z I_{\Om})^{-1}f=v+w.
\end{equation} 

Finally, every $z \in \Lambda_{\Om}\cup\{0\}$ is actually an eigenvalue
$($of finite multiplicity, if nonzero$)$ for the perturbed Krein Laplacian,
and the essential spectrum of this operator is given by
\begin{equation} \label{Mi-7S}
\sigma_{ess}(H_{K,\Om})=\{0\}.
\end{equation} 
\end{theorem}
%%%%%%%%%%%%%%%%%%%%%%%%%%%%%%%%%%
\begin{proof}
Let $0\not=z \in\bbC\backslash \Lambda_{\Om}$, fix $f\in L^2(\Om;d^nx)$,
and assume that $v,w$ are as in the statement of the theorem. That $v$
(hence also $w$) is well-defined follows from Lemma \ref{L-MM-2}. Set
\begin{align}\label{MM-26}
u :=  v+w \in & H^2_0(\Om)\dot{+}
\big\{\eta\in L^2(\Om;d^nx)\,\big|\,(-\Delta+V)\eta=0\mbox{ in }\Omega\big\}
\nonumber\\
& \quad = \ker\big(\tau_{N,V,0}\big)\hookrightarrow \dom(H_{max,\Om} ),
\end{align}
by \eqref{3.AKe}. Thus, $u\in \dom(H_{max,\Om} )$ and
$\tau_{N,V,0} u=0$ which force $u\in\dom(H_{K,\Om})$.
Furthermore,
\begin{equation} \label{Mi-8}
\|u\|_{L^2(\Om;d^nx)}+\|\Delta u\|_{L^2(\Om;d^nx)}
\leq C\|f\|_{L^2(\Om;d^nx)},
\end{equation} 
for some $C=C(\Om,V,z)>0$, and
\begin{align}\label{MM-27}
& (-\Delta+V-z)u = (-\Delta+V-z)v+(-\Delta+V-z)w
\nonumber\\
& \quad = (-\Delta+V-z)v
+z^{-1}(-\Delta+V-z)[(-\Delta+V-z)v-f]
\nonumber\\
& \quad = (-\Delta+V-z)v+z^{-1}(-\Delta+V)[(-\Delta+V-z)v-f]
-[(-\Delta+V-z)v-f]
\nonumber\\
& \quad = f+z^{-1}[(-\Delta+V)(-\Delta+V-z)v-(-\Delta+V)f]=f,
\end{align}
by \eqref{MM-23}, \eqref{MM-24}. As a consequence of this analysis,
we may conclude that the operator
\begin{equation} \label{Mi-9}
H_{K,\Om}-z I_{\Om}:\dom(H_{K,\Om})\subset L^2(\Om;d^nx)
\to L^2(\Om;d^nx)
\end{equation} 
is onto (with norm control), for every
$z \in\bbC\backslash (\Lambda_{\Om}\cup\{0\})$. When
$z \in\bbC\backslash (\Lambda_{\Om}\cup\{0\})$
the last part in Lemma \ref{L-MM-2}
together with Theorem \ref{T-MM-1} also yield that the operator
\eqref{Mi-9} is injective. Together, these considerations prove that
\begin{equation} \label{Mi-7X}
\sigma(H_{K,\Om})\subseteq\Lambda_{\Om}\cup\{0\}.
\end{equation} 
Since the converse inclusion also follows from the
last part in Lemma \ref{L-MM-2} together with Theorem \ref{T-MM-1},
equality \eqref{Mi-7} follows. Formula \eqref{MM-24X}, along with
the final conclusion in the statement of the theorem, is also implicit
in the above analysis plus the fact that $\ker(H_{K,\Om})$
is infinite-dimensional (cf.\ \eqref{dim} and \cite{MT00}).
\end{proof}
%%%%%%%%%%%%%%%%%%%%%%%%%%%%%%%%%%%%%%%%%%

%%%%%%%%%%%%%%%%%%%%%%%%%%%%%%%%%%%%%%
%%%%%%%%%%%%%%%%%%%%%%%%%%%%%%%%%%%%%%
\section{Eigenvalue Estimates for the Perturbed Krein Laplacian}
\label{s11}
%%%%%%%%%%%%%%%%%%%%%%%%%%%%%%%%%%%%%%
%%%%%%%%%%%%%%%%%%%%%%%%%%%%%%%%%%%%%%

The aim of this section is to study in greater detail the nature of the
spectrum of the operator $H_{K,\Om}$. We split the discussion into
two separate cases, dealing with the situation when the potential $V$
is as in Hypothesis \ref{h.V} (Subsection \ref{s11X}), and when $V\equiv 0$
(Subsection \ref{s11Y}).

%%%%%%%%%%%%%%%%%%%%%%%%%%%%%%%
\subsection{The Perturbed Case}
\label{s11X}
%%%%%%%%%%%%%%%%%%%%%%%%%%%%%%%

Given a domain $\Omega$ as in Hypothesis \ref{h.Conv} and a potential $V$
as in Hypothesis \ref{h.V}, we recall the exceptional set
$\Lambda_{\Om}\subset(0,\infty)$ associated with $\Omega$ as in
Section \ref{s10}, consisting of numbers
\begin{equation} \label{mam-1}
0<\lambda_{K,\Om,1}\leq\lambda_{K,\Om,2}\leq\cdots\leq\lambda_{K,\Om,j}
\leq\lambda_{K,\Om,j+1}\leq\cdots
\end{equation} 
converging to infinity. Above, we have displayed the $\lambda$'s
according to their (geometric) multiplicity which equals 
the dimension of the kernel of the (Fredholm) operator \eqref{Mi-1}. 

%%%%%%%%%%%%%%%%%%%%%%%%%%%%%%%%%
\begin{lemma}\label{T-MAM-1}
Assume Hypothesis \ref{h.VK}.
Then there exists a family of functions $\{u_j\}_{j\in\bbN}$ with the
following properties:
\begin{align}\label{mam-2}
& u_j\in H^2_0(\Om)\, \mbox{ and }\, 
(-\Delta+V)^2u_j=\lambda_{K,\Om,j}(-\Delta+V)u_j,   \quad  j\in\bbN,
\\
& ((-\Delta+V)u_j,(-\Delta+V)u_k)_{L^2(\Om;d^nx)}=\delta_{j,k},  \quad j,k\in\bbN,
\label{mam-3}\\
& u=\sum_{j=1}^\infty ((-\Delta+V)u,(-\Delta+V)u_j)_{L^2(\Om;d^nx)}\,u_j,
\quad u\in H^2_0(\Om),
\label{mam-4}
\end{align}
with convergence in $H^2(\Om)$.
\end{lemma}
%%%%%%%%%%%%%%%%%%%%%%%%%%%%%%%%%%%%%%
\begin{proof}
Consider the vector space and inner product
\begin{equation} \label{mam-5}
\cH_V:=H^2_0(\Om),\quad
[u,v]_{\cH_V}:=\int_{\Om}d^nx\,\ol{(-\Delta+V)u}\,(-\Delta+V)v,
\quad u,v\in\cH_V.
\end{equation} 
We claim that $\bigl(\cH_V,[\dott,\dott]_{\cH_V}\bigr)$ is a Hilbert space.
This readily follows as soon as we show that
\begin{equation} \label{mam-6}
\|u\|_{H^2(\Om)}\leq C\|(-\Delta+V)u\|_{L^2(\Om;d^nx)},\quad u\in H^2_0(\Om),
\end{equation} 
for some finite constant $C=C(\Om,V)>0$. To justify this, observe that
for every $u\in C^\infty_0(\Om)$ we have
\begin{align}\label{mam-7}
\int_{\Omega}d^nx\,|u|^2 &\leq  C \sum_{j=1}^n\int_{\Omega}d^nx\,
\Big|\frac{\partial u}{\partial x_j}\Big|^2
\nonumber\\
&\leq  C\sum_{j,k=1}^n\int_{\Omega}d^nx\,\Big|
\frac{\partial^2 u}{\partial x_j\partial x_k}\Big|^2
=\int_{\Omega}d^nx\,|\Delta u|^2,
\end{align}
where we have used Poincar\'e's inequality in the first two steps.
Based on this, the fact that $V$ is bounded, and the density of
$C^\infty_0(\Om)$ in $H^2_0(\Om)$ we therefore have
\begin{equation} \label{mam-6Y}
\|u\|_{H^2(\Om)}\leq C\bigl(\|(-\Delta+V)u\|_{L^2(\Om;d^nx)}
+\|u\|_{L^2(\Om;d^nx)}\bigr),\quad u\in H^2_0(\Om),
\end{equation} 
for some finite constant $C=C(\Om,V)>0$. Hence, the operator
\begin{equation} \label{mam-6YY}
-\Delta+V\in\cB\bigl(H^2_0(\Om),L^2(\Om;d^nx)\bigr)
\end{equation} 
is bounded from below modulo compact operators, since the embedding
$H^2_0(\Om)\hookrightarrow L^2(\Om;d^nx)$ is compact.
Hence, it follows that \eqref{mam-6YY} has closed range.
Since this operator is also one-to-one (as $0\not\in\sigma(H_{D,\Om})$),
estimate \eqref{mam-6} follows from the Open Mapping Theorem.
This shows that
\begin{equation} \label{mam-8}
\cH_V=H^2_0(\Om)\, \mbox{ as Banach spaces, with equivalence of norms}.
\end{equation} 
Next, we recall from the proof of Lemma \ref{L-MM-2}
that the operator \eqref{Mi-1} is invertible for
$\lambda\in\bbC\backslash \Lambda_{\Om}$ (cf.\ \eqref{Mi-6}), and that
$\Lambda_{\Om}\subset(0,\infty)$. Taking $\lambda=0$ this shows that
\begin{equation} \label{mam-9}
(-\Delta+V)^{-2}:=((-\Delta+V)^2)^{-1}\in\cB\bigl(H^{-2}(\Om),H^2_0(\Om)\bigr)
\end{equation} 
is well-defined. Furthermore, this operator is self-adjoint (viewed as
a linear, bounded operator mapping a Banach space into its dual, cf.\ \eqref{B.5}).
Consider now
\begin{equation} \label{mam-10}
B:=-(-\Delta+V)^{-2}(-\Delta+V).
\end{equation} 
Since $B$ admits the factorization
\begin{equation} \label{mam-11}
B:H^2_0(\Om)\stackrel{-\Delta+V}
{-\!\!\!-\!\!\!-\!\!\!\longrightarrow}
L^2(\Om;d^nx)\stackrel{\iota}{\hookrightarrow}
H^{-2}(\Om)\stackrel{-(-\Delta+V)^{-2}}
{-\!\!\!-\!\!\!-\!\!\!-\!\!\!-\!\!\!-\!\!\!\longrightarrow}H^2_0(\Om),
\end{equation} 
where the middle arrow is a compact inclusion, it follows that
\begin{equation} \label{mam-12}
B\in\cB(\cH_V)\, \mbox{ is compact and injective}.
\end{equation} 
In addition, for every $u,v\in C^\infty_0(\Om)$ we have via repeated
integrations by parts
\begin{align}\label{mam-13}
[Bu,v]_{\cH_V}&=
- \big((-\Delta+V)(-\Delta+V)^{-2}(-\Delta+V)u,
(-\Delta+V)v\big)_{L^2(\Om;d^nx)}
\nonumber\\
&= -\big((-\Delta+V)^{-2}(-\Delta+V)u,(-\Delta+V)^2v\big)_{L^2(\Om;d^nx)}
\nonumber\\
&= - \big((-\Delta+V)u,(-\Delta+V)^{-2}(-\Delta+V)^2v\big)_{L^2(\Om;d^nx)}
\nonumber\\
&= -((-\Delta+V)u,v)_{L^2(\Om;d^nx)}
\nonumber\\
&= -(\nabla u,\nabla v)_{(L^2(\Om;d^nx))^n}
- \big(V^{1/2}u, V^{1/2}v\big)_{L^2(\Om;d^nx)}.
\end{align}
Consequently, by symmetry,
$[Bu,v]_{\cH_V}=\ol{[Bv,u]_{\cH_V}}$, $u,v\in C^\infty_0(\Om)$
and hence,
\begin{equation} \label{mam-14}
[Bu,v]_{\cH_V}=\ol{[Bv,u]_{\cH_V}} \quad u,v\in\cH_V,
\end{equation} 
since $C^\infty_0(\Om)\hookrightarrow\cH_V$ densely. Thus,
\begin{equation} \label{mam-15}
B\in\cB_\infty (\cH_V)\, \mbox{ is self-adjoint and injective}.
\end{equation} 
To continue, we recall the operator $A_{V,\lambda}$ from \eqref{Mi-1}
and observe that
\begin{equation} \label{mam-16}
(-\Delta+V)^{-2}A_{V,z}
=I_{\cH_V}- z B, \quad z\in\bbC,
\end{equation} 
as operators in $\cB\bigl(H^2_0(\Om)\bigr)$. Thus, the spectrum
of $B$ consists (including multiplicities) precisely of the reciprocals of
those numbers $z\in\bbC$ for which the operator
$A_{V,z}\in\cB\bigl(H^2_0(\Om),H^{-2}(\Om)\bigr)$ fails to be invertible.
In other words, the spectrum of $B\in\cB(\cH_V)$ is given by
\begin{equation} \label{mam-17}
\sigma(B)=\{(\lambda_{K,\Om,j})^{-1}\}_{j\in\bbN}.
\end{equation} 
Now, from the spectral theory of compact, self-adjoint (injective) operators
on Hilbert spaces (cf., e.g., \cite[Theorem 2.36]{Mc00}), it follows
that there exists a family of functions $\{u_j\}_{j\in\bbN}$ for which
\begin{align}\label{mam-18}
& u_j\in\cH_V\, \mbox{ and }\, 
Bu_j=(\lambda_{K,\Om,j})^{-1}u_j,   \quad j\in\bbN,
\\
& [u_j,u_k]_{\cH_V}=\delta_{j,k},  \quad  j,k\in\bbN,
\label{mam-19}\\
& u=\sum_{j=1}^\infty[u,u_j]_{\cH_V}\,u_j,   \quad u\in\cH_V,
\label{mam-20}
\end{align}
with convergence in $\cH_V$. Unraveling notation,
\eqref{mam-2}--\eqref{mam-4} then readily follow from
\eqref{mam-18}--\eqref{mam-20}.  
\end{proof}
%%%%%%%%%%%%%%%%%%%%%%%%%%%%%%%%%%%%%%

%%%%%%%%%%%%%%%%%%%%%%%%%%%%%%%%%%%%%%%
\begin{remark}
We note that Lemma \ref{T-MAM-1} gives the orthogonality of the eigenfunctions $u_j$ in terms of the inner product for $\cH_V$ (cf.\ \eqref{mam-3} and \eqref{mam-5}, or see \eqref{mam-19} immediately above).  Here we remark that the given inner product for $\cH_V$ does not correspond to the inner product that has traditionally been used in treating the buckling problem for a clamped plate, even after specializing to the case $V \equiv 0$.  The traditional inner product in that case is the {\it Dirichlet inner product}, defined by 
\begin{equation} 
D(u,v)=\int_\Omega d^n x \, (\nabla u, \nabla v)_{\bbC^n}, \quad u, v \in H^1_0(\Om), 
\end{equation}
where 
$(\cdot,\cdot)_{\bbC^n}$ denotes the usual inner product for elements of $\bbC^n$, conjugate linear in its first entry, linear in its second.  When the potential $V \ge 0$ is included, the appropriate generalization of $D(u,v)$ is the inner product
\begin{equation} 
D_V (u,v) = D(u,v) + \int_\Omega d^nx \, V \ol{u} \, v, \quad u, v \in H^1_0(\Om)  
\end{equation} 
(we recall that throughout this paper $V$ is assumed nonnegative, and hence that this inner product gives rise to a well-defined norm).  Here we observe that orthogonality of the eigenfunctions of the buckling problem in the sense of $\cH_V$ is entirely equivalent to their orthogonality in the sense of $D_V (\cdot,\cdot)$: Indeed, starting from the orthogonality in \eqref{mam-19}, integrating by parts, and using the eigenvalue equation \eqref{mam-2}, one has, for $j \ne k$, 
\begin{align} 
0& = [u_j,u_k]_{\cH_V}=\int_\Omega d^nx \, \ol{(-\Delta+V)u_j}\,(-\Delta+V)u_k 
= \int_\Omega d^nx \, \ol{u_j}\,(-\Delta+V)^2 u_k   \no \\
& = \lambda_k \int_\Omega d^nx \,\ol{u_j}\,(-\Delta+V) u_k 
=\lambda_k \bigg[D(u_j,u_k)+\int_\Omega d^nx \, V \ol{u_j} \, u_k\bigg]    \no \\ 
& = \lambda_k \, D_V(u_j,u_k), \quad u, v \in H^2_0(\Om),
\end{align} 
where $\lambda_k$ is shorthand for $\lambda_{K,\Omega,k}$ of \eqref{mam-1}, the eigenvalue corresponding to the eigenfunction $u_k$ (cf.\ \eqref{mam-2}, which exhibits the eigenvalue equation for the eigenpair $(u_j,\lambda_j)$).  Since all the $\lambda_j$'s considered here are positive (see \eqref{mam-1}), this shows that the family of eigenfunctions $\{u_j\}_{j \in \bbN}$, orthogonal with respect to $[\cdot,\cdot]_{\cH_V}$, is also orthogonal with respect to the ``generalized Dirichlet inner product", 
$D_V (\cdot,\cdot)$.  Clearly, this argument can also be reversed (since all eigenvalues are positive), and one sees that a family of eigenfunctions of the generalized buckling problem orthogonal in the sense of the Dirichlet inner product $D_V (\cdot,\cdot)$ is also orthogonal with respect to the inner product for $\cH_V$, that is, with respect to 
$[\cdot,\cdot]_{\cH_V}$.  On the other hand, it should be mentioned that the normalization of each of the $u_k$'s changes if one passes from one of these inner products to the other, due to the factor of $\lambda_k$ encountered above (specifically, one has $[u_k,u_k]_{\cH_V}=\lambda_k \, D_V (u_k,u_k)$ for each $k$). 
\end{remark}
%%%%%%%%%%%%%%%%%%%%%%%%%%%%%%%%%%%%%%

Next, we recall the following result (which provides a slight variation
of the case $V\equiv 0$ treated in \cite{GM10}).

%%%%%%%%%%%%%%%%%%%%%%%%%%%%%%%%%%%%%%
\begin{lemma}\label{th-CL}
Assume Hypothesis \ref{h.VK}.
Then the subspace $(-\Delta+V)\,H^2_0(\Om)$ is closed in $L^2(\Omega;d^nx)$ and
\begin{equation} \label{Man-2}
L^2(\Omega;d^nx)=\ker(H_{V,\max,\Om}) \oplus \big[(-\Delta+V)\,H^2_0(\Om)\big],
\end{equation} 
as an orthogonal direct sum.
\end{lemma}
%%%%%%%%%%%%%%%%%%%%%%%%%%%%%%%%%%%%%%%%%

Our next theorem shows that there exists a countable family of orthonormal
eigenfunctions for the perturbed Krein Laplacian which span the orthogonal
complement of the kernel of this operator: 

%%%%%%%%%%%%%%%%%%%%%%%%%%%%%%%%%%%%%%
\begin{theorem}\label{TH-Mq1}
Assume Hypothesis \ref{h.VK}.
Then there exists a family of functions $\{w_j\}_{j\in\bbN}$ with the
following properties:
\begin{align}\label{mam-21}
& w_j\in\dom(H_{K,\Om})\cap H^{1/2}(\Om) \, \mbox{ and }\, 
H_{K,\Om}w_j=\lambda_{K,\Om,j} w_j,  \;\;  \lambda_{K,\Om,j}>0, \; j\in\bbN,
\\
& (w_j,w_k)_{L^2(\Om;d^nx)}=\delta_{j,k}, \; j,k\in\bbN,
\label{mam-22}\\
& L^2(\Omega;d^nx)=\ker(H_{K,\Om})\,\oplus\,
\ol{{\rm lin. \, span} \{w_j\}_{j\in\bbN}} \;\, \text{ $($orthogonal direct sum$)$.}
\label{mam-23}
\end{align}
\end{theorem}
%%%%%%%%%%%%%%%%%%%%%%%%%%%%%%%%%%%%%%%%%
\begin{proof}
That $w_j \in H^{1/2}(\Om)$, $j\in\bbN$, follows from Proposition \ref{pHKv}\,$(i)$. 
The rest is a direct consequence of Lemma \ref{th-CL}, the fact that
\begin{equation} \label{mam-24}
\ker(H_{V,\max,\Om})=\big\{u\in L^2(\Omega;d^nx)\,\big|\,(-\Delta+V)u=0\big\}
=\ker(H_{K,\Om}),
\end{equation} 
the second part of Theorem \ref{T-MM-1}, and Lemma \ref{T-MAM-1} in
which we set $w_j:=(-\Delta+V)u_j$, $j\in\bbN$. 
\end{proof}
%%%%%%%%%%%%%%%%%%%%%%%%%%%%%%%%%%%%%%%%

Next, we define the following Rayleigh quotient
\begin{equation} \label{mam-25}
R_{K,\Om}[u]:=\frac{\|(-\Delta+V)u\|^2_{L^2(\Om;d^nx)}}
{\|\nabla u\|^2_{(L^2(\Om;d^nx))^n}+\|V^{1/2}u\|^2_{L^2(\Om;d^nx)}},
\quad 0\not=u\in H^2_0(\Om).
\end{equation} 
Then the following min-max principle holds:

%%%%%%%%%%%%%%%%%%%%%%%%%%%%%%%%%%%%%%
\begin{proposition}\label{TH-Mq2}
Assume Hypothesis \ref{h.VK}. Then
\begin{equation} \label{mam-26}
\lambda_{K,\Om,j}
=\min_{\stackrel{W_j\text{ subspace of }H^2_0(\Om)}{\dim (W_j)=j}}
\Big(\max_{0\not=u\in W_j}R_{K,\Om}[u]\Big),\quad j\in\bbN.
\end{equation} 
As a consequence, given two domains $\Omega$, $\widetilde{\Om}$ as in
Hypothesis \ref{h.Conv} for which $\Omega\subseteq\widetilde{\Om}$, and given
a potential $0\leq \widetilde{V}\in L^\infty(\widetilde{\Om})$, one has 
\begin{equation} \label{mam-2S}
0 < \wti \lambda_{K,\widetilde{\Om},j} \leq \lambda_{K,\Om,j}, \quad j\in\bbN, 
\end{equation} 
where $V:=\widetilde{V}|_{\Om}$, and $\lambda_{K,\Om,j}$ and 
$\wti \lambda_{K,\widetilde{\Om},j}$, $j\in\bbN$, are the eigenvalues corresponding 
to the Krein--von Neumann extensions associated with $\Om, V$ and $\wti\Om, \wti V$, 
respectively.
\end{proposition}
%%%%%%%%%%%%%%%%%%%%%%%%%%%%%%%%%%%%%%%%%
\begin{proof}
Obviously, \eqref{mam-2S} is a consequence of \eqref{mam-26}, so we will
concentrate on the latter. We recall the Hilbert space $\cH_V$ from \eqref{mam-5}
and the orthogonal family $\{u_j\}_{j\in\bbN}$ in 
\eqref{mam-18}--\eqref{mam-20}. Next, consider the following
subspaces of $\cH_V$,
\begin{equation} \label{mam-27}
V_0:=\{0\},\quad
V_j:={\rm lin. \, span} \{u_i\,|\,1\leq i\leq j\},\quad j\in\bbN.
\end{equation} 
Finally, set
\begin{equation} \label{mam-28}
V_j^{\bot}
:=\{u\in\cH\,|\,[u,u_i]_{\cH_V}=0,\,1\leq i\leq j\},\quad j\in\bbN.
\end{equation} 
We claim that
\begin{equation} \label{mam-29}
\lambda_{K,\Om,j}=\min_{0\not=u\in V^{\bot}_{j-1}}R_{K,\Om}[u]
=R_{K,\Om}[u_j],\quad j\in\bbN.
\end{equation} 
Indeed, if $j\in\bbN$ and $u=\sum_{k=1}^\infty c_k u_k \in V^{\bot}_{j-1}$,
then $c_k = 0$ whenever $1\leq k \leq j-1$. Consequently,
\begin{equation} \label{mam-30}
\|(-\Delta+V)u\|^2_{L^2(\Om;d^nx)}
=\biggl\|\sum_{k=j}^\infty c_k (-\Delta+V)u_k \biggr\|^2_{L^2(\Om;d^nx)}
=\sum_{k=j}^\infty |c_k|^2
\end{equation} 
by \eqref{mam-3}, so that
\begin{align}\label{mam-31}
& \|\nabla u\|^2_{(L^2(\Om;d^nx))^n}+\|V^{1/2}u\|^2_{L^2(\Om;d^nx)}
=((-\Delta+V)u,u)_{L^2(\Om;d^nx)}
\nonumber\\
&\quad\quad
=\bigg(\sum_{k=j}^\infty
c_k (-\Delta+V)u_k, u\bigg)_{L^2(\Om;d^nx)}  \no \\
& \qquad =\bigg(\sum_{k=j}^\infty (\lambda_{K,\Om,k})^{-1}c_k (-\Delta+V)^2 u_k, 
u\bigg)_{L^2(\Om;d^nx)}
\nonumber\\
&\quad\quad
=\bigg(\sum_{k=j}^\infty (\lambda_{K,\Om,k})^{-1}c_k (-\Delta+V)u_k, 
(-\Delta+V)u\bigg)_{L^2(\Om;d^nx)}
\nonumber\\
&\quad\quad
=\bigg(\sum_{k=j}^\infty (\lambda_{K,\Om,k})^{-1}c_k (-\Delta+V)u_k, 
\sum_{k=j}^\infty c_k (-\Delta+V)u_k \bigg)_{L^2(\Om;d^nx)}
\nonumber\\
&\quad\quad
=\sum_{k=j}^\infty(\lambda_{K,\Om,k})^{-1}|c_k|^2
\leq (\lambda_{K,\Om,j})^{-1}\sum_{k=j}^\infty|c_k|^2
\nonumber\\
&\quad\quad
=(\lambda_{K,\Om,j})^{-1}\|(-\Delta+V)u\|^2_{L^2(\Om;d^nx)},
\end{align}
where in the third step we have relied on \eqref{mam-2}, and the last step
is based on \eqref{mam-30}.  Thus, $R_{K,\Om}[u]\geq\lambda_{K,\Om,j}$ with
equality if $u=u_j$ (cf.\ the calculation leading up to \eqref{MM-4}).
This proves \eqref{mam-29}.  In fact, the same type of argument as the
one just performed also shows that
\begin{equation} \label{mam-32}
\lambda_{K,\Om,j}=\max_{0\not=u\in V_j}R_{K,\Om}[u]=R_{K,\Om}[u_j],
\quad j\in\bbN.
\end{equation} 
Next, we claim that if $W_j$ is an arbitrary subspace of $\cH$
of dimension $j$ then
\begin{equation} \label{mam-33}
\lambda_{K,\Om,j}\leq\max_{0\not=u\in W_j}R_{K,\Om}[u],\quad j\in\bbN.
\end{equation} 
To justify this inequality, observe that
$W_j\cap V^{\bot}_{j-1}\not=\{0\}$ by dimensional considerations.
Hence, if $0\not=v_j\in W_j\cap V^{\bot}_{j-1}$ then
\begin{equation} \label{mam-34}
\lambda_{K,\Om,j}=\min_{0\not
=u\in V^{\bot}_{j-1}}R_{K,\Om}[u]\leq R_{K,\Om}[v_j]
\leq \max_{0\not=u\in W_j}R_{K,\Om}[u],
\end{equation} 
establishing \eqref{mam-33}. Now formula \eqref{mam-26} readily
follows from this and \eqref{mam-32}.
\end{proof}
%%%%%%%%%%%%%%%%%%%%%%%%%%%%%%%%%%%%%%%%

If $\Omega\subset{\mathbb{R}}^n$ is a bounded Lipschitz domain denote by
\begin{equation} \label{mam-35}
0<\lambda_{D,\Om,1}\leq\lambda_{D,\Om,2}\leq\cdots\leq\lambda_{D,\Om,j}
\leq\lambda_{D,\Om,j+1}\leq\cdots
\end{equation} 
the collection of eigenvalues for the perturbed Dirichlet Laplacian
$H_{D,\Om}$ (again, listed according to their multiplicity).
Then, if $0\leq V\in L^\infty(\Om;d^nx)$, we have the well-known formula
(cf., e.g., \cite{DL90} for the case where $V\equiv 0$)
\begin{equation} \label{mam-37}
\lambda_{D,\Om,j} 
=\min_{\stackrel{W_j\text{ subspace of }H^1_0(\Om)}{\dim (W_j)=j}}
\Big(\max_{0\not=u\in W_j}R_{D,\Om}[u]\Big),\quad j\in\bbN,
\end{equation} 
where $R_{D,\Om}[u]$, the Rayleigh quotient for the perturbed
Dirichlet Laplacian, is given by
\begin{equation} \label{mam-38}
R_{D,\Om}[u]:=\frac{\|\nabla u\|^2_{(L^2(\Om;d^nx))^n}
+\|V^{1/2}u\|^2_{L^2(\Om;d^nx)}}{\|u\|^2_{L^2(\Om;d^nx)}},
\quad 0\not=u\in H^1_0(\Om).
\end{equation} 
From Theorem \ref{AS-thK}, Theorem \ref{t2.5}, and Proposition \ref{L-Fri1}, 
we already know that, granted Hypothesis \ref{h.VK}, the nonzero eigenvalues of 
the perturbed Krein Laplacian are at least as large as the corresponding 
eigenvalues of the perturbed Dirichlet Laplacian.  It is nonetheless of 
interest to provide a direct, analytical proof of this result.  We do so 
in the proposition below.

%%%%%%%%%%%%%%%%%%%%%%%%%%%%%%%%%%%%%
\begin{proposition}\label{TH-Mq3}
Assume Hypothesis \ref{h.VK}. Then
\begin{equation}\label{mam-39}
0 < \lambda_{D,\Om,j}\leq\lambda_{K,\Om,j},\quad j\in\bbN.
\end{equation} 
\end{proposition}
%%%%%%%%%%%%%%%%%%%%%%%%%%%%%%%%%%%%%%%%%
\begin{proof}
By the density of $C^\infty_0(\Om)$ into $H^2_0(\Om)$ and $H^1_0(\Om)$,
respectively, we obtain from \eqref{mam-26} and \eqref{mam-37} that
\begin{align}\label{mam-40}
& \lambda_{K,\Om,j}
=\inf_{\stackrel{W_j\text{ subspace of }C^\infty_0(\Om)}{\dim(W_j)=j}}
\Big(\sup_{0\not=u\in W_j}R_{K,\Om}[u]\Big), 
\\
& \lambda_{D,\Om,j}
=\inf_{\stackrel{W_j\text{ subspace of }C^\infty_0(\Om)}{\dim (W_j)=j}}
\Big(\sup_{0\not=u\in W_j}R_{D,\Om}[u]\Big),
\label{mam-41}
\end{align}
for every $j\in\bbN$. Since, if $u\in C^\infty_0(\Om)$,
\begin{align}\label{mam-42}
& \|\nabla u\|^2_{(L^2(\Om;d^nx))^n}+\|V^{1/2}u\|^2_{L^2(\Om;d^nx)}
= ((-\Delta+V)u,u)_{L^2(\Om;d^nx)}
\nonumber\\[4pt]
& \quad \leq \|(-\Delta+V)u\|_{L^2(\Om;d^nx)}\|u\|_{L^2(\Om;d^nx)},
\end{align}
we deduce that
\begin{equation} \label{mam-43}
R_{D,\Om}[u]\leq R_{K,\Om}[u],
\, \mbox{ whenever }\, 0\not=u\in C^\infty_0(\Om).
\end{equation} 
With this at hand, \eqref{mam-39} follows from \eqref{mam-40}--\eqref{mam-41}.
\end{proof}
%%%%%%%%%%%%%%%%%%%%%%%%%%%%%%%%%%%%%%%%%%

%%%%%%%%%%%%%%%%%%%%%%%%%%%%%%%%%%%%%%%%%%
\begin{remark}\label{RRR-em}
Another analytical approach to \eqref{mam-39} which highlights the
connection between the perturbed Krein Laplacian
and a fourth-order boundary problem is as follows.
Granted Hypotheses \ref{h2.1} and \ref{h.V}, and given
$\lambda\in\bbC$, consider the following eigenvalue problem
\begin{equation}\label{MM-1H}
\begin{cases}
u\in\dom(-\Delta_{max,\Om}),\quad (-\Delta+V)u\in\dom(-\Delta_{max,\Om}),
\\
(-\Delta+V)^2u=\lambda\,(-\Delta+V)u \,\mbox{ in } \,\Omega,
\\
\widehat\ga_D(u)=0 \,\mbox{ in } \,\Bigl(N^{1/2}(\dOm)\Bigr)^*,
\\
\widehat\ga_D((-\Delta+V)u)=0 \,\mbox{ in } \,\Bigl(N^{1/2}(\dOm)\Bigr)^*.
\end{cases} 
\end{equation}
Associated with it is the sesquilinear form
\begin{equation}\label{MM-19X}
\begin{cases}
\widetilde{a}_{V,\lambda}(\dott,\dott):\widetilde{\cH}\times\widetilde{\cH}
\longrightarrow\bbC,\quad\widetilde{\cH}:=H^2(\Om)\cap H^1_0(\Om),
\\
\widetilde{a}_{V,\lambda}(u,v)
:=((-\Delta+V)u,(-\Delta+V)v)_{L^2(\Om;d^nx)}
+\big( V^{1/2}u,V^{1/2}v\big)_{L^2(\Om;d^nx)}
\\
\hskip 0.77in
-\lambda\,(\nabla u,\nabla v)_{(L^2(\Om;d^nx))^n},\quad
u,v\in\widetilde{\cH},
\end{cases}
\end{equation}
which has the property that
\begin{equation} \label{MM-19Y}
u\in\widetilde{\cH} \,\mbox{ satisfies }\,\widetilde{a}_{V,\lambda}(u,v)=0  
\,\mbox{ for every } \,v\in \widetilde{\cH} \,
\text{ if and only if } \, \mbox{$u$ solves \eqref{MM-1H}}.
\end{equation} 
We note that since the operator
$-\Delta+V:H^2(\Om)\cap H^1_0(\Om)\to L^2(\Om;d^nx)$ is an isomorphism,
it follows that $u\mapsto \|(-\Delta+V)u\|_{L^2(\Om;d^nx)}$ is an equivalent
norm on the Banach space $\widetilde{\cH}$, and the form
$\widetilde{a}_{V,\lambda}(\dott,\dott)$ is coercive if $\lambda<-M$,
where $M=M(\Om,V)>0$ is a sufficiently large constant.
Based on this and proceeding as in Section \ref{s10}, it can then be shown that
the problem \eqref{MM-1H} has nontrivial solutions if and only
if $\lambda$ belongs to an exceptional set
$\widetilde{\Lambda}_{\Om,V}\subset(0,\infty)$ which is discrete and
only accumulates at infinity. Furthermore,
$u$ solves \eqref{MM-1H} if and only if $v:=(-\Delta+V)u$ is an eigenfunction
for $H_{D,\Om}$, corresponding to the eigenvalue $\lambda$ and,
conversely, if $u$ is an eigenfunction for $H_{D,\Om}$
corresponding to the eigenvalue $\lambda$, then $u$ solves \eqref{MM-1H}.
Consequently, the problem \eqref{MM-1H} is spectrally equivalent to
$H_{D,\Om}$. From this, it follows that the eigenvalues
$\{\lambda_{D,\Om,j}\}_{j\in\bbN}$ of $H_{D,\Om}$ can be expressed as
\begin{equation} \label{mam-26X}
\lambda_{D,\Om,j}
=\min_{\stackrel{W_j\text{ subspace of }\widetilde{\cH}}{\dim (W_j)=j}}
\Big(\max_{0\not=u\in W_j}R_{K,\Om}[u]\Big),\quad j\in\bbN,
\end{equation} 
where the Rayleigh quotient $R_{K,\Om}[u]$ is as in \eqref{mam-25}.
The upshot of this representation is that it immediately yields
\eqref{mam-39}, on account of \eqref{mam-26} and the fact that
$H^2_0(\Om)\subset\widetilde{\cH}$.
\end{remark}
%%%%%%%%%%%%%%%%%%%%%%%%%%%%%%%%%%%%%%%%%%

Next, let $\Omega$ be as in Hypothesis \ref{h2.1} and
$0\leq V\in L^\infty(\Om;d^nx)$. For $\lambda\in\bbR$ set
\begin{equation} \label{mam-44}
N_{X,\Om}(\lambda)
:=\#\{j\in\bbN\,|\,\lambda_{X,\Om,j}\leq\lambda\},\quad X\in\{D,K\},
\end{equation} 
where $\#S$ denotes the cardinality of the set $S$.

%%%%%%%%%%%%%%%%%%%%%%%%%%%%%%%%%%%%%
\begin{corollary}\label{TH-Mq4}
Assume Hypothesis \ref{h.VK}. Then
\begin{equation} \label{mam-45}
N_{K,\Om}(\lambda)\leq N_{D,\Om}(\lambda),\quad \lambda\in\bbR.
\end{equation} 
In particular,
\begin{equation}\label{mam-46}
N_{K,\Om}(\lambda)=O(\lambda^{n/2})\, \mbox{ as }\, \lambda\to\infty.
\end{equation} 
\end{corollary}
%%%%%%%%%%%%%%%%%%%%%%%%%%%%%%%%%%%%%%%%%
\begin{proof}
Estimate \eqref{mam-45} is a trivial consequence of \eqref{mam-39},
whereas \eqref{mam-46} follows from \eqref{mam-35} and Weyl's asymptotic
formula for the Dirichlet Laplacian in a Lipschitz domain
(cf.\ \cite{BS70} and the references therein for very general results of
this nature).
\end{proof}
%%%%%%%%%%%%%%%%%%%%%%%%%%%%%%%%%%%%%%%%%%

%%%%%%%%%%%%%%%%%%%%%%%%%%%%%%%
\subsection{The Unperturbed Case}
\label{s11Y}
%%%%%%%%%%%%%%%%%%%%%%%%%%%%%%%%%%%%%%

What we have proved in Section \ref{s10} and Section \ref{s11X} shows 
that all known eigenvalue estimates for the (standard) buckling problem
\begin{equation} \label{MM-1F}
u\in H^2_0(\Om),\quad
\Delta^2 u=-\lambda\,\Delta u \,\mbox{ in } \,\Omega,
\end{equation} 
valid in the class of domains described in
Hypothesis \ref{h.Conv}, automatically hold, in the same format,
for the Krein Laplacian (corresponding to $V\equiv 0$).
For example, we have the following result with $\lambda^{(0)}_{K,\Om,j}$, 
$j\in\bbN$, denoting the nonzero eigenvalues of the Krein Laplacian 
$-\Delta_{K,\Om}$ and $\lambda^{(0)}_{D,\Om,j}$, $j\in\bbN$, denoting the 
eigenvalues of the Dirichlet Laplacian $-\Delta_{D,\Om}$:

%%%%%%%%%%%%%%%%%%%%%%%%%%%%%%%%%%%%%%%%%%%
\begin{theorem}\label{TRm-1}
If $\Omega\subset\bbR^n$ is as in Hypothesis \ref{h.Conv}, the nonzero 
eigenvalues of the Krein Laplacian $-\Delta_{K,\Om}$ satisfy
\begin{align}\label{mx-1}
& \lambda^{(0)}_{K,\Om,2}\leq\frac{n^2+8n+20}{(n+2)^2}\lambda^{(0)}_{K,\Om,1},
\\
& \sum_{j=1}^n\lambda^{(0)}_{K,\Om,j+1}< (n+4)\lambda^{(0)}_{K,\Om,1}
-\frac{4}{n+4}(\lambda^{(0)}_{K,\Om,2}-\lambda^{(0)}_{K,\Om,1})
\le (n+4)\lambda^{(0)}_{K,\Om,1},
\label{mx-2}
\\
& \sum_{j=1}^k \big(\lambda^{(0)}_{K,\Om,k+1}-\lambda^{(0)}_{K,\Om,j}\big)^2 
\leq\frac{4(n+2)}{n^2}
\sum_{j=1}^k \big(\lambda^{(0)}_{K,\Om,k+1}-\lambda_{K,0,j}\big)
\lambda^{(0)}_{K,\Om,j},
\quad k\in\bbN,
\label{mx-3}
\end{align}
Furthermore, if $j_{(n-2)/2,1}$ is the first positive zero of the
Bessel function of first kind and order $(n-2)/2$ $($cf.\ \cite[Sect.\ 9.5]{AS72}$)$, 
$v_n$ denotes the volume of the unit ball in $\bbR^n$, and $|\Omega|$ stands 
for the $n$-dimensional Euclidean volume of $\Omega$, then
\begin{equation} \label{mx-4} 
\frac{2^{2/n}j_{(n-2)/2,1}^2v_n^{2/n}}{|\Omega|^{2/n}} < \lambda^{(0)}_{D,\Om,2} 
\leq \lambda^{(0)}_{K,\Om,1}. 
\end{equation} 
\end{theorem}
%%%%%%%%%%%%%%%%%%%%%%%%%%%%%%%%%%%%%%%%%%%
\begin{proof}
With the eigenvalues of the buckling plate problem replacing the
corresponding eigenvalues of the Krein Laplacian, estimates 
\eqref{mx-1}--\eqref{mx-3} have been proved in 
\cite{As99}, \cite{As04}, \cite{As09}, \cite{CY06},  
%\cite{As04}, \cite{As09}, 
and \cite{HY84} (indeed, further strengthenings 
of \eqref{mx-2} are detailed in \cite{As04}, \cite{As09}), whereas the 
respective parts of \eqref{mx-4} are covered by results in \cite{Kra26} 
and \cite{Pa55} (see also \cite{AL96}, \cite{BP63}).
\end{proof}
%%%%%%%%%%%%%%%%%%%%%%%%%%%%%%%%%%%%%%%%%%%

%%%%%%%%%%%%%%%%%%%%%%%%%%%%%%%%%%%%%%%%%%%
\begin{remark}\label{Rm-1}
Given the physical interpretation of the first eigenvalue for \eqref{MM-1F},
it follows that $\lambda^{(0)}_{K,\Om,1}$, the first nonzero eigenvalue for 
the Krein Laplacian $-\Delta_{K,\Om}$, is proportional to the load 
compression at which the plate $\Omega$ (assumed to be as in Hypothesis 
\ref{h.Conv}) buckles.  In this connection, it is worth remembering the 
long-standing conjecture of P\'olya--Szeg{\H o}, to the effect that amongst 
all plates of a given area, the circular one will buckle first (assuming 
all relevant physical parameters being equal).  In \cite{AL96}, the authors 
have given a partial result in this direction which, in terms of the first 
eigenvalue  $\lambda^{(0)}_{K,\Om,1}$ for the Krein Laplacian 
$-\Delta_{K,\Om}$ in a domain $\Om$ as in Hypothesis \ref{h.Conv}, reads
\begin{equation} \label{A-L.1}
\lambda^{(0)}_{K,\Om,1}>\frac{2^{2/n}j_{(n-2)/2,1}^2v_n^{2/n}}{|\Omega|^{2/n}}
=c_n\lambda^{(0)}_{K,\Om^{\#},1}
\end{equation} 
where $\Om^{\#}$ is the $n$-dimensional ball with the same volume as $\Om$, 
and
\begin{equation} \label{A-L.2}
c_n= 2^{2/n}[j_{(n-2)/2,1}/j_{n/2,1}]^2 
=1-(4-{\rm log}\,4)/n+O(n^{-5/3})\to 1 \,\mbox{ as }\, n\to\infty.
\end{equation} 
This result implies an earlier inequality of Bramble and Payne
\cite{BP63} for the two-dimensional case, which reads 
\begin{equation} \label{A-L.3}
\lambda^{(0)}_{K,\Om,1}>\frac{2\pi j_{0,1}^2}{{\rm Area}\,(\Omega)}.
\end{equation} 
\end{remark}
%%%%%%%%%%%%%%%%%%%%%%%%%%%%%%%%%%%%%%%%%%%

Before stating an interesting universal inequality concerning the ratio 
of the first (nonzero) Dirichlet and Krein Laplacian eigenvalues for a 
bounded domain with boundary of nonnegative Gaussian mean curvature 
(which includes, obviously, the case of a bounded convex domain),
%bounded convex domain, 
we recall a well-known result due to Babu{\v s}ka 
and V\'yborn\'y \cite{BV65} concerning domain continuity of Dirichlet 
eigenvalues (see also \cite{BL08}, \cite{BLL08}, \cite{Da03}, \cite{Fu99}, 
\cite{St95}, \cite{We84}, and the literature cited therein):

%%%%%%%%%%%%%%%%%%%%%%%%%%%%%%%%%%%%%%%%%%%
\begin{theorem}\label{tDirichletapprox}
Let $\Om\subset \bbR^n$ be open and bounded, and suppose that $\Om_m\subset 
\Om$, $m\in\bbN$, are open and monotone increasing toward $\Om$, that is, 
\begin{equation}
\Om_m \subset \Om_{m+1} \subset \Om, \; m\in\bbN, \quad 
\bigcup_{m\in\bbN} \Om_m = \Om. 
\end{equation}
In addition, let $-\Delta_{D,\Om_m}$ and $-\Delta_{D,\Om}$ be the Dirichlet 
Laplacians in $L^2(\Om_m;d^n x)$ and $L^2(\Om;d^n x)$ $($cf.\ \eqref{3.QHD}, 
\eqref{3.HDF}$)$, 
and denote their respective spectra by 
\begin{equation}
\sigma(-\Delta_{D,\Om_m}) = \big\{\lambda^{(0)}_{D,\Om_m,j}\big\}_{j\in\bbN}, \; 
m\in\bbN, \, \text{ and } \, 
\sigma(-\Delta_{D,\Om}) = \big\{\lambda^{(0)}_{D,\Om,j}\big\}_{j\in\bbN}. 
\end{equation}
Then, for each $j\in\bbN$,  
\begin{equation} \label{A-P.1a}
\lim_{m\to\infty} \lambda^{(0)}_{D,\Om_m,j} = \lambda^{(0)}_{D,\Om,j}.
\end{equation} 
\end{theorem}
%%%%%%%%%%%%%%%%%%%%%%%%%%%%%%%%%%%%%%%%%%%

%%%%%%%%%%%%%%%%%%%%%%%%%%%%%%%%%%%%%%%%%%%
\begin{theorem}\label{T-Pay-1}
Assume that $\Om\subset\bbR^n$ is a bounded quasi-convex domain. 
In addition, assume there exists a sequence of $C^\infty$-smooth domains 
$\{\Om_m\}_{m\in\bbN}$ satisfying the following two conditions:
\begin{enumerate}
\item[$(i)$] The sequence $\{\Om_m\}_{m\in\bbN}$ monotonically converges 
to $\Om$ from inside, that is, 
\begin{equation} \label{A-P.2}
\Om_m \subset \Om_{m+1}\subset \Om, \; m\in\bbN,  \quad 
\bigcup_{m\in\bbN} \Om_m = \Om. 
\end{equation} 
\item[$(ii)$] If $\cG_m$ denotes the Gaussian mean curvature of 
$\partial\Om_m$, then 
\begin{equation}
\cG_m\geq 0 \, \text{ for all $m\in\bbN$.}    \lb{A-P.2a}
\end{equation}
\end{enumerate}
Then the first Dirichlet eigenvalue and the first nonzero eigenvalue for 
the Krein Laplacian in $\Om$ satisfy
\begin{equation} \label{A-P.1}
1\leq\frac{\lambda^{(0)}_{K,\Om,1}}{\lambda^{(0)}_{D,\Om,1}}\leq 4.
\end{equation} 
In particular, each bounded convex domain $\Om\subset\bbR^n$ satisfies 
conditions $(i)$ and $(ii)$ and hence \eqref{A-P.1} holds for such domains. 
\end{theorem}
%%%%%%%%%%%%%%%%%%%%%%%%%%%%%%%%%%%%%%%%%%%
\begin{proof} 
Of course, the lower bound in \eqref{A-P.1} is contained in \eqref{mam-39},
so we will concentrate on establishing the upper bound.  To this end, we 
recall that it is possible to approximate $\Omega$ with a sequence of 
$C^\infty$-smooth bounded domains satisfying \eqref{A-P.2} and 
\eqref{A-P.2a}.  By Theorem \ref{tDirichletapprox}, the Dirichlet 
eigenvalues are continuous under the domain perturbations described in 
\eqref{A-P.2} and one obtains, in particular, 
\begin{equation} \label{A-P.3}
\lim_{m\to\infty}\lambda^{(0)}_{D,\Om_{m},1} = \lambda^{(0)}_{D,\Om,1}.
\end{equation} 
On the other hand, \eqref{mam-2S} yields that
$\lambda^{(0)}_{K,\Om,1}\leq \lambda^{(0)}_{K,\Om_m,1}$.   
Together with \eqref{A-P.3}, this shows that it suffices to prove that
\begin{equation} \label{A-P.4}
\lambda^{(0)}_{K,\Om_m,1}\leq 4 \lambda^{(0)}_{D,\Om_m,1},\quad m\in\bbN.
\end{equation} 
Summarizing, it suffices to show that
\begin{align} \label{A-P.5}
\begin{split}
& \mbox{$\Om\subset\bbR^n$ a bounded, $C^\infty$-smooth domain, whose Gaussian 
mean}  \\
& \quad \text{curvature $\cG$ of $\partial\Om$ is nonnegative, implies }\, 
\lambda^{(0)}_{K,\Om,1} \leq 4\,\lambda^{(0)}_{D,\Om,1}. 
\end{split}
\end{align} 
Thus, we fix a bounded, $C^\infty$ domain $\Om\subset\bbR^n$ with $\cG\geq 0$ 
on $\partial\Om$ and denote by $u_1$ the (unique, up to normalization) first eigenfunction for the Dirichlet Laplacian in $\Om$. In the sequel, we abbreviate
$\lambda_D:=\lambda^{(0)}_{D,\Om,1}$ and $\lambda_K:=\lambda^{(0)}_{K,\Om,1}$. Then (cf.\ \cite[Theorems\ 8.13 and 8.38]{GT83}),
\begin{equation} \label{A-P.6}
u_1\in C^\infty(\ol{\Om}),\quad u_1|_{\dOm}=0,\quad u_1>0\mbox{ in }\Om,\quad
-\Delta u_1=\lambda_D\,u_1 \, \mbox{ in } \, \Om,
\end{equation} 
and
\begin{equation} \label{A-P.7}
\lambda_D=\frac{\int_{\Om}d^nx\,|\nabla u_1|^2}{\int_{\Om}d^nx\,|u_1|^2}.
\end{equation} 
In addition, \eqref{mam-29} (with $j=1$) and $u_1^2$ as a ``trial function'' yields
\begin{equation} \label{A-P.8}
\lambda_K \leq\frac{\int_{\Om}d^nx\,|\Delta(u_1^2)|^2}
{\int_{\Om}d^nx\,|\nabla(u_1^2)|^2}.
\end{equation} 
Then \eqref{A-P.5} follows as soon as one shows that the right-hand side
of \eqref{A-P.8} is less than or equal to the quadruple of the
right-hand side of \eqref{A-P.7}.  For bounded, smooth, convex domains
in the plane (i.e., for $n=2$), such an estimate was established in 
\cite{Pa60}.  For the convenience of the reader, below we review Payne's 
ingenious proof, primarily to make sure that it continues to hold in much 
the same format for our more general class of domains and in all space 
dimensions (in the process, we also shed more light on some less explicit 
steps in Payne's original proof, including the realization that the key 
hypothesis is not convexity of the domain, but rather nonnegativity of the 
Gaussian mean curvature $\cG$ of its boundary).  To get started, we expand
\begin{equation} \label{A-P.9}
(\Delta(u_1^2))^2=4\big[\lambda_D^2u_1^4-2\lambda_D\,u_1^2|\nabla u_1|^2
+|\nabla u_1|^4\big],\quad |\nabla(u_1^2)|^2=4\,u_1^2|\nabla u_1|^2,
\end{equation} 
and use \eqref{A-P.8} to write
\begin{equation} \label{A-P.10}
\lambda_K \leq\lambda_D^2\left(\frac{\int_{\Om}d^nx\,u_1^4}
{\int_{\Om}d^nx\,u_1^2|\nabla u_1|^2}\right)-2\lambda_D 
+\left(\frac{\int_{\Om}d^nx\,|\nabla u_1|^4}
{\int_{\Om}d^nx\,u_1^2|\nabla u_1|^2}\right).
\end{equation} 
Next, observe that based on \eqref{A-P.6} and the Divergence Theorem
we may write
\begin{align}\label{A-P.11}
\int_{\Om}d^nx\,\big[3u_1^2|\nabla u_1|^2-\lambda_D\,u_1^4\big] &=
\int_{\Om}d^nx\,\big[3u_1^2|\nabla u_1|^2+u_1^3\Delta u_1\big]
=\int_{\Om}d^nx\, {\rm div} \big(u_1^3\nabla u_1\big)
\nonumber\\
&= \int_{\partial\Om}d^{n-1}\omega\,u_1^3\partial_{\nu}u_1=0,
\end{align}
where $\nu$ is the outward unit normal to $\dOm$, and $d^{n-1}\omega$ 
denotes the induced surface measure on $\partial\Omega$.
This shows that the coefficient of $\lambda_D^2$ in \eqref{A-P.10} is
$3\lambda_D^{-1}$, so that
\begin{equation} \label{A-P.12}
\lambda_K \leq\lambda_D +\theta, \, \mbox{ where }\, 
\theta:=\frac{\int_{\Om}d^nx\,|\nabla u_1|^4}{\int_{\Om}d^nx\,u_1^2|\nabla u_1|^2}.
\end{equation} 
We begin to estimate $\theta$ by writing
\begin{align}\label{A-P.13}
\int_{\Om}d^nx\,|\nabla u_1|^4 &=
\int_{\Om}d^nx\,(\nabla u_1)\cdot(|\nabla u_1|^2\nabla u_1)
=-\int_{\Om}d^nx\,u_1\, {\rm div} (|\nabla u_1|^2\nabla u_1)
\nonumber\\
&= -\int_{\Om}d^nx\, \big[(u_1\,\nabla u_1)\cdot(\nabla|\nabla u_1|^2)
-\lambda_D\,u_1^2|\nabla u_1|^2\big],
\end{align}
so that
\begin{equation} \label{A-P.14}
\frac{\int_{\Om}d^nx\,(u_1\,\nabla u_1)\cdot(\nabla|\nabla u_1|^2)}
{\int_{\Om}d^nx\,u_1^2|\nabla u_1|^2}=\lambda_D-\theta.
\end{equation} 
To continue, one observes that because of \eqref{A-P.6} and the classical Hopf lemma 
(cf.\ \cite[Lemma\ 3.4]{GT83}) one has $\partial_{\nu} u_1 > 0$ on $\partial\Om$. 
Thus, $|\nabla u_1| \neq 0$ at points in $\Om$ near $\partial \Om$. This allows one to conclude that
\begin{equation} \label{A-P.15}
\nu=-\frac{\nabla u_1}{|\nabla u_1|}\, \mbox{ near and on }\, \dOm. 
\end{equation}

By a standard result from differential geometry (see, for example, 
\cite[p.\ 142]{Ca92})
\begin{equation} \label{A-P.16}
{\rm div} (\nu)=(n-1)\,\cG\, \mbox{ on }\, \partial\Om,
\end{equation}
where $\cG$ denotes the mean curvature of $\partial\Om$.  

To proceed further, we introduce the following notations for the second
derivative matrix, or {\it Hessian}, of $u_1$ and its norm:
\begin{equation} \label{A-P.18}
{\rm Hess} (u_1):=
\left(\frac{\partial^2 u_1}{\partial x_j\partial x_k}\right)_{1\leq j,k\leq n},
\quad
|{\rm Hess} (u_1)|:= \bigg(\sum_{j,k=1}^n|\partial_j\partial_k u_1|^2\bigg)^{1/2}.
\end{equation}
Relatively brief and straightforward computations (cf.\ \cite[Theorem 2.2.14]{KP99}) then yield
\begin{align} \label{A-P.19}
{\rm div} (\nu) = - \sum_{j=1}^n \partial_j \bigg(\frac{\partial_j u_1}{|\nabla u_1|}\bigg)
&=|\nabla u_1|^{-1}[-\Delta u_1 + \langle \nu, {\rm Hess}(u_1) \nu \rangle]
\nonumber\\
&=|\nabla u_1|^{-1}\langle \nu, {\rm Hess}(u_1) \nu \rangle  \, \mbox{ on } \, \partial\Om
\end{align}
(since $-\Delta u_1 = \lambda u_1 = 0$ on $\partial\Om$),
\begin{align} \label{A-P.20a}
\nu \cdot (\partial_{\nu} \nu) &= - \sum_{j,k=1}^n\nu_j \nu_k \partial_k
\bigg(\frac{\partial_j u_1}{|\nabla u_1|}\bigg) \no \\
& =-|\nabla u_1|^{-1} \langle \nu,{\rm Hess}(u_1) \nu \rangle
+ |\nabla u_1|^{-1}|\nu|^2 \langle \nu, {\rm Hess} (u_1) \nu \rangle
\nonumber\\
&=0,
\end{align}
and finally, by \eqref{A-P.19}, 
\begin{align}\label{A-P.21a}
\partial_\nu(|\nabla u_1|^2) &=\sum_{j,k=1}^n \nu_j \partial_j [(\partial_k u_1)^2] 
=2 \sum_{j,k=1}^n \nu_j (\partial_k u_1)(\partial_j \partial_k u_1)
\nonumber\\
&=-2|\nabla u_1| \langle \nu, {\rm Hess} (u_1) \nu \rangle 
=-2|\nabla u_1|^2 {\rm div} (\nu)
\nonumber\\
&=-2(n-1) \cG |\nabla u_1|^2 \leq 0 \, \mbox{ on } \, \partial\Om, 
\end{align}
given our assumption $\cG\geq 0$. 

Next, we compute
\begin{align}\label{A-P.20}
& \int_{\Om}d^nx\, \big[|\nabla(|\nabla u_1|^2)|^2-2\lambda_D\,|\nabla u_1|^4
+2|\nabla u_1|^2|{\rm Hess} (u_1)|^2\big]  \no \\
& \quad =\int_{\Om}d^nx\, div \big(|\nabla u_1|^2\nabla(|\nabla u_1|^2)\big) 
=\int_{\dOm}d^{n-1}\omega\,\nu\cdot \big(|\nabla u_1|^2\nabla(|\nabla u_1|^2)\big)  \no \\
& \quad =\int_{\dOm}d^{n-1}\omega\,|\nabla u_1|^2\partial_{\nu}\big(|\nabla u_1|^2\big)\leq
0,
\end{align}
since $\partial_{\nu}(|\nabla u_1|^2)\leq 0$ on $\dOm$ by \eqref{A-P.21a}.
As a consequence,
\begin{equation} \label{A-P.21}
2\lambda_D\,\int_{\Om}d^nx\,|\nabla u_1|^4\geq
\int_{\Om}d^nx\,\big[|\nabla(|\nabla u_1|^2)|^2+2|\nabla u_1|^2|{\rm Hess} (u_1)|^2\big].
\end{equation}
Now, simple algebra shows that
$|\nabla(|\nabla u_1|^2)|^2\leq 4\,|\nabla u_1|^2|{\rm Hess} (u_1)|^2$
which, when combined with \eqref{A-P.21}, yields
\begin{equation} \label{A-P.22}
\frac{4\lambda_D}{3}\,\int_{\Om}d^nx\,|\nabla u_1|^4\geq
\int_{\Om}d^nx\,|\nabla(|\nabla u_1|^2)|^2.
\end{equation}
Let us now return to \eqref{A-P.13} and rewrite this equality as
\begin{equation} \label{A-P.23}
\int_{\Om}d^nx\,|\nabla u_1|^4 =
-\int_{\Om}d^nx\,(u_1\,\nabla u_1)\cdot(\nabla|\nabla u_1|^2
-\lambda_D\,u_1\nabla u_1).
\end{equation}
An application of the Cauchy-Schwarz inequality then yields
\begin{equation} \label{A-P.24}
\left(\int_{\Om}d^nx\,|\nabla u_1|^4\right)^2
\leq\left(\int_{\Om}d^nx\,u_1^2\,|\nabla u_1|^2\right)
\left(\int_{\Om}d^nx\,|\nabla|\nabla u_1|^2-\lambda_D\,u_1\nabla u_1|^2\right).
\end{equation}
By expanding the last integrand and recalling the definition of $\theta$
we then arrive at
\begin{equation} \label{A-P.25}
\theta^2\leq\lambda_D^2-2\lambda_D\left(\frac{\int_{\Om}d^nx\,
(u_1\nabla u_1)\cdot(\nabla|\nabla u_1|^2)}{\int_{\Om}d^nx\,u_1^2|\nabla u_1|^2}\right)
+\left(\frac{\int_{\Om}d^nx\,|\nabla(|\nabla u_1|^2)|^2}
{\int_{\Om}d^nx\,u_1^2|\nabla u_1|^2}\right).
\end{equation}
Upon recalling \eqref{A-P.14} and \eqref{A-P.22}, this becomes
\begin{equation} \label{A-P.26}
\theta^2\leq
\lambda_D^2-2\lambda_D(\lambda_D-\theta)+\frac{4\lambda_D}{3}\theta
=-\lambda_D^2+\frac{10\lambda_D}{3}\theta.
\end{equation}
In turn, this forces $\theta\leq 3\lambda_D$ hence, ultimately,
$\lambda_K\leq 4\lambda_D$ due to this estimate and \eqref{A-P.12}.
This establishes \eqref{A-P.5} and completes the proof of the theorem.
\end{proof}
%%%%%%%%%%%%%%%%%%%%%%%%%%%%%%%%%%%%%%

%%%%%%%%%%%%%%%%%%%%%%%%%%%%%%%%%%%%%%
\begin{remark} 
$(i)$ The upper bound in \eqref{A-P.1} for two-dimensional smooth, convex 
$C^{\infty}$ domains $\Om$ is due to Payne \cite{Pa60} in 1960. He notes that the proof carries over without difficulty to dimensions $n\geq 2$ in \cite[p.\ 464]{Pa67}. In addition, one can avoid assuming smoothness in his proof by using smooth approximations $\Om_m$, $m\in\bbN$, of $\Om$ as discussed in our proof.  Of course, Payne did not consider the eigenvalues of the Krein Laplacian 
$-\Delta_{K,\Om}$, instead, he compared the first eigenvalue of the fixed membrane problem and the first eigenvalue of the problem of the buckling of a clamped plate. \\ 
$(ii)$ By thinking of ${\rm Hess} (u_1)$ represented in terms of an 
orthonormal basis for $\mathbb{R}^n$ that contains $\nu$, one sees that 
\eqref{A-P.19} yields
\begin{equation}\label{A-P.22a}
{\rm div} (\nu) = \bigg|\frac{\partial u_1}{\partial \nu}\bigg|^{-1} \,
\frac{\partial^2 u_1}{{\partial \nu}^2}
=-\bigg(\frac{\partial u_1}{\partial \nu}\bigg)^{-1} \frac{\partial^2 u_1}{{\partial \nu}^2}
\end{equation}
(the latter because $\partial u_1/\partial \nu < 0$ on $\partial\Om$ by our convention on the sign of $u_1$ (see \eqref{A-P.6})), and thus
\begin{equation}\label{A-P.23a}
\frac{\partial^2 u_1}{{\partial \nu}^2} = -(n-1) \cG \frac{\partial u_1}{\partial \nu} \,
\mbox{ on } \partial\Om. 
\end{equation}
For a different but related argument leading to this same result, see 
Ashbaugh and Levine \cite[pp.\ I-8, I-9]{AL97}. Aviles \cite{Av86}, Payne \cite{Pa55}, 
\cite{Pa60}, and Levine and Weinberger \cite{LW86} 
all use similar arguments as well. \\
$(iii)$ We note that Payne's basic result here, when done in $n$ dimensions, 
holds for smooth domains having a boundary which is everywhere of nonnegative 
mean curvature.  In addition, Levine and Weinberger \cite{LW86}, in the 
context of a related problem, consider nonsmooth domains for the nonnegative 
mean curvature case and a variety of cases intermediate between that and the 
convex case (including the convex case). \\
$(iv)$ Payne's argument (and the constant 4 in Theorem \ref{T-Pay-1}) would 
appear to be sharp, with any infinite slab in $\mathbb{R}^n$ bounded by 
parallel hyperplanes being a saturating case (in a limiting sense).  We note 
that such a slab is essentially one-dimensional, and that, up to 
normalization, the first Dirichlet eigenfunction $u_1$ for the interval 
$[0,a]$ (with $a>0$) is 
\begin{equation}
u_1(x)=\sin (\pi x/a) \, \text{ with eigenvalue } \, \lambda=\pi^2/a^2, 
\end{equation}
while the corresponding first buckling eigenfunction and eigenvalue are 
\begin{equation}
u_1(x)^2=\sin^2 (\pi x/a)=[1-\cos (2\pi x/a)]/2 \, \text{ and } \, 4 \lambda=4\pi^2/a^2.  
\end{equation}
Thus, Payne's choice of the trial function $u_1^2$, where $u_1$ is the first Dirichlet eigenfunction should be optimal for this limiting case,
implying that the bound 4 is best possible. Payne, too, made observations about the equality case of his inequality, and observed that the infinite strip saturates it in 2 dimensions.  His supporting arguments are via tracing the case of equality through the inequalities in his proof, which also yields interesting insights. 
\end{remark}
%%%%%%%%%%%%%%%%%%%%%%%%%%%%%%%%%%%%%%

%%%%%%%%%%%%%%%%%%%%%%%%%%%%%%%%%%%%%%
\begin{remark}\label{Rm-2}
The eigenvalues corresponding to the buckling of a two-dimensional
{\it square} plate, clamped along its boundary, have been analyzed numerically
by several authors (see, e.g., \cite{AD92}, \cite{AD93}, and \cite{BT99}).
All these results can now be naturally reinterpreted in the context of the
Krein Laplacian $-\Delta_{K,\Om}$ in the case where $\Om=(0,1)^2\subset\bbR^2$. 
Lower bounds for the first $k$ buckling problem eigenvalues were discussed in 
\cite{LP85}. The existence of convex domains $\Om$, for which the first eigenfunction of the problem of a clamped plate and the problem of the buckling of a clamped plate possesses a change of sign, was established in \cite{KKM90}. Relations between an eigenvalue problem governing the behavior of an elastic medium and the buckling problem were studied in \cite{Ho91}. Buckling eigenvalues as a function of the elasticity constant are investigated in \cite{KLV93}. Finally, spectral properties of linear operator pencils $A-\lambda B$ with discrete spectra, and basis properties of the corresponding eigenvectors, applicable to differential operators,  were discussed, for instance, in 
\cite{Pe68}, \cite{Tr00} (see also the references cited therein). 
\end{remark}
%%%%%%%%%%%%%%%%%%%%%%%%%%%%%%%%%%%%%%

Formula \eqref{mam-46} suggests the issue of deriving a Weyl asymptotic
formula for the perturbed Krein Laplacian $H_{K,\Om}$. This is the topic
of our next section.

%%%%%%%%%%%%%%%%%%%%%%%%%%%%%%%%%%%%%%
%%%%%%%%%%%%%%%%%%%%%%%%%%%%%%%%%%%%%%
\section{Weyl Asymptotics for the Perturbed Krein Laplacian in Nonsmooth Domains}
\label{s12}
%%%%%%%%%%%%%%%%%%%%%%%%%%%%%%%%%%%%%%
%%%%%%%%%%%%%%%%%%%%%%%%%%%%%%%%%%%%%%

We begin by recording a very useful result due to V.A.\ Kozlov which,
for the convenience of the reader, we state here in more generality than
is actually required for our purposes. To set the stage,
let $\Omega\subset\bbR^n$, $n\ge2$, be a bounded Lipschitz domain.
In addition, assume that $m>r\geq 0$ are two fixed integers and set
\begin{equation} \label{kko-0}
\eta:=2(m-r)>0.
\end{equation} 
Let $W$ be a closed subspace in $H^m(\Omega)$ such that
$H^m_0(\Omega)\subseteq W$. On $W$, consider the symmetric forms
\begin{equation} \label{kko-1}
a(u,v):=\sum_{0 \leq |\alpha|,|\beta|\leq m}
\int_{\Omega}d^nx\,a_{\alpha,\beta}(x)\ol{(\partial^\beta u)(x)}
(\partial^\alpha v)(x),  \quad u, v \in W, 
\end{equation} 
and
\begin{equation} \label{kko-2}
b(u,v):=\sum_{0 \leq |\alpha|,|\beta|\le r}
\int_{\Omega}d^nx\,b_{\alpha,\beta}(x)\ol{(\partial^\beta u)(x)}
(\partial^\alpha v)(x),   \quad u, v \in W.
\end{equation} 
Suppose that the leading coefficients in $a(\dott,\dott)$ and $b(\dott,\dott)$
are Lipschitz functions, while the coefficients of all lower-order terms are
bounded, measurable functions in $\Omega$. Furthermore,
assume that the following coercivity, nondegeneracy, and nonnegativity
conditions hold: For some $C_0 \in (0,\infty)$, 
\begin{align}\label{kko-3}
& a(u,u)\ge C_0\|u\|^2_{H^m(\Omega)}, \quad u\in W,
\\
& \sum_{|\alpha|=|\beta|=r}b_{\alpha,\beta}(x) \, \xi^{\alpha+\beta}\not=0, 
\quad x\in\ol{\Om}, \; \xi\not=0,
\label{kko-4}
\\
& b(u,u)\ge 0, \quad u\in W. 
\label{kko-5}
\end{align}
Under the above assumptions, $W$ can be regarded as a Hilbert space when
equipped with the inner product $a(\dott,\dott)$. Next, consider the
operator $T\in\cB(W)$ uniquely defined by the requirement that
\begin{equation} \label{kko-6}
a(u,T v)=b(u,v),  \quad u,v\in W.
\end{equation} 
Then the operator $T$ is compact, nonnegative and self-adjoint on $W$
(when the latter is viewed as a Hilbert space). Going further, denote by
\begin{equation} \label{kko-7}
0\leq\cdots\leq\mu_{j+1}(T)\leq\mu_j(T)\leq\cdots\leq\mu_1(T),
\end{equation} 
the eigenvalues of $T$ listed according to their multiplicity, and set
\begin{equation} \label{kko-8}
N(\lambda;W,a,b):=\#\,\{j\in\bbN\,|\,\mu_j(T)\geq \lambda^{-1}\},  \quad  \lambda>0.
\end{equation} 
The following Weyl asymptotic formula is a particular case of a slightly
more general result which can be found in \cite{Ko83}.

%%%%%%%%%%%%%%%%%%%%%%%%%%%%%%%%%
\begin{theorem}\label{T-Koz}
Assume Hypothesis \ref{h2.1} and retain the above notation and assumptions on
$a(\dott,\dott)$, $b(\dott,\dott)$, $W$, and $T$. In addition, we recall \eqref{kko-0}.
Then the distribution function of the spectrum of $T$ introduced
in \eqref{kko-8} satisfies the asymptotic formula
\begin{equation} \label{kko-9}
N(\lambda;W,a,b)
=\omega_{a,b,\Omega}\,\lambda^{n/\eta}+O\big(\lambda^{(n-(1/2))/\eta}\big)
\, \mbox{ as }\, \lambda\to\infty,
\end{equation} 
where, with $d\omega_{n-1}$ denoting the surface measure on the
unit sphere $S^{n-1}=\{\xi\in\bbR^n\,|\,|\xi|=1\}$ in $\bbR^n$,
\begin{equation} \label{kko-10}
\omega_{a,b,\Omega}:=\frac{1}{n(2\pi)^n}
\int_{\Omega}d^nx\,\left(\int_{|\xi|=1}d\omega_{n-1}(\xi)\,
\left[\frac{\sum_{|\alpha|=|\beta|=r}b_{\alpha,\beta}(x)\xi^{\alpha+\beta}}
{\sum_{|\alpha|=|\beta|=m}a_{\alpha,\beta}(x)\xi^{\alpha+\beta}}
\right]^{\frac{n}{\eta}}\right).
\end{equation} 
\end{theorem}
%%%%%%%%%%%%%%%%%%%%%%%%%%%%%%%%%%%%%%%%

Various related results can be found in \cite{Ko79}, \cite{Ko84}.
After this preamble, we are in a position to state and prove the main result
of this section:

%%%%%%%%%%%%%%%%%%%%%%%%%%%%%%%%%
\begin{theorem}\label{T-KrWe}
Assume Hypothesis \ref{h.VK}. In addition, we recall that
\begin{equation} \label{kko-11}
N_{K,\Om}(\lambda)=\#\{j\in\bbN\,|\,\lambda_{K,\Om,j}\leq\lambda\},
\quad\lambda\in\bbR,
\end{equation} 
where the $($strictly$)$ positive eigenvalues $\{\lambda_{K,\Om,j}\}_{j\in\bbN}$
of the perturbed Krein Laplacian $H_{K,\Om}$ are enumerated as
in \eqref{mam-1} $($according to their multiplicities$)$. Then the following
Weyl asymptotic  formula holds:
\begin{equation} \label{kko-12}
N_{K,\Om}(\lambda)
=(2\pi)^{-n}v_n|\Omega|\,\lambda^{n/2}+O\big(\lambda^{(n-(1/2))/2}\big)
\, \mbox{ as }\, \lambda\to\infty,
\end{equation} 
where, as before, $v_n$ denotes the volume of the unit ball in $\bbR^n$,
and $|\Omega|$ stands for the $n$-dimensional Euclidean volume of $\Omega$.
\end{theorem}
%%%%%%%%%%%%%%%%%%%%%%%%%%%%%%%%%%%%%%%%
\begin{proof}
Set $W:=H^2_0(\Om)$ and consider the symmetric forms
\begin{align}\label{kko-13}
& a(u,v):=\int_{\Om}d^nx\,\ol{(-\Delta+V)u}\,(-\Delta+V)v,\quad u,v\in W,
\\
& b(u,v):=\int_{\Om}d^nx\,\ol{\nabla u}\cdot\nabla v
+\int_{\Om}d^nx\,\ol{V^{1/2}u}\,V^{1/2}v,\quad u,v\in W,
\label{kko-13a}
\end{align}
for which conditions \eqref{kko-3}--\eqref{kko-5} (with $m=2$) are
verified (cf.\ \eqref{mam-6}). Next, we recall the operator
$(-\Delta+V)^{-2}:=
((-\Delta+V)^2)^{-1}\in\cB\bigl(H^{-2}(\Om),H^2_0(\Om)\bigr)$
from \eqref{mam-9} along with the operator
\begin{equation} \label{kko-14}
B\in\cB_{\infty}(W),\quad
Bu:=-(-\Delta+V)^{-2}(-\Delta+V)u,\quad u\in W,
\end{equation} 
from \eqref{mam-11}. Then, in the current notation, formula
\eqref{mam-13} reads $a(Bu,v)=b(u,v)$ for every $u,v\in C^\infty_0(\Om)$.
Hence, by density,
\begin{equation} \label{kko-15}
a(Bu,v)=b(u,v),\quad u,v\in W.
\end{equation} 
This shows that actually $B=T$, the operator originally introduced in
\eqref{kko-6}. In particular, $T$ is one-to-one. Consequently, $Tu=\mu\,u$
for $u\in W$ and $0\not=\mu\in\bbC$, if and only if $u\in H^2_0(\Omega)$
satisfies $(-\Delta+V)^{-2}(-\Delta+V)u=\mu\,u$, that is, 
$(-\Delta+V)^2 u=\mu^{-1}(-\Delta+V)u$.
Hence, the eigenvalues of $T$ are precisely the reciprocals of the
eigenvalues of the buckling clamped plate problem \eqref{Yan-14z}.
Having established this, formula \eqref{kko-12} then follows from
Theorem \ref{T-MM-2} and \eqref{kko-9}, upon observing that in our case
$m=2$, $r=1$ (hence $\eta=2$) and $\omega_{a,b,\Omega}=(2\pi)^{-n}v_n|\Om|$. 
\end{proof}
%%%%%%%%%%%%%%%%%%%%%%%%%%%%%%%%%%%%%%%%

Incidentally, Theorem \ref{T-KrWe} and Theorem \ref{T-MM-2} show that,
granted Hypothesis \ref{h.VK}, a Weyl asymptotic
formula holds in the case of the (perturbed) buckling problem \eqref{MM-1}.
For smoother domains and potentials, this is covered by Grubb's results
in \cite{Gr83}. In the smooth context, a sharpening of the remainder has 
been announced in \cite{Mik94} without proof.

In the case where $\Omega\subset\bbR^2$ is a bounded domain with
a $C^\infty$-boundary and $0\leq V\in C^\infty(\ol{\Om})$,
a more precise form of the error term in \eqref{kko-12}
was obtained in \cite{Gr83} where Grubb has shown that
\begin{equation} \label{kko-13x}
N_{K,\Om}(\lambda)=\frac{|\Omega|}{4\pi}\,\lambda+O\big(\lambda^{2/3}\big)
\, \mbox{ as }\, \lambda\to\infty,
\end{equation} 
In fact, in \cite{Gr83}, Grubb deals with the Weyl asymptotic for the
Krein--von Neumann extension of a general strongly elliptic, formally self-adjoint
differential operator of arbitrary order, provided both its coefficients
as well as the the underlying domain $\Omega\subset\bbR^n$ ($n\geq 2$)
are $C^\infty$-smooth. In the special case where $\Om$ equals the open ball 
$B_n(0;R)$, $R>0$, in $\bbR^n$, and when $V\equiv 0$, it turns out that \eqref{kko-12}, 
\eqref{kko-13x} can be further refined to
\begin{align} \label{NN-y} 
N^{(0)}_{K,B_n(0;R)}(\lambda)&=(2\pi)^{-n}v_n^2 R^n \lambda^{n/2} 
- (2\pi)^{-(n-1)}v_{n-1}[(n/4)v_n + v_{n-1}] R^{n-1} \lambda^{(n-1)/2}  \no \\
& \quad + O\big(\lambda^{(n-2)/2}\big) \mbox{ as }\, \lambda\to\infty,   
\end{align} 
for every $n\geq 2$. This will be the object of the final Section \ref{s1vi} (cf.\ Proposition 
\ref{p10.1}).

%%%%%%%%%%%%%%%%%%%%%%%%%%%%%%%%%%%%%%
%%%%%%%%%%%%%%%%%%%%%%%%%%%%%%%%%%%%%%
\section{A Class of Domains for which the Krein and Dirichlet Laplacians Coincide}
\label{s1v}
%%%%%%%%%%%%%%%%%%%%%%%%%%%%%%%%%%%%%%
%%%%%%%%%%%%%%%%%%%%%%%%%%%%%%%%%%%%%%

Motivated by the special example where $\Omega=\bbR^2\backslash\{0\}$ and 
$S=\ol{-\Delta_{C_0^\infty(\bbR^2\backslash\{0\})}}$, in which case one can show the 
interesting fact that $S_F=S_K$ (cf.\ \cite{AGHKH87}, \cite[Ch.\ I.5]{AGHKH88}, 
\cite{GKMT01}, and Subsections \ref{s10.3} and \ref{s10.4}) and hence the nonnegative self-adjoint extension of $S$ is unique, the aim of this section is to present a class of (nonempty, proper) 
open sets $\Omega=\bbR^n\backslash K$, $K\subset \bbR^n$ compact and subject to a vanishing Bessel capacity condition, with the property that the Friedrichs and
Krein--von Neumann extensions of $-\Delta\big|_{C^\infty_0(\Om)}$ in $L^2(\Om; d^n x)$, coincide. 
To the best of our knowledge, the case where the set $K$ differs from a single point is without 
precedent and so the following results for more general sets $K$ appear to be new.  

We start by making some definitions and discussing some preliminary results,
of independent interest. Given an arbitrary open set $\Om\subset\bbR^n$, $n\geq 2$, 
we consider three realizations of $-\Delta$ as unbounded operators
in $L^2(\Om;d^nx)$, with domains given by (cf.\ Subsection \ref{s4X})
\begin{align}\label{YF-1}
\dom(-\Delta_{max,\Om})&:=\big\{u\in L^2(\Omega;d^nx)\,\big|\,
\Delta u\in L^2(\Omega;d^nx)\big\},
\\
\dom(-\Delta_{D,\Om})&:=\big\{u\in H^1_0(\Omega)\,\big|\,
\Delta u\in L^2(\Omega;d^nx)\big\},
\label{YF-2}
\\
\dom(-\Delta_{c,\Om})&:=C^\infty_0(\Omega).
\label{YF-3}
\end{align}

%%%%%%%%%%%%%%%%%%%%%%%%%%%%%%%%%%%%%%%
\begin{lemma}\label{L-ea}
For any open, nonempty subset $\Om\subseteq \bbR^n$, $n\geq 2$, the following statements hold:
\begin{enumerate}
\item[$(i)$] One has
\begin{equation} \label{Fga-2}
(-\Delta_{c,\Om})^*=-\Delta_{max,\Om}.
\end{equation} 
\item[$(ii)$] The Friedrichs extension of $-\Delta_{c,\Om}$ is given by 
\begin{equation} \label{Fga-3}
(-\Delta_{c,\Om})_F=-\Delta_{D,\Om}.
\end{equation} 
\item[$(iii)$] The Krein--von Neumann extension of $-\Delta_{c,\Om}$ has the domain
\begin{align} 
& \dom((-\Delta_{c,\Om})_K)=
\big\{u\in\dom(-\Delta_{\max,\Om})\,\big|\,\mbox{there exists }
\{u_j\}_{j\in\bbN} \in C^\infty_0(\Om)     \label{Gkj-1} \\
& \quad \mbox{with } 
\lim_{j\to\infty}\|\Delta u_j\ - \Delta u\|_{L^2(\Om;d^nx)} = 0   
 \mbox{ and $\{\nabla u_j\}_{j\in\bbN}$ Cauchy
in $L^2(\Om;d^nx)^n$}\big\}.    \no 
\end{align}
\item[$(iv)$] One has
\begin{equation} \label{F-2Lb}
\ker((-\Delta_{c,\Om})_{K})
= \big\{u\in L^2(\Om;d^nx)\,\big|\,\Delta\,u=0 \mbox{ in } \Om\big\},
\end{equation} 
and
\begin{equation} \label{F-2La}
\ker((-\Delta_{c,\Om})_{F})=\{0\}.
\end{equation} 
\end{enumerate}
\end{lemma}
%%%%%%%%%%%%%%%%%%%%%%%%%%%%%%%%%%%%
\begin{proof}
Formula \eqref{Fga-2} follows in a straightforward fashion, by unraveling
definitions, whereas \eqref{Fga-3} is a direct consequence of \eqref{Fr-2} 
or \eqref{Fr-Q}  
(compare also with Proposition \ref{L-Fri1}). Next, \eqref{Gkj-1} is
readily implied by \eqref{Fr-2X} and \eqref{Fga-2}. In addition,
\eqref{F-2Lb} is easily derived from \eqref{Fr-4Tf}, \eqref{Fga-2} and
\eqref{YF-1}. Finally, consider \eqref{F-2La}. In a first stage,
\eqref{Fga-3} and \eqref{YF-2} yield that
\begin{equation} \label{F-2Lc}
\ker ((-\Delta_{c,\Om})_{F}) 
= \big\{u\in H^1_0(\Om)\,\big|\,\Delta\,u=0 \mbox{ in } \Om\big\},
\end{equation} 
so the goal is to show that the latter space is trivial. To this end,
pick a function $u\in H^1_0(\Om)$ which is harmonic in $\Om$, and observe
that this forces $\nabla u=0$ in $\Om$. Now, with tilde
denoting the extension by zero outside $\Om$, we have
$\widetilde{u}\in H^1(\bbR^n)$ and
$\nabla(\widetilde{u})=\widetilde{\nabla u}$.
In turn, this entails that $\widetilde{u}$ is a constant function
in $L^2(\bbR;d^nx)$ and hence $u\equiv 0$ in $\Om$, establishing 
\eqref{F-2La}.  
\end{proof}
%%%%%%%%%%%%%%%%%%%%%%%%%%%%%%%%%%%

Next, we record some useful capacity results. For an authoritative
extensive discussion on this topic see the monographs \cite{AH96},
\cite{Ma85}, \cite{Ta95}, and \cite{Zi89}. We denote by $B_{\alpha,2}(E)$ the Bessel
capacity of order $\alpha>0$ of a set $E\subset\bbR^n$. When 
$K\subset \bbR^n$ is a compact set, this is defined by
\begin{equation} \label{cap-1}
B_{\alpha,2}(K):=\inf\,\bigl\{\|f\|^2_{L^2(\bbR^n;d^nx)}\,\big|\,
g_\alpha\ast f\geq 1\mbox{ on }K,\,f\geq 0\bigr\},
\end{equation} 
where the Bessel kernel $g_\alpha$ is defined as the function whose
Fourier transform is given by
\begin{equation} \label{cap-2}
\widehat{g_\alpha}(\xi)=(2\pi)^{-n/2}(1+|\xi|^2)^{-\alpha/2},
\quad\xi\in\bbR^n.
\end{equation} 
When $\cO\subseteq\bbR^n$ is open, we define
\begin{equation} \label{cap-1X}
B_{\alpha,2}(\cO):=
\sup\,\{B_{\alpha,2}(K)\,|\,K\subset\cO,\,K\mbox{ compact}\,\},
\end{equation} 
and, finally, when $E\subseteq\bbR^n$ is an arbitrary set,
\begin{equation} \label{cap-1Y}
B_{\alpha,2}(E):=
\inf\,\{B_{\alpha,2}(\cO)\,|\,\cO\supset E,\,\cO\mbox{ open}\,\}.
\end{equation} 
In addition, denote by $\cH^k$ the $k$-dimensional Hausdorff measure on $\bbR^n$,
$0\leq k\leq n$. Finally, a compact subset $K\subset\bbR^n$ is said to
be {\it $L^2$-removable for the Laplacian} provided every
bounded, open neighborhood $\cO$ of $K$ has the property that
\begin{align} \label{Fga-2L.2}
& u\in L^2 (\cO\backslash  K;d^nx)\mbox{ with }\Delta u=0
\mbox{ in }\cO\backslash  K   % \no \\
% \quad 
\, \text{ imply } 
\begin{cases}
\mbox{there exists $\widetilde{u}\in L^2 (\cO;d^nx)$ so that}
\\
\mbox{$\widetilde{u}\Bigl|_{\cO\backslash  K}=u$ and
$\Delta\widetilde{u}=0$ in $\cO$}.
\end{cases}
\end{align}

%%%%%%%%%%%%%%%%%%%%%%%%%%%%%%%%%%%%%%
\begin{proposition}\label{R-Ca.1}
For $\alpha>0$, $k\in\bbN$, $n\geq 2$ and $E\subset\bbR^n$, the following
properties are valid:
\begin{enumerate}
\item[$(i)$] A compact set $K\subset\bbR^n$ is $L^2$-removable for
the Laplacian if and only if $B_{2,2}(K)=0$.
\item[$(ii)$] Assume that $\Om\subset\bbR^n$ is an open set and that
$K\subset\Omega$ is a closed set. Then the space $C^\infty_0(\Om\backslash  K)$
is dense in $H^k(\Om)$ $($i.e., one has the natural identification
$H^k_0(\Om)\equiv H^k_0(\Om\backslash  K)$$)$, if and only if $B_{k,2}(K)=0$. 
\item[$(iii)$] If $2\alpha\leq n$ and $\cH^{n-2\alpha}(E)<+\infty$
then $B_{\alpha,2}(E)=0$. Conversely, if $2\alpha\leq n$ and
$B_{\alpha,2}(E)=0$ then $\cH^{n-2\alpha+\varepsilon}(E)=0$ for
every $\varepsilon>0$. 
\item[$(iv)$] Whenever $2\alpha>n$ then there exists $C=C(\alpha,n)>0$
such that $B_{\alpha,2}(E)\geq C$ provided $E\not=\emptyset$.
\end{enumerate}
\end{proposition}
%%%%%%%%%%%%%%%%%%%%%%%%%%%%%%%%%%%%%%%

See, \cite[Corollary 3.3.4]{AH96}, \cite[Theorem 3]{Ma85}, 
\cite[Theorem 2.6.16 and Remark 2.6.15]{Zi89}, respectively.
For other useful removability criteria the interested reader may
wish to consult \cite{Ca67}, \cite{Ma65}, \cite{RS08}, and \cite{Tr08}.

The first main result of this section is then the following:

%%%%%%%%%%%%%%%%%%%%%%%%%%%%%%%%%%%%%%%
\begin{theorem}\label{L-ea-5}
Assume that $K\subset\bbR^n$, $n\geq 3$, is a compact set
with the property that
\begin{equation} \label{Ha-z}
B_{2,2}(K)=0.
\end{equation} 
Define $\Omega:=\bbR^n\backslash  K$. Then, in the domain $\Om$,
the Friedrichs and Krein--von Neumann extensions of $-\Delta$, initially
considered on $C^\infty_0(\Om)$, coincide, that is, 
\begin{equation} \label{exa-13}
(-\Delta_{c,\Om})_F = (-\Delta_{c,\Om})_K.
\end{equation} 
As a consequence, $-\Delta|_{C^\infty_0(\Om)}$ has a unique nonnegative self-adjoint extension in $L^2(\Om;d^nx)$.
\end{theorem}
%%%%%%%%%%%%%%%%%%%%%%%%%%%%%%%%%%%%%%%
\begin{proof}
We note that \eqref{Ha-z} implies that $K$ has zero $n$-dimensional
Lebesgue measure, so that $L^2(\Om;d^nx)\equiv L^2(\bbR^n;d^nx)$.
In addition, by $(iii)$ in Proposition \ref{R-Ca.1}, we also have $B_{1,2}(K)=0$.
Now, if $u\in\dom(-\Delta_{c,\Om})_K$, \eqref{Gkj-1}
entails that $u\in L^2(\Om;d^nx)$, $\Delta u\in L^2(\Om;d^nx)$,
and that there exists a sequence $u_j\in C^\infty_0(\Om)$, $j\in\bbN$,
for which
\begin{equation} \label{exa-14}
\Delta u_j\to\Delta u \,\mbox{ in } \,L^2(\Om;d^nx)
\,\mbox{ as } \, j\to\infty,
\mbox{ and $\{\nabla u_j\}_{j\in\bbN}$ is Cauchy in $L^2(\Om;d^nx)$}.
\end{equation} 
In view of the well-known estimate (cf.\ the Corollary on p.\ 56 of \cite{Ma85}),
\begin{equation} \label{exa-15}
\|v\|_{L^{2^*}(\bbR^n;d^nx)}\leq C_n\|\nabla v\|_{L^2(\bbR^n;d^nx)},\quad 
v\in C^\infty_0(\bbR^n),
\end{equation} 
where $2^*:=(2n)/(n-2)$, the last condition in \eqref{exa-14} implies
that there exists $w\in L^{2^*}(\bbR^n;d^nx)$ with the property that
\begin{equation} \label{exa-16}
u_j\to w \,\mbox{ in } \,L^{2^*}(\bbR^n;d^nx)\, \mbox{ and }\, 
\nabla u_j\to\nabla w \,\mbox{ in } \,L^2(\bbR^n;d^nx) 
\,\mbox{ as } \,j\to\infty.
\end{equation} 
Furthermore, by the first convergence in \eqref{exa-14}, we also have that
$\Delta w=\Delta u$ in the sense of distributions in $\Om$.
In particular, the function
\begin{equation} \label{exa-17}
f:=w-u\in L^{2^*}(\bbR^n;d^nx)+L^2(\bbR^n;d^nx)
\hookrightarrow L^2_{\loc}(\bbR^n;d^nx)
\end{equation} 
satisfies $\Delta f=0$ in $\Om=\bbR^n\backslash  K$.
Granted \eqref{Ha-z}, Proposition \ref{R-Ca.1} yields that $K$ is
$L^2$-removable for the Laplacian, so we may conclude that
$\Delta f=0$ in $\bbR^n$. With this at hand, Liouville's theorem then ensures
that $f\equiv 0$ in $\bbR^n$. This forces $u=w$ as distributions in $\Om$
and hence, $\nabla u=\nabla w$ distributionally in $\Om$. In view of the
last condition in \eqref{exa-16} we may therefore conclude that
$u\in H^1(\bbR^n)=H^1_0(\bbR^n)$. With this at hand, Proposition \ref{R-Ca.1}
yields that $u\in H^1_0(\Om)$. This proves that
$\dom(-\Delta_{c,\Om})_K \subseteq\dom(-\Delta_{c,\Om})_F$ and hence,
$(-\Delta_{c,\Om})_K \subseteq (-\Delta_{c,\Om})_F$. Since both operators
in question are self-adjoint, \eqref{exa-13} follows.
\end{proof}
%%%%%%%%%%%%%%%%%%%%%%%%%%%%%%%%%%%%%%

We emphasize that equality of the Friedrichs and Krein Laplacians necessarily 
requires that fact that 
$\inf (\sigma((-\Delta_{c,\Om})_F)) = \inf(\sigma((-\Delta_{c,\Om})_K)) = 0$, 
and hence rules out the case of bounded domains $\Omega \subset \bbR^n$, 
$n \in \bbN$ (for which $\inf (\sigma((-\Delta_{c,\Om})_F)) > 0$). 

%%%%%%%%%%%%%%%%%%%%%%%%%%%%%%%%%%%%%%
\begin{corollary}\label{C-capF}
Assume that $K\subset\bbR^n$, $n\geq 4$, is a compact set
with finite $(n-4)$-dimensional Hausdorff measure, that is, 
\begin{equation} \label{Ha-zX}
\cH^{n-4}(K)<+\infty.
\end{equation} 
Then, with $\Omega:=\bbR^n\backslash  K$, one has
$(-\Delta_{c,\Om})_F = (-\Delta_{c,\Om})_K$, and hence, $-\Delta|_{C^\infty_0(\Om)}$ has a unique nonnegative self-adjoint extension in $L^2(\Om;d^nx)$.
\end{corollary}
%%%%%%%%%%%%%%%%%%%%%%%%%%%%%%%%%%%%%%
\begin{proof}
This is a direct consequence of Proposition \ref{R-Ca.1}
and Theorem \ref{L-ea-5}.
\end{proof}
%%%%%%%%%%%%%%%%%%%%%%%%%%%%%%%%%%%%%%

In closing, we wish to remark that, as a trivial particular case of the
above corollary, formula \eqref{exa-13} holds for the punctured space
\begin{equation} \label{puct-u1}
\Omega:=\bbR^n\backslash \{0\},\quad n\geq 4,
\end{equation} 
however, this fact is also clear from the well-known fact that 
$-\Delta|_{C_0^\infty (\bbR^n\backslash \{0\})}$ is essentially self-adjoint 
in $L^2(\bbR^n; d^nx)$ if (and only if) $n\geq 4$ (cf., e.g., \cite[p.\ 161]{RS75}, and our 
discussion concerning the Bessel operator \eqref{10.86}). 
In \cite[Example 4.9]{GKMT01} (see also our discussion in Subsection\ 10.3), it has been shown (by using different methods) that \eqref{exa-13} continues to hold
for the choice \eqref{puct-u1} when $n=2$, but that the Friedrichs and
Krein--von Neumann extensions of $-\Delta$, initially considered on
$C^\infty_0(\Om)$ with $\Om$ as in \eqref{puct-u1}, are different when $n=3$. 

In light of Theorem \ref{L-ea-5}, a natural question is whether the
coincidence of the Friedrichs and Krein--von Neumann extensions of $-\Delta$,
initially defined on $C^\infty_0(\Om)$ for some open set $\Om\subset\bbR^n$,
actually implies that the complement of $\Om$ has zero Bessel capacity of
order two. Below, under some mild background assumptions on the domain
in question, we shall establish this type of converse result.
Specifically, we now prove the following fact:

%%%%%%%%%%%%%%%%%%%%%%%%%%%%%%%%%%%%%%
\begin{theorem}\label{L-ea-5C}
Assume that $K\subset\bbR^n$, $n>4$, is a compact set of zero
$n$-dimensional Lebesgue measure, and set $\Omega:=\bbR^n\backslash  K$. Then
\begin{equation} \label{exa-13H}
(-\Delta_{c,\Om})_F =(-\Delta_{c,\Om})_K \,\text{ implies } \, B_{2,2}(K)=0.
\end{equation} 
\end{theorem}
%%%%%%%%%%%%%%%%%%%%%%%%%%%%%%%%%%%%%%
\begin{proof}
Let $K$ be as in the statement of the theorem.
In particular, $L^2(\Om;d^nx)\equiv L^2(\bbR^n;d^nx)$.
Hence, granted that
$(-\Delta_{c,\Om})_K=(-\Delta_{c,\Om})_F$, in view of
\eqref{F-2Lb}, \eqref{F-2La} this yields
\begin{equation} \label{F-2Ld}
\big\{u\in L^2(\bbR^n;d^nx) \,\big|\, \Delta\,u=0 \mbox{ in } \bbR^n\backslash K\big\}=\{0\}.
\end{equation} 
It is useful to think of \eqref{F-2Ld} as a capacitary condition.
More precisely, \eqref{F-2Ld} implies that ${\rm Cap} (K)=0$, where
\begin{equation} \label{F-3Ld}
{\rm Cap} (K):=\sup\,\bigl\{
\bigl|{}_{\cE'(\bbR^n)}\langle\Delta u,1\rangle_{\cE(\bbR^n)}\bigr|\,\big|\,
\|u\|_{L^2(\bbR^n;d^nx)}\leq 1\mbox{ and }{\rm supp} (\Delta u)\subseteq K
\bigr\}.
\end{equation} 
Above, $\cE(\bbR^n)$ is the space of smooth functions in $\bbR^n$ equipped
with the usual Frech\'et topology, which ensures that its dual,
$\cE'(\bbR^n)$, is the space of compactly supported distributions in
$\bbR^n$. At this stage, we recall the fundamental solution for the
Laplacian in $\bbR^n$, $n\geq 3$, that is, 
\begin{equation} \label{Fr-Ta1}
E_n(x):=\frac{\Gamma(n/2)}{2(2-n)\pi^{n/2}|x|^{n-2}},\quad
x\in\bbR^n\backslash \{0\}
\end{equation} 
($\Gamma(\cdot)$ the classical Gamma function \cite[Sect.\ 6.1]{AS72}), and introduce a related capacity, namely
\begin{align} \label{F-4Ld}
%\begin{split}
& {\rm Cap}_{\ast} (K):=\sup\,\bigl\{
\big|{}_{\cE'(\bbR^n)}\langle f,1\rangle_{\cE(\bbR^n)}\bigr| \,\big|\,
f\in \cE'(\bbR^n),\,\,{\rm supp} (f) \subseteq K,  % \\
% & \hspace*{6.1cm} 
\|E_n\ast f\|_{L^2(\bbR^n;d^nx)}\leq 1\big\}.
%\end{split} 
\end{align} 
Then
\begin{equation} \label{Fr-Ta2}
0\leq{\rm Cap}_{\ast} (K)\leq {\rm Cap} (K)=0
\end{equation} 
so that ${\rm Cap}_{\ast} (K)=0$. With this at hand,
\cite[Theorem 1.5\,(a)]{HP72} (here we make use of the
fact that $n>4$) then allows us to
strengthen \eqref{F-2Ld} to
\begin{equation} \label{F-5Ld}
\big\{u\in L^2_{\loc}(\bbR^n;d^nx) \,\big|\, \Delta\,u=0 \mbox{ in } 
\bbR^n\backslash  K\big\}=\{0\}.
\end{equation} 
Next, we follow the argument used in the proof of \cite[Lemma 5.5]{MH73} and 
\cite[Theorem 2.7.4]{AH96}.
Reasoning by contradiction, assume that $B_{2,2}(K)>0$. Then there exists
a nonzero, positive measure $\mu$ supported in $K$ such that
$g_2\ast\mu\in L^2(\bbR^n)$.
Since $g_2(x)=c_n\,E_n (x)+o(|x|^{2-n})$ as $|x|\to 0$ (cf.\ the discussion in
Section 1.2.4 of \cite{AH96}) this further implies that
$E_n\ast\mu\in L^2_{\loc}(\bbR^n;d^nx)$. However, $E_n\ast\mu$ is a
harmonic function in $\bbR^n\backslash  K$, which is not identically zero since
\begin{equation} \label{Niz}
\lim_{x\to\infty}|x|^{n-2}(E_n\ast\mu)(x)=c_n\mu(K)>0,
\end{equation} 
so this contradicts \eqref{F-5Ld}. This shows that $B_{2,2}(K)=0$.
\end{proof}
%%%%%%%%%%%%%%%%%%%%%%%%%%%%%%%%%%%

Theorems \ref{L-ea-5}--\ref{L-ea-5C} readily generalize to other types of
elliptic operators (including higher-order systems). For example, using the
polyharmonic operator $(-\Delta)^\ell$, $\ell\in\bbN$, as a prototype,
we have the following result:

%%%%%%%%%%%%%%%%%%%%%%%%%%%%%%%%%%%
\begin{theorem}\label{L-ea-8}
Fix $\ell\in\bbN$, $n\geq 2\ell+1$, and assume that $K\subset\bbR^n$
is a compact set of zero $n$-dimensional Lebesgue measure.
Define $\Omega:=\bbR^n\backslash  K$. Then, in the domain $\Om$,
the Friedrichs and Krein--von Neumann extensions of the polyharmonic operator
$(-\Delta)^\ell$, initially considered on $C^\infty_0(\Om)$, coincide
if and only if $B_{2\ell,2}(K)=0$.
\end{theorem}
%%%%%%%%%%%%%%%%%%%%%%%%%%%%%%%%%%%

For some related results in the punctured space
$\Om:=\bbR^n\backslash \{0\}$, see also the recent article \cite{Ad07}. Moreover, we mention 
that in the case of the Bessel operator $h_\nu = (-d^2/dr^2) + (\nu^2 - (1/4))r^{-2}$ defined on 
$C_0^\infty((0,\infty))$, equality of the Friedrichs and Krein extension of $h_\nu$ in 
$L^2((0,\infty); dr)$ if and only if $\nu = 0$ has been established in \cite{MT07}. (The sufficiency 
of the condition $\nu = 0$ was established earlier in \cite{GKMT01}.)

While this section focused on differential operators, we conclude with a very brief 
remark on half-line Jacobi, that is, tridiagonal (and hence, second-order finite difference) 
operators: As discussed in depth by Simon \cite{Si98}, the Friedrichs and Krein--von Neumann extensions of a minimally defined symmetric half-line Jacobi operator (cf.\ also 
\cite{BC05}) coincide, if and only if the associated Stieltjes moment problem is determinate 
(i.e., has a unique solution) while the corresponding Hamburger moment problem is 
indeterminate (and hence has uncountably many solutions).

%%%%%%%%%%%%%%%%%%%%%%%%%%%%%%%%%%%%%%
%%%%%%%%%%%%%%%%%%%%%%%%%%%%%%%%%%%%%%
\section{Examples}
\label{s1vi}
%%%%%%%%%%%%%%%%%%%%%%%%%%%%%%%%%%%%%%
%%%%%%%%%%%%%%%%%%%%%%%%%%%%%%%%%%%%%%

%%%%%%%%%%%%%%%%%%%%%%%%%%%%%%%%%%%%%%
\subsection{The Case of a Bounded Interval $(a,b)$, $-\infty < a < b < \infty$, $V=0$.}
%%%%%%%%%%%%%%%%%%%%%%%%%%%%%%%%%%%%%%

We briefly recall the essence of the one-dimensional example $\Om=(a,b)$, 
$-\infty < a < b < \infty$, and $V=0$. This was first discussed in detail by \cite{AS80} 
and \cite[Sect.\ 2.3]{Fu80} (see also \cite[Sect.\ 3.3]{FOT94}). 

Consider the minimal operator $-\Delta_{min,(a,b)}$ 
in $L^2((a,b);dx)$, given by
\begin{align}
& -\Delta_{min,(a,b)} u = -u'',  \no \\
& \;u \in \dom(-\Delta_{min,(a,b)}) 
=\big\{v \in L^2((a,b);dx) \,\big|\, v, v' \in AC([a,b]);    \lb{10.1} \\ 
& \hspace*{1.7cm} v(a)=v'(a)=v(b)=v'(b)=0; \, 
v'' \in L^2((a,b);dx)\big\},     \no
\end{align}
where $AC([a,b])$ denotes the set of absolutely continuous functions on $[a,b]$. Evidently, 
\begin{equation}
-\Delta_{min,(a,b)} = \ol{- \f{d^2}{dx^2}\bigg|_{C_0^\infty((a,b))}} \, , 
\end{equation}
and one can show that 
\begin{equation}
-\Delta_{min,(a,b)} \geq [\pi/(b-a)]^2 I_{L^2((a,b);dx)}. 
\end{equation}
In addition, one infers that 
\begin{equation}
(-\Delta_{min,(a,b)})^* = -\Delta_{max,(a,b)},
\end{equation}
where 
\begin{align} \lb{10.5}
& -\Delta_{max,(a,b)} u = -u'',   \no \\ 
& \; u \in \dom(-\Delta_{max,(a,b)}) = \big\{v \in L^2((a,b);dx) \,\big|\, v, v' \in AC([a,b]);  \, 
%  \no \\
% & \hspace*{6.8cm}  
v'' \in L^2((a,b);dx)\big\}.   
\end{align}
In particular,
\begin{equation}
{\rm def} (-\Delta_{min,(a,b)}) = (2,2) \, \text{ and } \, 
\ker ((-\Delta_{min,(a,b)})^*) = {\rm lin. \, span}\{1, x\}. 
\end{equation}

The Friedrichs (equivalently, the Dirichlet) extension $-\Delta_{D,(a,b)}$ of 
$-\Delta_{min,(a,b)}$ is then given by 
\begin{align}
& -\Delta_{D,(a,b)} u = -u'', \no \\
& \; u \in \dom(-\Delta_{D,(a,b)}) 
=\big\{v \in L^2((a,b);dx) \,\big|\, v, v' \in AC([a,b]);    \lb{10.7} \\ 
& \hspace*{3.85cm} v(a)=v(b)=0; \, v'' \in L^2((a,b);dx)\big\}.    \no 
\end{align}
In addition,
\begin{equation}
\sigma(-\Delta_{D,(a,b)}) = \{j^2 \pi^2 (b-a)^{-2}\}_{j\in\bbN}, 
\end{equation}
and
\begin{align} 
\dom \big((-\Delta_{D,(a,b)})^{1/2}\big) % & 
= \big\{v \in L^2((a,b);dx) \,\big|\, v \in AC([a,b]); \,  
v(a)=v(b)=0;  \,   % \no \\ 
% & \hspace*{5.6cm} 
v' \in L^2((a,b);dx)\big\}.
\end{align}
By \eqref{SK}, 
\begin{equation}
\dom(-\Delta_{K,(a,b)}) = \dom(-\Delta_{min,(a,b)}) \dotplus \ker((-\Delta_{min,(a,b)})^*),
\end{equation} 
and hence any $u \in \dom(-\Delta_{K,(a,b)})$ is of the type 
\begin{align}
& u = f + \eta, \quad f \in \dom(-\Delta_{min,(a,b)}), 
\quad \eta (x) = u(a) + [u(b)-u(a)] \bigg(\f{x-a}{b-a}\bigg),   % \no \\ 
%& \hspace*{9.3cm}  
\; x \in (a,b), 
\end{align}
in particular, $f(a)=f'(a)=f(b)=f'(b)=0$. Thus, the Krein--von Neumann extension 
$-\Delta_{K,(a,b)}$ of $-\Delta_{min,(a,b)}$ is given by 
\begin{align}
& -\Delta_{K,(a,b)} u = -u'', \no \\
& \; u \in \dom(-\Delta_{K,(a,b)}) 
=\big\{v \in L^2((a,b);dx) \,\big|\, v, v' \in AC([a,b]);    \lb{10.9} \\ 
& \hspace*{8mm} v'(a)=v'(b)=[v(b)-v(a)]/(b-a); \, v'' \in L^2((a,b);dx)\big\}.  \no 
\end{align}
Using the characterization of all self-adjoint extensions of general Sturm--Liouville operators 
in \cite[Theorem 13.14]{We03}, one can also directly verify that 
$-\Delta_{K,(a,b)}$ as given by \eqref{10.9} is a self-adjoint extension of 
$-\Delta_{min,(a,b)}$. 

In connection with \eqref{10.1}, \eqref{10.5}, \eqref{10.7}, and \eqref{10.9},  we also 
note that the well-known fact that 
\begin{equation}
v, v'' \in  L^2((a,b);dx) \, \text{ implies } \, v' \in  L^2((a,b);dx).   \lb{10.10} 
\end{equation}

Utilizing \eqref{10.10}, we briefly consider the quadratic form associated with  
the Krein Laplacian $-\Delta_{K,(a,b)}$. By \eqref{SKform1} and \eqref{SKform2}, 
one infers,
\begin{align}
&\dom\big((-\Delta_{K,(a,b)})^{1/2}\big) = \dom\big((-\Delta_{D,(a,b)})^{1/2}\big) 
\dotplus \ker ((-\Delta_{min,(a,b)})^*),   \lb{SKformab1}  \\
&\big\|(-\Delta_{K,(a,b)})^{1/2}(u+g)\big\|_{L^2((a,b);dx)}^2 
=\big\|(-\Delta_{D,(a,b)})^{1/2} u\big\|_{L^2((a,b);dx)}^2  \no \\
& \quad = ((u+g)',(u+g)')_{L^2((a,b);dx)} - [\ol{g(b)}g'(b) - \ol{g(a)} g'(a)]   \no \\
& \quad = ((u+g)',(u+g)')_{L^2((a,b);dx)} - |[u(b) + g(b)] - [u(a) + g(a)]|^2/(b-a),  \no \\
& \hspace*{3cm} u \in \dom\big((-\Delta_{D,(a,b)})^{1/2}\big), \; 
g \in \ker ((-\Delta_{min,(a,b)})^*).    \lb{SKformab2}
\end{align}

Finally, we turn to the spectrum of $-\Delta_{K,(a,b)}$. The boundary conditions in 
\eqref{10.9} lead to two kinds of (nonnormalized) eigenfunctions and eigenvalue equations 
\begin{align}
\begin{split}
& \psi(k,x) = \cos(k(x-[(a+b)/2])), \quad 
k \sin(k(b-a)/2) = 0,  \\  
& k_{K,(a,b),j} = (j+1)\pi/(b-a), \;  j=-1, 1, 3, 5, \dots,  
\end{split}
\end{align}
and
\begin{align}
& \phi(k,x) = \sin(k(x-[(a+b)/2])), \quad 
k(b-a)/2 = \tan(k(b-a)/2) , \no \\ 
& k_{K,(a,b),0} =0, \;  j\pi < k_{K,(a,b),j} < (j+1)\pi, \; j=2, 4, 6, 8, \dots,  \\ 
& \lim_{\ell\to\infty} [k_{K,(a,b),2\ell} - ((2\ell +1) \pi/(b-a))] =0.  \no 
\end{align}
The associated eigenvalues of $-\Delta_{K,(a,b)}$ are thus given by 
\begin{equation}
\sigma(-\Delta_{K,(a,b)}) = \{0\} \cup \{k_{K,(a,b),j}^2\}_{j\in\bbN},  
\end{equation}
where the eigenvalue $0$ of $-\Delta_{K,(a,b)}$ is of multiplicity two, but the remaining nonzero eigenvalues of $-\Delta_{K,(a,b)}$ are all simple. 

%%%%%%%%%%%%%%%%%%%%%%%%%%%%%%%%%%%%%%
\subsection{The Case of the Ball $B_n(0;R)$, $R>0$, in $\bbR^n$, $n\geq 2$, $V=0$.}
%%%%%%%%%%%%%%%%%%%%%%%%%%%%%%%%%%%%%%

In this subsection, we consider in great detail the scenario when
the domain $\Omega$ equals a ball of radius $R>0$ (for convenience, centered at the origin) in $\bbR^n$,  
\begin{equation}
\Om=B_n(0;R)\subset\bbR^n, \quad R>0, \, n\ge 2.    \lb{10.19}
\end{equation}
Since both the domain $B_n(0;R)$ in \eqref{10.19}, as well as the Laplacian $-\Delta$ are invariant under rotations in $\bbR^n$ centered at the origin, we will employ the (angular momentum)  decomposition of $L^2(B_n(0;R); d^nx)$ into the direct sum of tensor products 
\begin{align}
& L^2(B_n(0;R); d^nx) = L^2((0,R); r^{n-1}dr) \otimes L^2(S^{n-1}; d\omega_{n-1}) 
= \bigoplus_{\ell\in\bbN_0} \cH_{n,\ell,(0,R)},   \lb{10.20} \\
& \cH_{n,\ell,(0,R)} = L^2((0,R); r^{n-1}dr) \otimes \cK_{n,\ell}, 
\quad \ell \in \bbN_0, \; n\geq 2,   \lb{10.21}
\end{align}
where $S^{n-1}= \partial B_n(0;1)=\{x\in\bbR^n\,|\, |x|=1\}$ denotes the $(n-1)$-dimensional unit sphere in $\bbR^n$, $d\omega_{n-1}$ represents the surface measure on $S^{n-1}$, $n\geq 2$, and $\cK_{n,\ell}$ denoting the eigenspace of the Laplace--Beltrami operator 
$-\Delta_{S^{n-1}}$ in $L^2(S^{n-1}; d\omega_{n-1})$ corresponding to the $\ell$th eigenvalue $\kappa_{n,\ell}$ of $-\Delta_{S^{n-1}}$ counting multiplicity,
\begin{align}
\begin{split} 
& \kappa_{n,\ell} = \ell(\ell + n-2), \\ 
& \dim(\cK_{n,\ell}) 
= \f{(2\ell+n-2)\Gamma(\ell+n-2)}{\Gamma(\ell+1)\Gamma(n-1)} : =d_{n,\ell}, 
\quad \ell\in\bbN_0, \; n\geq 2   \lb{10.22}
\end{split} 
\end{align}
(cf.\ \cite[p.\ 4]{Mu66}). In other words, $\cK_{n,\ell}$ is spanned by the $n$-dimensional spherical harmonics of degree $\ell\in\bbN_0$. For more details in this connection we refer to \cite[App.\ to Sect.\ X.1]{RS75} and \cite[Ch.\ 18]{We03}.

As a result, the minimal Laplacian in $L^2(B_n(0;R);d^n x)$ can be decomposed as follows  
\begin{align}
\begin{split} 
& -\Delta_{min,B_n(0;R)}= \ol{-\Delta|_{C_0^\infty(B_n(0;R))}} 
= \bigoplus_{\ell\in\bbN_0} H_{n,\ell,min}^{(0)} \otimes I_{\cK_{n,\ell}},  \\ 
& \; \dom (-\Delta_{min,B_n(0;R)}) = H^2_0(B_n(0;R)),    \lb{10.23}
\end{split}
\end{align}
where $H_{n,\ell,min}^{(0)}$ in $L^2((0,R); r^{n-1}dr)$ are given by
\begin{equation}
H_{n,\ell,min}^{(0)} = \ol{\bigg(-\f{d^2}{dr^2} - \f{n-1}{r}\f{d}{dr} 
+ \f{\kappa_{n,\ell}}{r^2}\bigg)_{C_0^\infty((0,R))}} \, , \quad \ell \in \bbN_0.
\end{equation}
Using the unitary operator $U_n$ defined by 
\begin{equation}
U_n \colon \begin{cases}   
L^2((0,R); r^{n-1}dr) \to L^2((0,R); dr),  \\
\hspace*{2.55cm} \phi  \mapsto (U_n \phi)(r) = r^{(n-1)/2} \phi(r), 
\end{cases}
\end{equation} 
it will also be convenient to consider the unitary transformation of $H_{n,\ell,min}^{(0)}$ given by 
\begin{equation}
h_{n,\ell,min}^{(0)} = U_n H_{n,\ell,min}^{(0)} U_n^{-1}, \quad \ell \in \bbN_0, 
\end{equation} 
where 
\begin{align}
& h_{n,0,min}^{(0)} = -\f{d^2}{dr^2} + \f{(n-1)(n-3)}{4 r^2}, \quad 0<r<R,  \no \\
& \dom\big(h_{n,0,min}^{(0)}\big) = \big\{f \in L^2((0,R); dr) \,\big|\, f, f' \in AC([\varepsilon,R]) \, 
\text{for all $\varepsilon>0$};   \no \\  
& \hspace*{6.1cm} f(R_-)=f'(R_-)=0, \,  f_0=0;    \\
& \hspace*{3.1cm}   (-f'' +[(n-1)(n-3)/4]r^{-2}f) \in L^2((0,R); dr)\big\}  \, 
\text{ for $n=2,3$},   \no \\
& h_{n,\ell,min}^{(0)} = -\f{d^2}{dr^2} + \f{4 \kappa_{n,\ell} + (n-1)(n-3)}{4 r^2},  
\quad 0<r<R, \no \\
& \dom\big(h_{n,\ell,min}^{(0)}\big) 
= \big\{f \in L^2((0,R); dr) \,|\, f, f' \in AC([\varepsilon,R]) \, 
\text{for all $\varepsilon>0$};   \no \\ 
& \hspace*{7.3cm} f(R_-)=f'(R_-)=0;    \\ 
& \hspace*{1.9cm} (-f'' +[\kappa_{n,\ell} + ((n-1)(n-3)/4)]r^{-2}f) \in L^2((0,R); dr)\big\}  \no \\
& \hspace*{5.4cm} \text{ for $\ell\in\bbN$, $n\geq 2$ and $\ell=0$, $n\geq 4$.}  \no
\end{align}
In particular, for $\ell\in\bbN$, $n\geq 2$, and $\ell=0$, $n\geq 4$, one obtains 
\begin{align}
\begin{split}
& h_{n,\ell,min}^{(0)} = \ol{\bigg(-\f{d^2}{dr^2} + \f{4 \kappa_{n,\ell} 
+ (n-1)(n-3)}{4 r^2}\bigg)\bigg|_{C_0^\infty((0,R))}} \\ 
& \hspace*{3.2cm} \text{ for $\ell \in \bbN$, $n\geq 2$, and $\ell=0$, $n\geq 4$.}  
\end{split} 
\end{align}
On the other hand, for $n=2,3$, the domain of the closure of 
$h_{n,0,min}^{(0)}\big|_{C_0^\infty((0,R))}$ is strictly contained in that of 
$\dom\big(h_{n,0,min}^{(0)}\big)$, and in this case one obtains for 
\begin{equation}
\hatt h_{n,0, min}^{(0)} = \ol{\bigg(-\f{d^2}{dr^2} + \f{(n-1)(n-3)}{4 r^2}\bigg)\bigg|_{C_0^\infty((0,R))}}, 
\quad n=2,3, 
\end{equation}
that
\begin{align}
& \hatt h_{n,0, min}^{(0)} = -\f{d^2}{dr^2} + \f{(n-1)(n-3)}{4 r^2}, \quad 0<r<R, \no \\
& \dom\big(\hatt h_{n,0, min}^{(0)}\big) = 
\big\{f\in L^2((0,R); dr) \,\big|\, f, f' \in AC([\varepsilon,R]) \, \text{for all $\varepsilon>0$}; 
\no \\ 
& \hspace*{5.35cm} f(R_-)=f'(R_-)=0, \, f_0=f'_0=0; \\ 
& \hspace*{3.3cm} (-f'' +[(n-1)(n-3)/4]r^{-2}f) \in L^2((0,R); dr)\big\}.   \no 
\end{align}
Here we used the abbreviations (cf.\ \cite{BG85} for details)
\begin{align}
\begin{split} 
& f_0 = \begin{cases} \lim_{r\downarrow 0} [-r^{1/2}\ln(r)]^{-1} f(r), & n=2, \\
f(0_+), & n=3, \end{cases} \\
& f'_0 = \begin{cases} \lim_{r\downarrow 0} r^{-1/2} [f(r) + f_0 r^{1/2}\ln(r)], & n=2, \\
f'(0_+), & n=3.  \end{cases}   \lb{10.33}
\end{split} 
\end{align}
We also recall the adjoints of $h_{n,\ell, min}^{(0)}$ which are given by
\begin{align}
& \big(h_{n,0,min}^{(0)}\big)^* = -\f{d^2}{dr^2} + \f{(n-1)(n-3)}{4 r^2}, 
\quad 0<r<R, \no \\
& \dom\big(\big(h_{n,0,min}^{(0)}\big)^*\big) = \big\{f \in L^2((0,R); dr) \,\big|\, f, f' \in 
AC([\varepsilon,R]) \, \text{for all $\varepsilon>0$};  \lb{10.34}  \\
& \hspace*{.4cm}   f_0=0; \, (-f'' +[(n-1)(n-3)/4]r^{-2}f) \in L^2((0,R); dr)\big\}  \, 
\text{ for $n=2,3$},  \no \\
& \big(h_{n,\ell,min}^{(0)}\big)^* = -\f{d^2}{dr^2} + \f{4 \kappa_{n,\ell} + (n-1)(n-3)}{4 r^2},   \quad 0<r<R,   \no \\
& \dom\big(\big(h_{n,\ell,min}^{(0)}\big)^*\big) = \big\{f \in L^2((0,R); dr) \,\big|\, f, f' \in 
AC([\varepsilon,R]) \, 
\text{for all $\varepsilon>0$};     \\ 
& \hspace*{2.35cm} (-f'' +[\kappa_{n,\ell} + ((n-1)(n-3)/4)]r^{-2}f) \in L^2((0,R); dr)\big\}   
 \no \\
& \hspace*{5.85cm} \text{ for $\ell\in\bbN$, $n\geq 2$ and $\ell=0$, $n\geq 4$.}  \no
\end{align}
In particular,
\begin{equation}
h_{n,\ell,max}^{(0)} = \big(h_{n,\ell,min}^{(0)}\big)^*, \quad \ell\in\bbN_0, \; n\geq 2. 
\end{equation}
All self-adjoint extensions of $h_{n,\ell, min}^{(0)}$ are given by the following 
one-parameter families $h_{n,\ell, \alpha_{n,\ell}}^{(0)}$, 
$\alpha_{n,\ell} \in\bbR\cup\{\infty\}$,
\begin{align}
& h_{n,0,\alpha_{n,0}}^{(0)} = -\f{d^2}{dr^2} + \f{(n-1)(n-3)}{4 r^2}, \quad 0<r<R, \no \\
& \dom\big(h_{n,0,\alpha_{n,0}}^{(0)}\big) = \big\{f \in L^2((0,R); dr) \,\big|\, f, f' \in AC([\varepsilon,R]) \, \text{for all $\varepsilon>0$};   \no \\ 
& \hspace*{5.6cm}  f'(R_-)+\alpha_{n,0}f(R_-)=0, \,  f_0=0;    \lb{10.37} \\
& \hspace*{3cm}  (-f'' +[(n-1)(n-3)/4]r^{-2}f) \in L^2((0,R); dr)\big\}  \, 
\text{ for $n=2,3$},   \no \\
& h_{n,\ell,\alpha_{n,\ell}}^{(0)} = -\f{d^2}{dr^2} + \f{4 \kappa_{n,\ell} + (n-1)(n-3)}{4 r^2},   
\quad 0<r<R,  \no \\
& \dom\big(h_{n,\ell,\alpha_{n,\ell}}^{(0)}\big) = \big\{f \in L^2((0,R); dr) \,\big|\, f, f' \in 
AC([\varepsilon,R]) \, \text{for all $\varepsilon>0$};   \no \\
& \hspace*{6.73cm} f'(R_-)+\alpha_{n,\ell}f(R_-)=0;   \lb{10.38}  \\ 
& \hspace*{1.95cm} (-f'' +[\kappa_{n,\ell} + ((n-1)(n-3)/4)]r^{-2}f) \in L^2((0,R); dr)\big\} 
\no \\
& \hspace*{5.5cm} \text{ for $\ell\in\bbN$, $n\geq 2$ and $\ell=0$, $n\geq 4$.}  \no
\end{align}
Here, in obvious notation, the boundary condition for $\alpha_{n,\ell}=\infty$ simply represents the Dirichlet boundary condition $f(R_-)=0$. In particular, the Friedrichs or Dirichlet extension $h_{n,\ell, D}^{(0)}$ of $h_{n,\ell, min}^{(0)}$ is given by 
$h_{n,\ell, \infty}^{(0)}$, that is, by
\begin{align}
& h_{n,0,D}^{(0)} = -\f{d^2}{dr^2} + \f{(n-1)(n-3)}{4 r^2}, \quad 0<r<R,  \no \\
& \dom\big(h_{n,0,D}^{(0)}\big) = \big\{f \in L^2((0,R); dr) \,\big|\, f, f' \in AC([\varepsilon,R]) \, \text{for all $\varepsilon>0$}; \, f(R_-)=0,   \no \\
& \hspace*{.4cm}   f_0=0; \, (-f'' +[(n-1)(n-3)/4]r^{-2}f) \in L^2((0,R); dr)\big\}  \, 
\text{ for $n=2,3$},   \lb{10.39} \\
& h_{n,\ell,D}^{(0)} = -\f{d^2}{dr^2} + \f{4 \kappa_{n,\ell} + (n-1)(n-3)}{4 r^2},   
\quad 0<r<R,  \no \\
& \dom\big(h_{n,\ell,D}^{(0)}\big) = \big\{f \in L^2((0,R); dr) \,\big|\, f, f' \in 
AC([\varepsilon,R]) \, \text{for all $\varepsilon>0$}; \, f(R_-)=0;   \no   \\ 
& \hspace*{3.5cm} (-f'' +[\kappa_{n,\ell} + ((n-1)(n-3)/4)]r^{-2}f) \in L^2((0,R); dr)\big\}  
\no \\
& \hspace*{7.1cm} \text{ for $\ell\in\bbN$, $n\geq 2$ and $\ell=0$, $n\geq 4$.}   \lb{10.40}
\end{align}
To find the boundary condition for the Krein--von Neumann extension $h_{n,\ell,K}^{(0)}$ of $h_{n,\ell,min}^{(0)}$, that is, to find the corresponding boundary condition parameter 
$\alpha_{n,\ell,K}$ in \eqref{10.37}, \eqref{10.38}, we recall \eqref{SK}, that is,
\begin{equation}
\dom\big(h_{n,\ell,K}^{(0)}\big) = \dom\big(h_{n,\ell,min}^{(0)}\big) \dotplus 
\ker\big(\big(h_{n,\ell,min}^{(0)}\big)^*\big). 
\end{equation}
By inspection, the general solution of 
\begin{equation}
\bigg(-\f{d^2}{dr^2} + \f{4 \kappa_{n,\ell} + (n-1)(n-3)}{4 r^2}\bigg) \psi(r) = 0,  
\quad r \in (0,R),   \lb{10.41}
\end{equation}
is given by
\begin{equation}
\psi(r) = A r^{\ell +[(n-1)/2]} + B r^{-\ell - [(n-3)/2]}, \quad A, B \in \bbC, \; r \in (0,R).   
\lb{10.42}
\end{equation}
However, for $\ell\geq 1$, $n \geq 2$ and for $\ell=0$, $n\geq 4$, the requirement 
$\psi \in L^2((0,R); dr)$ requires $B=0$ in \eqref{10.42}. Similarly, also the requirement 
$\psi_0=0$ (cf.\ \eqref{10.34}) for $\ell=0$, $n=2,3$, enforces $B=0$ in \eqref{10.42}.  

Hence, any $u \in \dom\big(h_{n,\ell,K}^{(0)}\big)$ is of the type 
\begin{equation}
u = f + \eta, \quad f \in \dom\big(h_{n,\ell,min}^{(0)}\big), 
\quad \eta (r) = u(R_-)r^{\ell +[(n-1)/2]}, \; r \in [0,R), 
\end{equation}
in particular, $f(R_-)=f'(R_-)=0$. Denoting by $\alpha_{n,\ell,K}$ the boundary condition parameter for $h_{n,\ell,K}^{(0)}$ one thus computes
\begin{equation}
-\alpha_{n,\ell,K} = \f{u'(R_-)}{u(R_-)} = \f{\eta'(R_-)}{\eta(R_-)} = [\ell + ((n-1)/2)]/R.
\end{equation}
Thus, the Krein--von Neumann extension $h_{n,\ell,K}^{(0)}$ of $h_{n,\ell,min}^{(0)}$ is given by 
\begin{align}
& h_{n,0,K}^{(0)} = -\f{d^2}{dr^2} + \f{(n-1)(n-3)}{4 r^2},  \quad 0<r<R,  \no \\
& \dom\big(h_{n,0,K}^{(0)}\big) = \big\{f \in L^2((0,R); dr) \,\big|\, f, f' \in AC([\varepsilon,R]) \, \text{for all $\varepsilon>0$}; \no \\ 
& \hspace*{3.6cm}  f'(R_-)- [(n-1)/2]R^{-1} f(R_-)=0, \,  f_0=0;    \lb{10.46} \\
& \hspace*{3cm}   (-f'' +[(n-1)(n-3)/4]r^{-2}f) \in L^2((0,R); dr)\big\}  \, 
\text{ for $n=2,3$},   \no \\
& h_{n,\ell,K}^{(0)} = -\f{d^2}{dr^2} + \f{4 \kappa_{n,\ell} + (n-1)(n-3)}{4 r^2},   
\quad 0<r<R,  \no \\
& \dom\big(h_{n,\ell,K}^{(0)}\big) = \big\{f \in L^2((0,R); dr) \,\big|\, f, f' \in 
AC([\varepsilon,R]) \, \text{for all $\varepsilon>0$};   \no \\ 
& \hspace*{3.94cm}  f'(R_-)-[\ell +((n-1)/2)] R^{-1} f(R_-)=0;    \lb{10.47}  \\ 
& \hspace*{1.65cm} (-f'' +[\kappa_{n,\ell} + ((n-1)(n-3)/4)]r^{-2}f) \in L^2((0,R); dr)\big\} 
\no  \\
& \hspace*{5.15cm} \text{ for $\ell\in\bbN$, $n\geq 2$ and $\ell=0$, $n\geq 4$.}  \no
\end{align}

Next we  briefly turn to the eigenvalues of $h_{n,\ell,D}^{(0)}$ and $h_{n,\ell,K}^{(0)}$. 
In analogy to \eqref{10.41}, the solution $\psi$ of 
\begin{equation}
\bigg(-\f{d^2}{dr^2} + \f{4 \kappa_{n,\ell} + (n-1)(n-3)}{4 r^2} -z \bigg) \psi(r,z) = 0,  
\quad r \in (0,R), 
\lb{10.48}
\end{equation}
satisfying the condition $\psi(\cdot,z) \in L^2((0,R); dr)$ for $\ell=0$, $n\geq 4$ and 
$\psi_0(z)=0$ (cf.\ \eqref{10.34}) for $\ell=0$, $n=2,3$, yields
\begin{equation}
\psi(r,z) = A r^{1/2} J_{l+[(n-2)/2]}(z^{1/2} r), \quad A \in \bbC, \; r \in (0,R),   
\lb{10.49}
\end{equation}
Here $J_{\nu}(\cdot)$ denotes the Bessel function of the first kind and order $\nu$ 
(cf.\ \cite[Sect.\ 9.1]{AS72}).
Thus, by the boundary condition $f(R_-)=0$ in \eqref{10.39}, \eqref{10.40}, the eigenvalues of the Dirichlet extension $h_{n,\ell,D}^{(0)}$ are determined by the equation 
$\psi(R_-,z)=0$, and hence by 
\begin{equation}
J_{l+[(n-2)/2]}(z^{1/2}R) = 0. 
\end{equation}
Following \cite[Sect.\ 9.5]{AS72}, we denote the zeros of $J_{\nu}(\cdot)$ by $j_{\nu,k}$, 
$k\in\bbN$, and hence obtain for the spectrum of $h_{n,\ell,F}^{(0)}$,
\begin{equation}
\sigma\big(h_{n,\ell,D}^{(0)}\big) 
= \big\{\lambda^{(0)}_{n,\ell,D,k}\big\}_{k\in\bbN}  
= \big\{ j_{\ell+[(n-2)/2],k}^2R^{-2}\big\}_{k\in\bbN}, \quad \ell \in \bbN_0, 
\; n\geq 2.   \lb{10.51}
\end{equation}
Each eigenvalue of of $h_{n,\ell,D}^{(0)}$ is simple.

Similarly, by the boundary condition $f'(R_-) - [\ell +((n-1)/2)]R^{-1}f(R_-)=0$ in 
\eqref{10.46}, \eqref{10.47}, the eigenvalues of the Krein--von Neumann extension 
$h_{n,\ell,K}^{(0)}$ are determined by the equation
\begin{equation}
\psi'(R,z) - [\ell+ ((n-1)/2)] \psi(R,z) = - A z^{1/2} R^{1/2} J_{\ell+(n/2)}(z^{1/2}R)=0
\end{equation}
(cf.\ \cite[eq.\ (9.1.27)]{AS72}), and hence by 
\begin{equation}
z^{1/2} J_{\ell+(n/2)}(z^{1/2}R) = 0. 
\end{equation}
Thus, one obtains for the spectrum of $h_{n,\ell,K}^{(0)}$,
\begin{equation}
\sigma\big(h_{n,\ell,K}^{(0)}\big) 
= \{0\} \cup \big\{\lambda^{(0)}_{n,\ell,K,k}\big\}_{k\in\bbN}  
= \{0\} \cup \big\{ j_{\ell+(n/2),k}^2R^{-2}\big\}_{k\in\bbN}, \quad \ell \in \bbN_0, \; n\geq 2. 
\lb{10.54}
\end{equation}
Again, each eigenvalue of $h_{n,\ell,K}^{(0)}$ is simple, and 
$\eta (r) = Cr^{\ell +[(n-1)/2]}$, $C\in\bbC$, represents the (unnormalized) eigenfunction 
of $h_{n,\ell,K}^{(0)}$ corresponding to the eigenvalue $0$.

Combining Propositions \ref{p2.2a}--\ref{p2.4}, one then obtains 
\begin{align}
& -\Delta_{max,B_n(0;R)} = (-\Delta_{min,B_n(0;R)})^* 
= \bigoplus_{\ell\in\bbN_0} \big(H_{n,\ell,min}^{(0)}\big)^* \otimes I_{\cK_{n,\ell}}, \\
& -\Delta_{D,B_n(0;R)}  
= \bigoplus_{\ell\in\bbN_0} H_{n,\ell,D}^{(0)} \otimes I_{\cK_{n,\ell}}, \\
& -\Delta_{K,B_n(0;R)}  
= \bigoplus_{\ell\in\bbN_0} H_{n,\ell,K}^{(0)} \otimes I_{\cK_{n,\ell}}, 
\end{align}
where (cf.\ \eqref{10.23})
\begin{align}
& H_{n,\ell,max}^{(0)} = \big(H_{n,\ell,min}^{(0)}\big)^* 
= U_n^{-1} \big(h_{n,\ell,min}^{(0)}\big)^* U_n, \quad \ell \in \bbN_0, \\
& H_{n,\ell,D}^{(0)} = U_n^{-1} h_{n,\ell,D}^{(0)} U_n, \quad \ell \in \bbN_0, \\ 
& H_{n,\ell,K}^{(0)} = U_n^{-1} h_{n,\ell,K}^{(0)} U_n, \quad \ell \in \bbN_0.
\end{align}
Consequently,
\begin{align}
& \sigma( -\Delta_{D,B_n(0;R)}) 
= \big\{\lambda^{(0)}_{n,\ell,D,k}\big\}_{\ell\in\bbN_0, k\in\bbN}  
= \big\{ j_{\ell+[(n-2)/2],k}^2R^{-2}\big\}_{\ell\in\bbN_0, k\in\bbN},  \\
& \sigma_{\rm ess}( -\Delta_{D,B_n(0;R)}) = \emptyset,  \\
& \sigma( -\Delta_{K,B_n(0;R)}) 
= \{0\} \cup \big\{\lambda^{(0)}_{n,\ell,K,k}\big\}_{\ell\in\bbN_0, k\in\bbN}  
= \{0\} \cup \big\{  j_{\ell+(n/2),k}^2R^{-2}\big\}_{\ell\in\bbN_0, k\in\bbN},  \\ 
& \dim(\ker( -\Delta_{K,B_n(0;R)})) = \infty, \quad 
\sigma_{\rm ess}( -\Delta_{K,B_n(0;R)}) = \{0\}.  
\end{align}
By \eqref{10.22}, each eigenvalue $\lambda^{(0)}_{n,\ell,D,k}$, $k\in\bbN$, of 
$-\Delta_{D,B_n(0;R)}$ has multiplicity $d_{n,\ell}$ and similarly, again by \eqref{10.22}, each eigenvalue $\lambda^{(0)}_{n,\ell,K,k}$, $k\in\bbN$, of $ -\Delta_{K,B_n(0;R)}$ has multiplicity $d_{n,\ell}$. 

Finally, we briefly turn to the Weyl asymptotics for the eigenvalue counting function  
\eqref{mam-44} associated with the Krein Laplacian $-\Delta_{K,B_n(0;R)}$ for the ball 
$B_n(0;R)$, $R>0$, in $\bbR^n$, $n\geq 2$. We will discuss a direct approach to the Weyl asymptotics that is independent of the general treatment presented in Section \ref{s12}. Due to the smooth nature of the ball, we will obtain an improvement in the remainder term of the Weyl asymptotics of the Krein Laplacian. 

First we recall the well-known fact that in the case of the Dirichlet Laplacian associated with the ball $B_n(0;R)$,  
\begin{align}\label{10.65} 
N^{(0)}_{D,B_n(0;R)}(\lambda) &= (2\pi)^{-n}v_n^2 R^n\lambda^{n/2} 
- (2\pi)^{-(n-1)} v_{n-1} (n/4) v_n R^{n-1} \lambda^{(n-1)/2}   \no \\
& \quad + O\big(\lambda^{(n-2)/2}\big) \, \mbox{ as }\, \lambda\to\infty,  
\end{align}
with $v_n=\pi^{n/2}/ \Gamma((n/2)+1)$ the volume of the unit ball in $\bbR^n$ (and 
$n v_n$ representing the surface area of the unit ball in $\bbR^n$). 

%%%%%%%%%%%%%%%%%%%%%%%%%%%%%%%%%%%%%%%
\begin{proposition}\label{p10.1}
The strictly positive eigenvalues of the Krein Laplacian associated with the ball of radius $R>0$, 
$B_n(0;R)\subset \bbR^n$, $R>0$, $n\geq 2$, satisfy the following Weyl-type 
eigenvalue asymptotics,
\begin{align}\label{10.66}
N^{(0)}_{K,B_n(0;R)}(\lambda)&=(2\pi)^{-n}v_n^2 R^n \lambda^{n/2} 
- (2\pi)^{-(n-1)} v_{n-1} [(n/4) v_n + v_{n-1}] R^{n-1} \lambda^{(n-1)/2}  \no \\
& \quad + O\big(\lambda^{(n-2)/2}\big) \, \mbox{ as }\, \lambda\to\infty.   
\end{align}
\end{proposition}
%%%%%%%%%%%%%%%%%%%%%%%%%%%%%%%%%%%%%%% 
\begin{proof}
From the outset one observes that 
\begin{equation}
\lambda^{(0)}_{n,\ell,D,k} \le \lambda^{(0)}_{n,\ell,K,k} \le \lambda^{(0)}_{n,\ell,D,k+1}, \quad \ell\in\bbN_0, \;
k\in\bbN,
\end{equation}
implying 
\begin{equation}\label{10.67}
N^{(0)}_{K,B_n(0;R)}(\lambda) \le N^{(0)}_{D,B_n(0;R)}(\lambda), \quad \lambda\in\bbR.
\end{equation}
Next, introducing 
\begin{equation}
\cN_\nu (\lambda) :=\begin{cases} \text{the largest $k\in\bbN$ such that 
$j_{\nu,k}^2 R^{-2} \leq \lambda$}, \\
0, \text{ if no such $k\geq 1$ exists,} \end{cases} \quad \lambda \in\bbR,  
\end{equation}
we note the well-known monotonicity of $j_{\nu,k}$ with respect to $\nu$ 
(cf.\ \cite[Sect.\ 15.6, p.\ 508]{Wa96}), implying that for each $\lambda \in \bbR$ 
(and fixed $R>0$), 
\begin{equation}
\cN_{\nu'} (\lambda) \leq \cN_{\nu} (\lambda) \, \text{ for } \, \nu' \geq \nu \geq 0.
\end{equation}
Then one infers 
\begin{equation}
N^{(0)}_{D,B_n(0;R)}(\lambda) = \sum_{\ell\in\bbN_0} d_{n,\ell} \, \cN_{(n/2)-1+\ell} (\lambda), \quad 
N^{(0)}_{K,B_n(0;R)}(\lambda) = \sum_{\ell\in\bbN_0} d_{n,\ell} \, 
\cN_{(n/2)+\ell} (\lambda).
\end{equation}
Hence, using the fact that 
\begin{equation}
d_{n,\ell} = d_{n-1,\ell} + d_{n,\ell-1} 
\end{equation}
(cf.\ \eqref{10.22}), setting $d_{n, -1} =0$, $n\geq 2$, one computes 
\begin{align}
N^{(0)}_{D,B_n(0;R)}(\lambda) &= \sum_{\ell\in\bbN} d_{n,\ell-1} \, \cN_{(n/2)-1+\ell} (\lambda) 
+ \sum_{\ell\in\bbN_0} d_{n-1,\ell} \, \cN_{(n/2)-1+\ell} (\lambda)  \no \\
& \leq \sum_{\ell\in\bbN_0} d_{n,\ell} \, \cN_{(n/2)+\ell} (\lambda) 
+ \sum_{\ell\in\bbN_0} d_{n-1,\ell} \, \cN_{((n-1)/2)-1+\ell} (\lambda)  \no \\
& = N^{(0)}_{K,B_n(0;R)}(\lambda) + N^{(0)}_{D,B_{n-1}(0;R)}(\lambda),
\end{align}
that is,
\begin{equation}
N^{(0)}_{D,B_n(0;R)}(\lambda) \leq N^{(0)}_{K,B_n(0;R)}(\lambda) 
+ N^{(0)}_{D,B_{n-1}(0;R)}(\lambda).   \lb{10.74a}
\end{equation}
Similarly,
\begin{align}
N^{(0)}_{D,B_n(0;R)}(\lambda) &= \sum_{\ell\in\bbN} d_{n,\ell-1} \, 
\cN_{(n/2)-1+\ell} (\lambda) 
+ \sum_{\ell\in\bbN_0} d_{n-1,\ell} \, \cN_{(n/2)-1+\ell} (\lambda)  \no \\
& \geq \sum_{\ell\in\bbN_0} d_{n,\ell} \, \cN_{(n/2)+\ell} (\lambda) 
+ \sum_{\ell\in\bbN_0} d_{n-1,\ell} \, \cN_{((n-1)/2)+\ell} (\lambda)  \no \\
& = N^{(0)}_{K,B_n(0;R)}(\lambda) + N^{(0)}_{K,B_{n-1}(0;R)}(\lambda), 
\end{align}
that is,
\begin{equation}
N^{(0)}_{D,B_n(0;R)}(\lambda) \geq N^{(0)}_{K,B_n(0;R)}(\lambda) 
+ N^{(0)}_{K,B_{n-1}(0;R)}(\lambda),   \lb{10.76a}
\end{equation}
and hence,
\begin{equation}
N^{(0)}_{K,B_{n-1}(0;R)}(\lambda) \leq \big[N^{(0)}_{D,B_n(0;R)}(\lambda) 
- N^{(0)}_{K,B_n(0;R)}(\lambda)\big] \leq N^{(0)}_{D,B_{n-1}(0;R)}(\lambda).   \lb{10.77A}
\end{equation}
Thus, using 
\begin{equation}
0 \leq \big[N^{(0)}_{D,B_n(0;R)}(\lambda) - N^{(0)}_{K,B_n(0;R)}(\lambda)\big] 
\leq N^{(0)}_{D,B_{n-1}(0;R)}(\lambda) = \Oh\big(\lambda^{(n-1)/2}\big) \, 
\text{ as $\lambda\to\infty$,}
\end{equation}
one first concludes that 
$\big[N^{(0)}_{D,B_n(0;R)}(\lambda) - N^{(0)}_{K,B_n(0;R)}(\lambda)\big] 
= \Oh\big(\lambda^{(n-1)/2}\big)$ as $\lambda\to\infty$, and hence using \eqref{10.65}, 
\begin{equation}
N^{(0)}_{K,B_n(0;R)}(\lambda) 
= (2\pi)^{-n}v_n^2 R^n \lambda^{n/2} + \Oh\big(\lambda^{(n-1)/2}\big)
\, \mbox{ as }\, \lambda\to\infty.    
\end{equation}
This type of reasoning actually yields a bit more: Dividing \eqref{10.77A} by $\lambda^{(n-1)/2}$, and using that both, 
$N^{(0)}_{D,B_{n-1}(0;R)}(\lambda)$ and $N^{(0)}_{K,B_{n-1}(0;R)}(\lambda)$ have the same leading asymptotics $(2\pi)^{-(n-1)}v_{n-1}^2 R^{n-1}\lambda^{(n-1)/2}$ as 
$\lambda \to\infty$, one infers, using \eqref{10.65} again, 
\begin{align}
N^{(0)}_{K,B_n(0;R)}(\lambda) &= N^{(0)}_{D,B_n(0;R)}(\lambda)  
- \big[N^{(0)}_{D,B_n(0;R)}(\lambda) - N^{(0)}_{K,B_n(0;R)}(\lambda)\big]  \no \\
&= N^{(0)}_{D,B_n(0;R)}(\lambda) - (2\pi)^{-(n-1)} v_{n-1}^2 R^{n-1} \lambda^{(n-1)/2} 
+ \oh\big(\lambda^{(n-1)/2}\big)   \no \\
&=(2\pi)^{-n}v_n^2 R^n \lambda^{n/2} 
- (2\pi)^{-(n-1)} v_{n-1} [(n/4) v_n + v_{n-1}] R^{n-1} \lambda^{(n-1)/2}   \no \\
& \quad + \oh\big(\lambda^{(n-1)/2}\big) \, \mbox{ as }\, \lambda\to\infty.  
 \lb{10.79A}
\end{align}
Finally, it is possible to improve the remainder term in \eqref{10.79A} from 
$\oh\big(\lambda^{(n-1)/2}\big)$ to $O\big(\lambda^{(n-2)/2}\big)$ as follows: Replacing 
$n$ by $n-1$ in \eqref{10.74a} yields
\begin{equation}
N^{(0)}_{D,B_{n-1}(0;R)}(\lambda) \leq N^{(0)}_{K,B_{n-1}(0;R)}(\lambda) 
+ N^{(0)}_{D,B_{n-2}(0;R)}(\lambda).   \lb{10.80A}
\end{equation}
Insertion of \eqref{10.80A} into \eqref{10.76a} permits one to eliminate 
$N_{K,B_{n-1}(0;R)}^{(0)}$ as follows:  
\begin{equation}
N^{(0)}_{D,B_n(0;R)}(\lambda) \geq N^{(0)}_{K,B_n(0;R)}(\lambda) 
+ N^{(0)}_{D,B_{n-1}(0;R)}(\lambda) - N^{(0)}_{D,B_{n-2}(0;R)}(\lambda),   \lb{10.81A}
\end{equation}
implies  
\begin{align}
\begin{split}
& \big[N^{(0)}_{D,B_{n}(0;R)}(\lambda) - N^{(0)}_{D,B_{n-1}(0;R)}(\lambda)\big]  \leq 
N^{(0)}_{K,B_n(0;R)}(\lambda)   \\
& \quad \leq \big[N^{(0)}_{D,B_{n}(0;R)}(\lambda) - N^{(0)}_{D,B_{n-1}(0;R)}(\lambda)\big]+ N^{(0)}_{D,B_{n-2}(0;R)}(\lambda), 
\end{split} 
\end{align}
and hence,
\begin{equation}
0 \leq N^{(0)}_{K,B_n(0;R)}(\lambda) - 
\big[N^{(0)}_{D,B_{n}(0;R)}(\lambda) - N^{(0)}_{D,B_{n-1}(0;R)}(\lambda)\big]  \leq
N^{(0)}_{D,B_{n-2}(0;R)}(\lambda). 
\end{equation}
Thus, $N^{(0)}_{K,B_n(0;R)}(\lambda) - \big[N^{(0)}_{D,B_{n}(0;R)}(\lambda) 
- N^{(0)}_{D,B_{n-1}(0;R)}(\lambda)\big] = \Oh\big(\lambda^{(n-2)/2}\big)$ as 
$\lambda\to \infty$, proving \eqref{10.66}.  
\end{proof}
%%%%%%%%%%%%%%%%%%%%%%%%%%%%%%%%%%%%%%

Due to the smoothness of the domain $B_n(0;R)$, the remainder terms in \eqref{10.66} represent a marked improvement over the general result \eqref{kko-12} for domains 
$\Om$ satisfying Hypothesis \ref{h.VK}. A comparison of the second term in the asymptotic relations \eqref{10.65} and \eqref{10.66} exhibits the difference between Dirichlet and Krein--von Neumann eigenvalues.

%%%%%%%%%%%%%%%%%%%%%%%%%%%%%%%%%%%%%%
\subsection{The Case $\Om=\bbR^n\backslash\{0\}$, $n=2,3$, $V=0$.}   \lb{s10.3}
%%%%%%%%%%%%%%%%%%%%%%%%%%%%%%%%%%%%%%

In this subsection we consider the following minimal operator 
$-\Delta_{min, \bbR^n\backslash\{0\}}$ in $L^2(\bbR^n;d^nx)$, $n=2,3$,
\begin{equation}
-\Delta_{min, \bbR^n\backslash\{0\}}  
=\ol{-\Delta\big|_{C^\infty_0 (\bbR^n\backslash\{0\})}}\geq 0, \quad n=2,3. 
\lb{10.74}
\end{equation}
Then
\begin{align}
\begin{split} 
& H_{F,\bbR^2\backslash\{0\}}=H_{K,\bbR^2\backslash\{0\}} 
=-\Delta,   \\
& \dom(H_{F,\bbR^2\backslash\{0\}})
= \dom(H_{K,\bbR^2\backslash\{0\}})=H^{2}(\bbR^2) \, \text{ if } \, n=2 \lb{10.75}
\end{split} 
\end{align}
is the unique nonnegative self-adjoint  
extension of $-\Delta_{min, \bbR^2\backslash\{0\}}$ in $L^2(\bbR^2;d^2x)$ and 
\begin{align}
\begin{split}
& H_{F,\bbR^3\backslash\{0\}}=H_{D,\bbR^3\backslash\{0\}} =-\Delta,    \\ 
& \dom(H_{F,\bbR^3\backslash\{0\}}) = \dom(H_{D,\bbR^3\backslash\{0\}})
=H^{2}(\bbR^3) \, \text{ if } \, n=3, \lb{10.76} 
\end{split} \\
&H_{K,\bbR^3\backslash\{0\}} = H_{N,\bbR^3\backslash\{0\}}
=U^{-1}h_{0,N,\bbR_+}^{(0)}U \oplus\bigoplus_{\ell\in\bbN} 
U^{-1}h_{\ell,\bbR_+}^{(0)} U \, \text{ if } \, n=3,  \lb{10.77}
\end{align} 
where $H_{D,\bbR^3\backslash\{0\}}$ and $H_{N,\bbR^3\backslash\{0\}}$ denote the Dirichlet and Neumann\footnote{The Neumann extension $H_{N,\bbR^3\backslash\{0\}}$ of 
$-\Delta_{min, \bbR^n\backslash\{0\}}$, associated with a Neumann boundary condition, 
in honor of Carl Gottfried Neumann, should of course not be confused with the Krein--von Neumann 
extension $H_{K,\bbR^3\backslash\{0\}}$ of $-\Delta_{min, \bbR^n\backslash\{0\}}$.} extension of 
$-\Delta_{min, \bbR^n\backslash\{0\}}$ in $L^2(\bbR^3; d^3 x)$, respectively. 
Here we used the angular momentum decomposition (cf.\ also \eqref{10.20}, \eqref{10.21}), 
\begin{align}
& L^2(\bbR^n; d^nx) = L^2((0,\infty); r^{n-1}dr) \otimes L^2(S^{n-1}; d\omega_{n-1}) 
= \bigoplus_{\ell\in\bbN_0} \cH_{n,\ell,(0,\infty)},   \lb{10.77a} \\
& \cH_{n,\ell,(0,\infty)} = L^2((0,\infty); r^{n-1}dr) \otimes \cK_{n,\ell}, 
\quad \ell \in \bbN_0, \; n=2,3.   \lb{10.77b}
\end{align}
Moreover, we abbreviated $\bbR_+=(0,\infty)$ and introduced 
\begin{align}
&h_{0,N,\bbR_+}^{(0)}=-\f{d^2}{dr^2}, \quad r>0,   \no \\
&\dom\big(h_{0,N,\bbR_+}^{(0)}\big)= \big\{f\in L^2((0,\infty);dr)\,\big|\, 
f,f'\in AC([0,R]) \text{ for all } R>0;    \lb{10.78} \\ 
& \hspace*{5.95cm} f'(0_+)=0; \,  f''\in L^2((0,\infty);dr)\big\}, \no \\
&h_{\ell,\bbR_+}^{(0)}=-\f{d^2}{dr^2}+\f{\ell(\ell +1)}{r^2}, \quad r>0,   \no \\
&\dom\big(h_{\ell,\bbR_+}^{(0)}\big)= \big\{f\in L^2((0,\infty);dr)\,\big|\, 
f,f'\in AC([0,R]) \text{ for all } R>0;    \lb{10.79} \\
&\hspace*{4.65cm} -f''+\ell(\ell +1)r^{-2}f\in 
L^2((0,\infty);dr)\big\}, \quad \, \ell\in\bbN.  \no 
\end{align}
The operators $h_{\ell,\bbR_+}^{(0)}|_{C_0^{\infty}((0,\infty))}$, $\ell\in\bbN$, are 
essentially self-adjoint in $L^2((0,\infty);dr)$ (but we note that 
$f\in \dom\big(h_{\ell,\bbR_+}^{(0)}\big)$ implies that $f(0_+)=0$).  
In addition, $U$ in \eqref{10.77} denotes the unitary operator,
\begin{equation}
U:\begin{cases} L^2((0,\infty); r^2dr)\to L^2((0,\infty);dr), \\
\hspace*{1.81cm} f(r)\mapsto (Uf)(r)= r f(r). \end{cases}   \lb{10.80}
\end{equation}
As discussed in detail in \cite[Sects.\ 4, 5]{GKMT01}, equations 
\eqref{10.75}--\eqref{10.77} follow from Corollary\ 4.8 in \cite{GKMT01} and the facts that 
\begin{equation} \lb{10.81}
(u_+,M_{H_{F,\bbR^n\backslash\{0\}},\cN_+}(z)u_+)_{L^2(\bbR^n;d^nx)} = 
\begin{cases} 
-(2/\pi) \ln(z) +2i, & n=2, \\
i(2z)^{1/2} +1, & n=3,
\end{cases}
\end{equation}
and
\begin{equation}
(u_+,M_{H_{K,\bbR^3\backslash\{0\}},\cN_+}(z)u_+)_{L^2(\bbR^3;d^3x)} = i(2/z)^{1/2} - 1. \lb{10.82}
\end{equation}
Here   
\begin{align}
\begin{split} 
& \cN_+= {\rm lin. \, span}\{u_+\}, \\
&  u_+ (x) = G_0(i,x,0)/\|G_0(i,\cdot,0)\|_{L^2(\bbR^n;d^nx)}, 
\; x\in\bbR^n\backslash\{0\}, \; n=2,3,  \lb{10.83} 
\end{split} 
\end{align}
and 
\begin{equation} \lb{10.84}
G_0(z,x,y) = 
\begin{cases}
\f{i}{4}H_0^{(1)}(z^{1/2}|x-y|), &x\neq y, \, n=2, \\
e^{iz^{1/2}|x-y|}/(4\pi |x-y|), &x\neq y, \, n=3
\end{cases}
\end{equation}
denotes the Green's function of $-\Delta$ defined on 
$H^{2}(\bbR^n),$ $n=2,3$ (i.e., the 
integral kernel of the resolvent $(-\Delta -z)^{-1}$),  
and $H_0^{(1)}(\cdot)$ abbreviates the Hankel function 
of the first kind and order zero (cf., \cite[Sect.\ 9.1]{AS72}). Here the Donoghue-type 
Weyl--Titchmarsh operators (cf.\ \cite{Do65} in the case 
where $\dim (\cN_+) =1$ and \cite{GKMT01}, \cite{GMT98}, and \cite{GT00} in the general abstract 
case where $\dim (\cN_+) \in \bbN \cup \{\infty\}$) $M_{H_{F,\bbR^n\backslash\{0\}},\cN_+}$ and 
$M_{H_{K,\bbR^n\backslash\{0\}},\cN_+}$ are defined according to equation (4.8) 
in \cite{GKMT01}: More precisely, given a self-adjoint extension $\wti S$ of the densely defined closed symmetric operator $S$ in a complex separable Hilbert space $\cH$, and a closed linear 
subspace $\cN$ of $\cN_+ = \ker({S}^* - i I_{\cH})$, 
$\cN\subseteq \cN_+$, the Donoghue-type Weyl--Titchmarsh operator $M_{\wti S,\cN}(z)
\in\cB(\cN)$ associated with the pair $(\wti S,\cN)$  is defined by
\begin{align}
\begin{split}
M_{\wti S,\cN}(z)&=P_\cN (z\wti S+I_\cH)(\wti S-z I_{\cH})^{-1} P_\cN\big\vert_\cN   \\
&=zI_\cN+(1+z^2)P_\cN(\wti S-z I_{\cH})^{-1} P_\cN\big\vert_\cN\,, \quad  
z\in \bbC\backslash \bbR, \lb{10.84a} 
\end{split}
\end{align}
with $I_\cN$ the identity operator in $\cN$ and $P_\cN$ the orthogonal projection in 
$\cH$ onto $\cN$.

Equation \eqref{10.81} then immediately follows from 
repeated use of the 
identity (the first resolvent equation),
\begin{align}
&\int_{\bbR^n} d^nx' G_0(z_1,x,x')G_0(z_2,x',0) = 
(z_1-z_2)^{-1}[G_0(z_1,x,0)-G_0(z_2,x,0)], \no \\
&\hspace*{7.7cm} x\neq 0, \, z_1\neq z_2, \, n=2,3, 
\lb{10.85}
\end{align}
and its limiting case as $x\to 0$. 

Finally, \eqref{10.82} 
follows from the following arguments: First one notices that 
\begin{equation} 
\big[-(d^2/dr^2)+ \alpha \, r^{-2}\big]\big|_{C_0^\infty((0,\infty))}     \lb{10.86}
\end{equation} 
is essentially self-adjoint in $L^2(\bbR_+; dr)$ if and only if $\alpha \geq 3/4$. 
Hence it suffices to consider the restriction of 
$H_{min, \bbR^3\backslash\{0\}}$ to 
the centrally symmetric subspace $\cH_{3,0,(0,\infty)}$ of $L^2(\bbR^3;d^3x)$ 
corresponding to angular momentum $\ell=0$ in \eqref{10.77a}, \eqref{10.77b}. 
But then it is a well-known fact (cf.\ \cite[Sects.\ 4,5]{GKMT01}) that the Donoghue-type 
Dirichlet $m$-function 
$(u_+,M_{H_{D,\bbR^3\backslash\{0\}},\cN_+}(z)u_+)_{L^2(\bbR^3;d^3 x)}$, satisfies
\begin{align} 
(u_+,M_{H_{D,\bbR^3\backslash\{0\}},\cN_+}(z)u_+)_{L^2(\bbR^3;d^3 x)} &= 
(u_{0,+},M_{h_{0,D,\bbR_+}^{(0)},\cN_{0,+}}(z)u_{0,+})_{L^2(\bbR_+; dr)},   \no \\
& =  i (2z)^{1/2} + 1,   \lb{10.89}
\end{align}
where
\begin{equation}
\cN_{0,+}= {\rm lin. \, span} \{u_{0,+}\}, \quad u_{0,+} (r)=
e^{iz^{1/2}r}/[2 \Im(z^{1/2}]^{1/2},  \; r>0,
\end{equation}
and $M_{h_{0,D,\bbR_+}^{(0)},\cN_{0,+}}(z)$ denotes the Donoghue-type Dirichlet  
$m$-function corresponding to the operator 
\begin{align}
&h_{0,D,\bbR_+}^{(0)}=-\f{d^2}{dr^2}, \quad r>0,  \no \\
&\dom\big(h_{0,D,\bbR_+}^{(0)}\big)= \big\{f\in L^2((0,\infty);dr)\,\big|\, 
f,f'\in AC([0,R]) \text{ for all } R>0;  \lb{10.91}  \\ 
& \hspace*{6.05cm} \, f(0_+)=0; \, f''\in L^2((0,\infty);dr)\big\}, \no
\end{align}
Next, turning to the Donoghue-type Neumann $m$-function given by 
$(u_+,M_{H_{N,\bbR^3\backslash\{0\}},\cN_+}(z)u_+)_{L^2(\bbR^3;d^3 x)}$ one obtains 
analogously to \eqref{10.89} that 
\begin{equation}
(u_+,M_{H_{N,\bbR^3\backslash\{0\}},\cN_+}(z)u_+)_{L^2(\bbR^3;d^3 x)} = 
(u_{0,+},M_{h_{0,N,\bbR_+}^{(0)},\cN_{0,+}}(z)u_{0,+})_{L^2(\bbR_+; dr)}, 
\end{equation}
where $M_{h_{0,N,\bbR_+}^{(0)},\cN_{0,+}}(z)$ denotes the Donoghue-type Neumann  
$m$-function corresponding to the operator $h_{0,N,\bbR_+}^{(0)}$ in \eqref{10.78}. The well-known linear fractional transformation relating the operators 
$M_{h_{0,D,\bbR_+}^{(0)},\cN_{0,+}}(z)$ and $M_{h_{0,N,\bbR_+}^{(0)},\cN_{0,+}}(z)$ 
(cf.\ \cite[Lemmas 5.3, 5.4, Theorem 5.5, and Corollary 5.6]{GKMT01}) then yields
\begin{equation}
(u_{0,+},M_{h_{0,N,\bbR_+}^{(0)},\cN_{0,+}}(z)u_{0,+})_{L^2(\bbR_+; dr)} = 
i (2/z)^{1/2} - 1,
\end{equation}
verifying \eqref{10.82}. 

The fact that the operator $T=-\Delta$, $\dom(T)=H^{2}(\bbR^2)$ 
is the unique nonnegative self-adjoint extension of 
$-\Delta_{min, \bbR^2\backslash\{0\}}$ in $L^2(\bbR^2;d^2x)$, has been shown 
in \cite{AGHKH87} (see also \cite[Ch.\ I.5]{AGHKH88}).

%%%%%%%%%%%%%%%%%%%%%%%%%%%%%%%%%%%%%%
\subsection{The Case $\Om=\bbR^n\backslash\{0\}$,  
$V=-[(n-2)^2/4]|x|^{-2}$, $n\geq 2$.}   \lb{s10.4}
%%%%%%%%%%%%%%%%%%%%%%%%%%%%%%%%%%%%%%

In our final subsection we briefly consider the following minimal operator 
$H_{min, \bbR^n\backslash\{0\}}$ in $L^2(\bbR^n;d^nx)$, $n\geq 2$,
\begin{equation}
H_{min, \bbR^n\backslash\{0\}}
=\ol{(-\Delta - ((n-2)^2/4)|x|^{-2})\big|_{C^\infty_0 (\bbR^n\backslash\{0\})}}\geq 0, 
\quad n\geq 2. 
\lb{10.94}
\end{equation}
Then, using again the angular momentum decomposition (cf.\ also \eqref{10.20}, 
\eqref{10.21}), 
\begin{align}
& L^2(\bbR^n; d^nx) = L^2((0,\infty); r^{n-1}dr) \otimes L^2(S^{n-1}; d\omega_{n-1}) 
= \bigoplus_{\ell\in\bbN_0} \cH_{n,\ell,(0,\infty)},   \lb{10.95} \\
& \cH_{n,\ell,(0,\infty)} = L^2((0,\infty); r^{n-1}dr) \otimes \cK_{n,\ell}, 
\quad \ell \in \bbN_0, \; n\geq 2,    \lb{10.96}
\end{align}
one finally obtains that 
\begin{equation}
H_{F,\bbR^n\backslash\{0\}}=H_{K,\bbR^n\backslash\{0\}} 
= U^{-1}h_{0,\bbR_+} U \oplus\bigoplus_{\ell\in\bbN} 
U^{-1}h_{n,\ell,\bbR_+} U,  \; n\geq 2,  \lb{10.97}
\end{equation}
is the unique nonnegative self-adjoint  
extension of $H_{min, \bbR^n\backslash\{0\}}$ in $L^2(\bbR^n;d^n x)$, where  
\begin{align}
& h_{0,\bbR_+} = -\f{d^2}{dr^2} - \f{1}{4 r^2}, \quad r>0,  \no \\
& \dom(h_{0,\bbR_+}) = \big\{f \in L^2((0,\infty); dr) \,\big|\, f, f' \in 
AC([\varepsilon,R]) \, \text{for all $0<\varepsilon<R$};     \lb{10.98} \\
& \hspace*{4.25cm}   f_0=0; \, (-f'' -(1/4)r^{-2}f) \in L^2((0,\infty); dr)\big\},  \no \\
& h_{n,\ell,\bbR_+} = -\f{d^2}{dr^2} + \f{4 \kappa_{n,\ell} - 1}{4 r^2},   
\quad r>0,  \no \\
& \dom(h_{n,\ell,\bbR_+}) = \big\{f \in L^2((0,\infty); dr) \,\big|\, f, f' \in 
AC([\varepsilon,R]) \, \text{for all $0<\varepsilon<R$};   \no  \\ 
& \hspace*{1.2cm} (-f'' +[\kappa_{n,\ell} - (1/4)]r^{-2}f) \in L^2((0,\infty); dr)\big\}, 
\quad \ell\in\bbN, \; n\geq 2.   \lb{10.99}
\end{align}
Here $f_0$ in \eqref{10.98} is defined by (cf.\ also \eqref{10.33})
\begin{equation}
f_0 = \lim_{r\downarrow 0} [-r^{1/2}\ln(r)]^{-1} f(r). 
\end{equation}
As in the previous subsection, $h_{n,\ell,\bbR_+}|_{C_0^{\infty}((0,\infty))}$, $\ell\in\bbN$,  $n\geq 2$, are essentially self-adjoint in $L^2((0,\infty);dr)$. In addition, $h_{0,\bbR_+}$ is the unique nonnegative self-adjoint extension of $h_{0,\bbR_+}|_{C_0^{\infty}((0,\infty))}$ in $L^2((0,\infty);dr)$. We omit further details. 

\medskip

%%%%%%%%%%%%%%%%%%%%%%%%%%%%%%%%%%%%%
\noindent {\bf Acknowledgments.}
We are indebted to Yury Arlinskii, Gerd Grubb, John Lewis, Konstantin Makarov, 
Mark Malamud, Vladimir Maz'ya, Michael Pang, Larry Payne, Barry Simon, 
Nikolai Tarkhanov, Hans Triebel, and Eduard Tsekanovskii for many helpful discussions and very valuable correspondence on various topics of this paper.

One of us (F.G.)\ gratefully acknowledges the extraordinary hospitality 
of the Faculty of Mathematics of the University of Vienna, Austria, 
during his three month visit in the first half of 2008.
%%%%%%%%%%%%%%%%%%%%%%%%%%%%%%%%%%%%%

%%%%%%%%%%%%%%%%%%%%%%%%%%%%%%%%%%%%%%
%%%%%%%%%%%%%%%%%%%%%%%%%%%%%%%%%%%%%%

\end{document}